\documentclass[11pt]{article}
\usepackage[all,2cell]{xy} \UseAllTwocells \SilentMatrices \SelectTips{cm}{}
\usepackage{latexsym,amsfonts,amssymb}
\usepackage{amsmath,amsthm,amscd}
\usepackage[dvips]{epsfig}
\usepackage{psfrag, pstricks}
\def\drawing#1{\begin{center}\epsfig{file=fig/#1}\end{center}}
\usepackage{diagcat}
\newstrandstyle{thick}{type=string,color=green,width=3pt} 

\usepackage{color}
\usepackage{etoolbox}
\usepackage{hyperref}
\usepackage{graphicx}
\usepackage{a4wide}
\usepackage{pinlabel}
\usepackage{tabularx}
\normalfont\upshape

    \def\BM{{\mathbb{B}}}

    \def\NM{{\mathbb{N}}}
    \def\OM{{\mathbb{O}}}

    \def\XM{{\mathbb{X}}}
    \def\YM{{\mathbb{Y}}}

  \def\ab{{\mathbf a}}  \def\AC{{\mathcal{A}}}
    \def\BC{{\mathcal{B}}}
    
    \def\DC{{\mathcal{D}}}
    \def\EC{{\mathcal{E}}}
    \def\FC{{\mathcal{F}}}
    
    \def\HC{{\mathcal{H}}}

    \def\MC{{\mathcal{M}}}
    \def\NC{{\mathcal{N}}}
    
    \def\PC{{\mathcal{P}}}
    
    \def\RC{{\mathcal{R}}}

    \def\UC{{\mathcal{U}}}
  \def\vb{{\mathbf v}}

\def\a{\alpha}
\def\b{\beta}
\def\g{\gamma}

\def\d{\delta}

\def\e{\varepsilon}

\def\l{\lambda}

\def\w{\omega}

\let\phi=\varphi

\newcommand{\ot}{\otimes}
\newcommand{\co}{\colon}

\newcommand{\odif}{\overline{\dif}}

\newcommand{\Proj}{\textrm{-proj}}

\newcommand{\Obj}{\textrm{Obj}}
\newcommand{\Kar}{\textrm{Kar}}
\newcommand{\Uthick}{\dot{\UC}}

\newcommand{\ig}[2]{\vcenter{\xy (0,0)*{\includegraphics[scale=#1]{fig/#2}} \endxy}}

\newcommand{\scs}{\scriptstyle}

\theoremstyle{plain}
\newtheorem{thm}{Theorem}[section]
\newtheorem{prop}[thm]{Proposition}
\newtheorem{lemma}[thm]{Lemma}
\newtheorem{cor}[thm]{Corollary}

\theoremstyle{definition}

\newtheorem{defn}[thm]{Definition}
\newtheorem{example}[thm]{Example}
\newtheorem{rmk}[thm]{Remark}

\numberwithin{equation}{section}

\usepackage{bbm}
\def\1{\mathbbm 1}

\def\Z{\mathbb Z}
\def\N{\mathbb N}
\def\C{\mathbb C}
\def\F{\mathbb F}
\def\o{\otimes}
\def\lra{\longrightarrow}

\def\Id{\mathrm{Id}}
\def\mc{\mathcal}
\def\mf{\mathfrak}

\def\END{\mathrm{END}}
 
\newcommand{\dmod}{\!-\!\mathrm{mod}}

\newcommand{\mH}{\mathrm{H}}  
\newcommand{\NH}{\mathrm{NH}} 
\newcommand{\pol}{\mathrm{Pol}}

\newcommand{\s}{\mathcal{S}}
\newcommand{\sym}{\mathrm{Sym}}
\newcommand{\Mat}{\mathrm{Mat}}

\newcommand{\Hom}{{\rm Hom}}
\newcommand{\HOM}{{\rm HOM}}

\newcommand{\op}{{\rm op}}
\renewcommand{\Im}{{\rm Im}\;}

\renewcommand{\sl}{\mathfrak{sl}}

\def\dif{{\partial}}

\def\lra{{\longrightarrow}}
\def\dmod{{\mathrm{-mod}}}   

\def\END{{\mathrm{END}}}

\def\Id{\mathrm{Id}}
\def\mc{\mathcal}
\def\mf{\mathfrak}
\def\shuffle{\,\raise 1pt\hbox{$\scriptscriptstyle\cup{\mskip
               -4mu}\cup$}\,}

\newcommand{\refequal}[1]{\xy {\ar@{=}^{#1}
(-1,0)*{};(1,0)*{}};
\endxy}

\def\lr{\mbox{\begin{picture}(7,10)
\put(1,0){\line(1,0){10}}
\put(1,0){\line(1,1){10}}
\put(11,0){\line(0,1){10}}
\end{picture}\;
}}

\def\ll{\mbox{\begin{picture}(7,10)
\put(1,0){\line(1,0){10}}
\put(11,0){\line(-1,1){10}}
\put(1,0){\line(0,1){10}}
\end{picture}\;
}}

\title{A categorification of quantum \texorpdfstring{$\mathfrak{sl}(2)$}{sl(2)} at prime roots of unity}
\author{Ben Elias, You Qi}

\date{\today}
\begin{document}
%

\maketitle

\begin{abstract}
We categorify the Beilinson-Lusztig-MacPherson integral form of quantum $\mf{sl}_2$ specialized at a prime root of unity.
\end{abstract}

\setcounter{tocdepth}{2}
\tableofcontents


\section{Introduction}

\subsection{Motivation}
The goal of this sequence of papers is to categorify quantum groups, as well as their representation theory, at a prime root of unity. There are several notions of
what the ``quantum group at a root of unity" can be. The ``small'' version of the quantum group, introduced by Lusztig in \cite{LusfdHopf}, is a finite dimensional Hopf algebra over
the field of cyclotomic integers. Alternatively, the Beilinson-Lusztig-MacPherson (BLM) integral form of the quantum group, constructed in \cite{BLM}, can be specialized to the same root of unity,
yielding an infinite dimensional Hopf algebra in which the small version sits as a Hopf subalgebra. Small quantum groups are intimately related to quantum topological invariants for
three-manifolds, while the BLM form connects the representation theory of quantum groups at roots of unity to that of affine Lie algebras at certain fixed levels. In \cite{EQ1} we categorified the small version of quantum $\sl_2$ at a prime root of unity, while in this paper we categorify the BLM form specialized at the same cyclotomic ring.

The original motivation for categorification, given by Crane and Frenkel \cite{CF}, was to find a combinatorial construction of certain gauge-theoretic quantum four-manifold invariants. They hoped to achieve this by lifting the three-dimensional Reshetikhin-Turaev-Witten topological field theory to a four-dimensional theory. The simple and elegant construction of Khovanov
homology \cite{KhJones} for links in the three-dimensional sphere is strong evidence for Crane and Frenkel's conjectural proposal. In order to lift link invariants into three-manifold invariants, one
should specialize the generic quantum parameters for link invariants to a root of unity, at least from the point of view of quantum physics.

The use of $p$-DG structures and $p$-homological algebra (or ``hopfological algebra") to categorify algebras specialized at a prime root of unity was first suggested in \cite{Hopforoots}. The previous papers in this series, \cite{KQ} and \cite{EQ1}, applied the theory of hopfological algebra developed in \cite{QYHopf} to categorify quantum $\sl_2$ at prime roots of unity in
various aspects. We expect the reader to have slight familiarity with \cite{KQ} and \cite{EQ1}, although we will recall the facts used from them when needed. We also expect the reader to be familiar with the definitions of a $p$-DG category and its $p$-DG Grothendieck group, but see Chapter~\ref{sec-idempotents} for a refresher.

\subsection{A brief summary}
We now give a rough sketch of the contents of each chapter.

The second chapter of this paper studies the $p$-DG algebra of symmetric polynomials in $n$ variables, and related structures. Recall that in the prequel \cite{KQ}, Khovanov and the second
author studied the polynomial ring, viewed as a $p$-DG module for symmetric polynomials, and its endomorphism ring, the nilHecke algebra. The polynomial ring can be equipped with a family
of possible $p$-differentials, which in turn equip the nilHecke algebra with a family of differentials; they showed that only two members of this family (which are, in some sense, dual to each other,
and therefore essentially unique) yield a $p$-DG nilHecke algebra with satisfactory hopfological properties. We investigate how the differential acts on the subalgebra of symmetric
polynomials, and compute its $p$-cohomology groups. The $p$-DG algebra of symmetric polynomials is shown to be quasi-isomorphic to a polynomial subalgebra with the trivial differential
(Proposition \ref{prop-cohomology-of-sym-n}). Therefore, it can be thought of as a ``formal" $p$-DG algebra.

In \S\ref{sec-grassmannian}, we generalize the method of \cite[Chapter 3]{KQ} to investigate partially symmetric polynomials, i.e., polynomials invariant under a Young subgroup of the
symmetric group. We study them as $p$-DG algebras, and as $p$-DG modules for the algebra of (fully) symmetric polynomials. It is known that the endomorphism algebras of these modules are
given by the ``positive half" of the thick calculus developed by Khovanov-Lauda-Mackaay-Sto\v{s}i\'{c} in \cite{KLMS}. This provides a natural explanation why the ``half thick calculus'' is
equipped with a $p$-differential, which is compatible with the differentials used in \cite{KQ, EQ1}. Derived categories of these $p$-DG endomorphism algebras are then used to construct a
categorification of the positive half of the divided-power form of quantum $\sl_2$. This is analogous to the categorification of the positive half of the small form in \cite{KQ}. We
formulate this result as Theorem \ref{thm-half-U-thick}. Furthermore, in Corollary \ref{cor-categorical-embedding-half-u}, we establish a categorical embedding of the categorified small $\sl_2$ at a prime root of unity into the divided-power form.

Recall that the $p$-DG nilHecke algebra is acyclic when the number of strands $n$ is at least $p$. Meanwhile, the results of \S\ref{sec-sym} show that the $p$-DG symmetric polynomial
algebras, and the algebras appearing in the half thick calculus, are never acyclic. This (to the informed reader) should explain why the former categorifies the small version, and the
latter the BLM version, of the positive half of the quantum group at a root of unity.

In \S\ref{sec-idempotents}, we provide a detailed discussion of the interactions between hopfological algebra and (partial) idempotent completion. This chapter is somewhat technical, but
also adds an interesting new wrinkle, and so we summarize its basic points in the next part of the introduction. In preparation for this discussion, \S\ref{sec-idempotents} also recalls
many of the necessary facts from hopfological algebra that we use in the remainder of the paper.

We then seek to ``double" the above half-construction of the categorified quantum group. In \S\ref{sec-thick} we equip the thick calculus $\Uthick$ of \cite{KLMS} with a $p$-differential
$\dif$, and provide some formulas for its computation. This differential agrees with the differential on the positive half from \S\ref{sec-grassmannian}, and with the differential induced
on a partial idempotent completion from \S\ref{sec-idempotents}. We prove in Theorem \ref{thm-U-thick} that (the derived category of) the $p$-DG monoidal category $(\Uthick,\dif)$
categorifies the BLM form of $\mathfrak{sl}_2$ at a $p$th root of unity.
There is, as yet, no satisfactory notion of a ``categorified Drinfeld double" in either the abelian categorification of the big quantum group pioneered by Khovanov-Lauda \cite{KL1,KL2,KL3,Lau1} and Rouquier \cite{Rou2}, or the $p$-DG context studied in this paper and \cite{EQ1}. However, we expect, when such a concept is developed, that the $p$-DG categories $(\UC,\dif)$ and $(\Uthick,\dif)$ from our work will be interesting examples.

There is one categorified relation (i.e., a direct sum decomposition) needed for $\Uthick$ which is not already present in the positive half, guaranteed by the Sto\v{s}i\'{c} formula
(\cite[Theorem 5.9]{KLMS}). In order for this categorified relation to descend to a relation on the $p$-DG Grothendieck group, we need to check a technical $p$-DG filtered module condition,
called the ``fantastic filtration condition" (Definition \ref{defnFcfiltration}). This computation occupies most of \S\ref{sec-categorification}. This chapter is highly computational, and
follows \cite{KLMS} extremely closely; consequently, we have gone light on the related introductory and background material, assuming at this point that the reader is very familiar with
\cite{KLMS}.

A diagrammatically-minded, computation-loving reader is welcome to skip to \S\ref{sec-thick}. The first part of the paper (\S\ref{sec-sym} and \S\ref{sec-grassmannian}) is logically
independent, and is mostly unnecessary for the main result of \S \ref{sec-categorification}. However, it helps to explain from a theoretical perspective why the differential on $\Uthick$ is what it is, unrelated to its
combinatorial construction, or the computational requirements of \S\ref{sec-categorification}. Moreover, it will certainly play a role in the categorical representation theory of quantum
$\sl_2$ at a prime root of unity, whose study is initiated in \cite{QiSussan}, and which we plan to pursue in future work.

\subsection{Idempotent completion and hopfological algebra}

The thick calculus $\Uthick$ can be thought of as a realization of the Karoubi envelope or idempotent completion of Lauda's category $\UC$ \cite{Lau1}. Therefore, our main result is an
example of a surprising phenomenon: a $p$-DG category and its Karoubi envelope need not be $p$-DG Morita equivalent! The $p$-DG Grothendieck group of $\Uthick$ is much larger than that of
$\UC$, despite the fact that the underlying additive categories are Morita equivalent.

In fact, the abelian categories of $p$-DG modules over a $p$-DG category and its Karoubi envelope are equivalent, if the Karoubi envelope admits a compatible $p$-DG structure. However, this
equivalence does not preserve the hopfological properties of $p$-DG modules, and for this reason the compact derived categories may not be equivalent. For example, the divided power object
$\EC^{(a)}$ for $a \ge p$ is acyclic when viewed as a module for $\UC$, and thus its symbol is zero in the $p$-DG Grothendieck group of $\UC$; however, it is not acyclic when viewed as a
module for $\Uthick$, and its symbol contributes non-trivially to the $p$-DG Grothendieck group of $\Uthick$.

This may seem a bit counter-intuitive at first glance, as it did to the authors at the beginning of this project. In abelian categorification it is a typical practice to take the Karoubi
envelope as a matter of course (as was done in the works of Khovanov, Lauda, and Rouquier \cite{KL1,KL2,KL3,Lau1,Rou2}), while the $p$-DG world requires a great deal more caution. We feel that this
poses the greatest stumbling block for people entering the field. Accordingly, \S\ref{sec-idempotents} is devoted to a careful study of partial idempotent completions and their
Grothendieck groups.

It is a complication that, given an additive $p$-DG category, its Karoubi envelope (i.e., the category obtained by adding the images of all idempotents as new objects) need not come
equipped with a $p$-DG structure. After all, the idempotents do not necessarily ``commute" with the differential. However, it may be the case that a Karoubian partial idempotent
completion (i.e., a category obtained by adding the images of enough idempotents) can be equipped with a $p$-DG structure. This, amazingly, is what happens in our example.

One should note that nothing in this discussion is unique to the $p$-DG situation; these observations apply equally to ordinary DG categories. Given a DG category, there need not be a DG
structure on the Karoubi envelope of the underlying additive category. This is not a problem that appears in the derived category of an abelian category, where the homological direction
is ``orthogonal" to the algebraic direction and the differential commutes with any idempotent, but it can happen in other DG categories. When a partial idempotent completion can be equipped
with a DG structure, its DG category may have a larger Grothendieck group than the original DG category. Unfortunately, we could not find a similar treatment of this situation for DG
categories in the literature.

\subsection{Acknowledgements}

We thank Mikhail Khovanov for his constant support and encouragement. We also thank Rapha\"{e}l Rouquier for helpful discussions and suggestions on further extensions of the current work.
Our appreciation goes to Aaron Lauda, and by extension to the other authors Mikhail Khovanov, Marco Mackaay, and Marko Sto\v{s}i\'{c} of \cite{KLMS}, for allowing us electronic access to
the diagrams used in their paper.

\section{The nilHecke algebra and symmetric polynomials}\label{sec-sym}

Fix once and for all a base field $\Bbbk$ of characteristic $p > 0$. The unadorned tensor product $\o$ denotes the tensor product over $\Bbbk$. We adopt the French convention that the set of natural numbers $\N$ contains zero.

%
\subsection{Review of the \texorpdfstring{$p$}{p}-DG nilHecke algebra}
%

Here we briefly review the $p$-DG categorification of \cite{KQ}. Let $\pol_n=\Bbbk[x_1,\dots, x_n]$ be the polynomial algebra with the differential $\dif(x_i):=x_i^2$ for all $i=1, \dots ,n$. The subring $\sym_n$
of symmetric polynomials is preserved by this differential, so it is also a $p$-DG algebra.

We define a family of $p$-DG modules for $\pol_n$, whose underlying modules are free of rank one. For each $n$-tuple of numbers $\mathtt{a}=(a_1, \dots, a_n)\in \F_p^n$, let $\PC_n(\mathtt{a}):=\pol_n\cdot v_{\mathtt{a}}$ be a free rank one module over $\pol_n$, equipped with the differential $\dif_\mathtt{a}$ determined by
\[
\dif_\mathtt{a}(v_\mathtt{a})=\sum_{i=1}^na_ix_iv_\mathtt{a}.
\]
The generator $v_\mathtt{a}$ is in degree $n(1-n)/2$. By restriction, one can regard $\PC_n(\mathtt{a})$ as a $p$-DG module over $\sym_n$.

\begin{rmk}
Because of the Leibniz rule, it is enough to specify the differential $\dif_\mathtt{a}$ on the generator $v_\mathtt{a}$. The Leibniz rule implies that
\[
\dif_\mathtt{a} (f v_\mathtt{a}) = \dif(f) v_\mathtt{a} + f\dif_\mathtt{a}(v_\mathtt{a}).
\]
In this chapter, the unadorned $\dif$ will denote a differential on an algebra, like $\pol_n$, $\sym_n$, or $\NH_n$ below, whereas a differential symbol with a subscript will denote the differential on a $p$-DG module.
\end{rmk}

Let $\mathtt{a}^+ = (-1,0) \in \F_p^2$, and denote $\PC_2(\mathtt{a}^+)$ by $\PC_2^+$. Similarly, let $\mathtt{a}^- = (0,-1)$. It is shown in \cite[Section 3.2]{KQ} that, up to reordering
the variables, the $p$-DG module $\PC_2(\mathtt{a})$ is a finite cell module over $\sym_2$ if and only if $\mathtt{a} = \mathtt{a}^\pm$. The modules $\PC_2(\mathtt{a}^+)$ and
$\PC_2(\mathtt{a}^-)$ are dual, in a sense to be described in the next chapter; we will restrict our attention to $\mathtt{a}^+$.

The finite-cell filtration on $\PC_2^+$ is given by
\[
0\subset \sym_2\cdot x_1v_{\mathtt{a}^+} \subset \sym_2\cdot x_1v_{\mathtt{a}^+}+\sym_2\cdot v_{\mathtt{a}^+} = \PC_2^+,
\]
arising from the basis $\{x_1 v_{\mathtt{a}^+}, v_{\mathtt{a}^+}\}$.
The induced differential on the algebra $\END_{\sym_2}(\PC_2({\mathtt{a}^+}))\cong \NH_2$ will be denoted $\dif$. It is given on the local diagrammatic generators by
\begin{equation}\label{eqn-dif-on-nilHecke-generator}
\dif\left(
\begin{DGCpicture}
\DGCstrand(1,0)(1,1) \DGCdot{0.5}
\end{DGCpicture} \right)
~ =~
\begin{DGCpicture}
\DGCstrand(1,0)(1,1) \DGCdot{0.5}[ur]{$\scs 2$}
\end{DGCpicture}
\ , \ \ \ \
\dif\left(~
\begin{DGCpicture}
\DGCstrand(1,0)(0,1)
\DGCstrand(0,0)(1,1)
\end{DGCpicture}~\right)
=
-
\begin{DGCpicture}
\DGCstrand(1,0)(0,1)\DGCdot{0.75}
\DGCstrand(0,0)(1,1)
\end{DGCpicture}
-
\begin{DGCpicture}
\DGCstrand(1,0)(0,1)\DGCdot{0.25}
\DGCstrand(0,0)(1,1)
\end{DGCpicture}.
\end{equation}
Here a dot on the $i$-th strand ($i=1,2$) indicates the endomorphism of multiplication on $\PC_2(\mathtt{a}^+)$ by $x_i$, while the crossing stands for the divided difference operator $D_2$
on $\PC_2^+$, which acts on any $f(x_1,x_2)v_{\mathtt{a}^+}\in \PC_2^+$ by
\[
D_2(f(x_1,x_2)v_{\mathtt{a}^+}):=\frac{f(x_1,x_2)-f(x_2, x_1)}{x_1-x_2}v_{\mathtt{a}^+}.
\]
Furthermore, it is shown that when $\mathrm{char}(\Bbbk)>2$, the $p$-DG module $\PC_2^+$ is a compact cofibrant generator of the derived category $\DC(\NH_2,\dif)$.

\begin{rmk} In \cite{KQ}, a family of differentials on $\NH_2$ were introduced. The differential $\dif$ here corresponds to $\dif_1$ from that paper, and is the only member of their family
we will consider in this work. Had we chosen $\mathtt{a}^-$ instead, the induced differential on $\END_{\sym_2}(\PC_2({\mathtt{a}^-})) \cong \NH_2$ would correspond to $\dif_{-1}$.
\end{rmk}

\begin{rmk}\label{rmk-dif-on-NH}
In \cite{EQ1}, the divided difference operator $D_2$ was denoted $\d(2)$, while in previous papers of Khovanov-Lauda and others it is denoted $\dif$. We have altered the notation to be consistent with \cite{KLMS}, who use $\d_a$ instead for the polynomial soon to be defined. In this paper, we will prefer the notation from \cite{KLMS} over \cite{EQ1} whenever the two are in conflict, because we will be following some difficult computations from \cite{KLMS} later. \end{rmk}

More generally, let $\mathtt{a}^+ = (-(n-1),\dots,-1,0) \in \F_p^n$, and let $\PC_n^+ := \PC_n(\mathtt{a}^+)$. The restriction of $\PC_n^+$ to the $p$-DG subalgebra $\sym_n$ is a finite cell module, with basis
\[
B_n^+:=\left\{x_1^{a_1}\cdots x_{n-1}^{a_{n-1}}v_{\mathtt{a}^+}| 0 \leq a_i \leq n-i\right\}.
\]
Let $\d_n = \prod_{i=1}^n x_i^{n-i}$ be the monomial coefficient in front of the maximal degree element of $B_n^+$. The endomorphism algebra $\NH_n$ of $\PC_n^+$ has an induced differential, also denoted $\dif$.

Let us introduce some notation for certain linear polynomials, viewed as elements in $\pol_n$ or $\NH_n$. We have

\begin{subequations}
\noindent
\begin{tabularx}{\textwidth}{@{}XXX@{}}
\begin{equation}
\ll_n = \sum_{i=1}^n (n-i)x_i \label{lltri}
\end{equation} &
\begin{equation} \lr_n = \sum_{i=1}^n (i-1)x_i. \label{lrtri}
\end{equation}
\end{tabularx}
\end{subequations}
For instance, we have
\[
\dif_{\mathtt{a}^+}(v_{\mathtt{a}^+}) = - \ll_n v_{\mathtt{a}^+}
\]
and
\begin{equation}
\dif(\d_n) =\dif(x_1^{n-1}x_2^{n-2}\dots x_{n-1})=\ll_n \d_n.
\end{equation}

We record the following useful formula for the $\dif$ action on the lowest degree element $D_n \in \NH_n$, attached to the longest element of $S_n$. It follows from \eqref{eqn-dif-on-nilHecke-generator} by induction on the number of strands (see \cite[Lemma 3.12]{KQ}).
\begin{equation}\label{eqn-diff-on-deltas}
\dif(D_n) = -\ll_n D_n - D_n \lr_n
\end{equation}
This is diagrammatically depicted as
\begin{eqnarray}\label{eq-dif-of-Dn}
\dif\left(
\begin{DGCpicture}
\DGCstrand(0,0)(0,2) \DGCstrand(0.5,0)(0.5,2) \DGCstrand(1,0)(1,2)
\DGCstrand(1.5,0)(1.5,2)
\DGCcoupon(-0.15,0.5)(1.65,1.5){$D_n$}
\end{DGCpicture}
\right)&  =  & - (n-1) {\begin{DGCpicture}
\DGCstrand(0,0)(0,2)\DGCdot{1.75} \DGCstrand(0.5,0)(0.5,2)
\DGCstrand(1,0)(1,2) \DGCstrand(1.5,0)(1.5,2)
\DGCcoupon(-0.15,0.5)(1.65,1.5){$D_n$}
\end{DGCpicture}}-(n-2)
{\begin{DGCpicture} \DGCstrand(0,0)(0,2)
\DGCstrand(0.5,0)(0.5,2)\DGCdot{1.75} \DGCstrand(1,0)(1,2)
\DGCstrand(1.5,0)(1.5,2)
\DGCcoupon(-0.15,0.5)(1.65,1.5){$D_n$}
\end{DGCpicture}}- \cdots -
{\begin{DGCpicture} \DGCstrand(0,0)(0,2) \DGCstrand(0.5,0)(0.5,2)
\DGCstrand(1,0)(1,2)\DGCdot{1.75} \DGCstrand(1.5,0)(1.5,2)
\DGCcoupon(-0.15,0.5)(1.65,1.5){$D_n$}
\end{DGCpicture}}\nonumber \\
&& \nonumber \\
&& - {\begin{DGCpicture} \DGCstrand(0,0)(0,2)
\DGCstrand(0.5,0)(0.5,2)\DGCdot{0.25} \DGCstrand(1,0)(1,2)
\DGCstrand(1.5,0)(1.5,2)
\DGCcoupon(-0.15,0.5)(1.65,1.5){$D_n$}
\end{DGCpicture}}
-\cdots-(n-2) {\begin{DGCpicture} \DGCstrand(0,0)(0,2)
\DGCstrand(0.5,0)(0.5,2) \DGCstrand(1,0)(1,2)\DGCdot{0.25}
\DGCstrand(1.5,0)(1.5,2)
\DGCcoupon(-0.15,0.5)(1.65,1.5){$D_n$}
\end{DGCpicture}}
-(n-1) {
\begin{DGCpicture} \DGCstrand(0,0)(0,2)
\DGCstrand(0.5,0)(0.5,2) \DGCstrand(1,0)(1,2)
\DGCstrand(1.5,0)(1.5,2)\DGCdot{0.25}
\DGCcoupon(-0.15,0.5)(1.65,1.5){$D_n$}
\end{DGCpicture}
} \ .
\end{eqnarray}

%
\subsection{Schur polynomials}
%

We identify partitions with Young diagrams. For instance, the following Young diagram corresponds to the partition $(5,3,3,1)$.
\vspace{0.1in}

\drawing{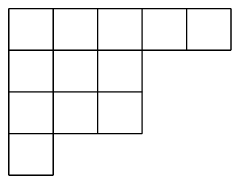}

\vspace{0.1in}

The partitions below are those can be obtained from the Young diagram above by adding one (shaded) box.

\vspace{0.1in}

\drawing{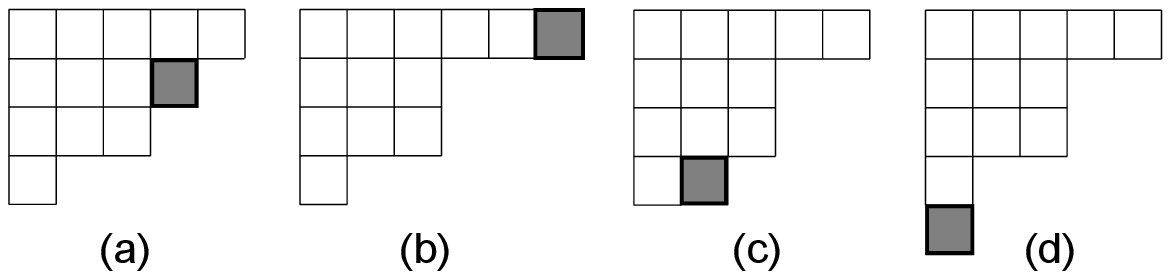}

\vspace{0.1in}

The diagram of case (d) does not occur if we confine the partitions to have less or equal to four rows. The following is a diagram with one box added that is not allowed.

\vspace{0.1in}

\drawing{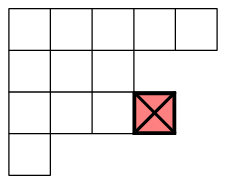}

\vspace{0.1in}

Let $P(n)$ denote the set of partitions with at most $n$ rows. For $\l = (\l_1, \dots, \l_n) \in P(n)$ (we allow $\lambda_i=\lambda_{i+1}=\cdots =\lambda_{n}=0$), let $\pi_\l \in \sym_n$ be the corresponding Schur polynomial in $n$ variables. It is well-known that $\{\pi_\l\}_{\l \in P(n)}$ forms an integral basis for $\sym_n$. The following lemma indicates how $\dif$ acts with respect to this basis.

\begin{lemma}\label{lemma-difofpilambdalonghand} The differential $\dif$ on $\sym_n$ acts on $\pi_\lambda$ as follows,
\begin{equation}\label{difofpilambdalonghand}
\dif(\pi_\lambda)=\sum_{\mu_i}(\lambda_i+1-i)\pi_{\mu_i},
\end{equation}
where the sum is over Young diagrams $\mu_i$ that are obtained from $\lambda$ by adding a box such that $\mu_i$ remains a partition in $P(n)$. The index $i$ indicates the row in which the box was added.
\end{lemma}
\begin{proof}Let $\lambda=(\lambda_1,\dots, \lambda_n)$ be an $n$-part partition. The Schur function $\pi_\lambda$ in $n$-variables is obtained by applying the longest divided difference operator $D_n$ to the monomial
$$x^\lambda \d_n v_{\mathtt{a}} = x_1^{\lambda_1+n-1}\cdots x_{n-1}^{\lambda_{n-1}+1}x_n^{\lambda_n}v_\mathtt{a}\in \PC^+_n.$$
When equipped with compatible $p$-DG structures, we have,
\begin{align*}
\dif_{\mathtt{a}^+}(\pi_{\lambda}v_{\mathtt{a}^+})  =\dif_{\mathtt{a}^+}(D_n(x^\lambda \d_n v_{\mathtt{a}^+}))= \dif(D_n)(x^\lambda \d_n v_{\mathtt{a}^+})+ D_n(\dif_{\mathtt{a}^+}(x^\lambda \d_n v_{\mathtt{a}^+})),
\end{align*}
which implies that
\begin{eqnarray*}
\dif(\pi_{\lambda})v_{\mathtt{a}^+} & = & \dif(D_n)(x^\lambda \d_n v_{\mathtt{a}^+})+ D_n(\dif_{\mathtt{a}^+}(x^\lambda \d_n v_{\mathtt{a}^+})) - \pi_{\lambda}\dif_{\mathtt{a}^+}(v_{\mathtt{a}^+})\\
& = & -\sum_{i=1}^n(n-i)x_iD_n(x^\lambda \d_n v_{\mathtt{a}^+})- \sum_{i=1}^n(i-1)D_n(x_ix^\lambda \d_n v_{\mathtt{a}^+})\\
&   &  + \sum_{i=1}^n (\lambda_i+n-i)D_n(x_1^{\lambda_1+n-1}\cdots x_{i}^{\lambda_i+n-i+1}\cdots x_{n}^{\lambda_n}v_{\mathtt{a}^+})\\
&   &  - \sum_{i=1}^n (n-i) D_n(x^\lambda \d_n x_iv_{\mathtt{a}^+})+\sum_{i=1}^n(n-i)\pi_{\lambda}x_iv_{\mathtt{a}^+}\\
& = & \sum_{i=1}^n (\lambda_i-i+1)D_n(x_1^{\lambda_1+n-1}\cdots x_{i}^{\lambda_i+n-i+1}\cdots x_{n}^{\lambda_n}v_{\mathtt{a}^+}).
\end{eqnarray*}
The $i$-th term in this sum is equal to the $i$-th term in the claimed formula, so long as adding a box to the $i$-th row of $\lambda$ produces a Young diagram $\mu_i$. If it does not produce a Young diagram (see the non-example before the lemma), then $\lambda_{i-1}=\lambda_i$ and
\[
x_1^{\lambda_1+n-1}\cdots x_{i-1}^{\lambda_{i-1}+n-i+1}x_{i}^{\lambda_{i}+n-i+1} \cdots x_n^{\lambda_n}v_{\mathtt{a}^+}
\]
is symmetric in $x_{i-1}$ and $x_i$, and hence the corresponding term is killed by $D_n$.
\end{proof}

Notice that, in the statement of Lemma \ref{lemma-difofpilambdalonghand}, the number $\lambda_i+1-i$ is the \emph{content} or \emph{residue} of the box in $\mu_i$ that was added to $\lambda$. The following example is a Young diagram with its boxes labeled by their content numbers.
\vspace{0.1in}
\drawing{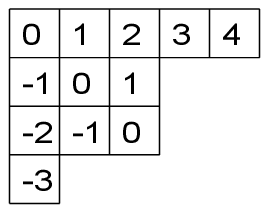}
\vspace{0.1in}
We will use a shorthand when writing sums like \eqref{lemma-difofpilambdalonghand}, as exemplified below.
\begin{equation}
\dif(\pi_\l) = \sum_{\l + \square} C(\square) \pi_{\l + \square}. \label{difpil}
\end{equation}
The sum is taken over all ways to add a box to $\l$ to obtain another partition $\l + \square$ in $P(n)$, and $C(\square)$ is the content of the new box. In such a sum, the maximum number of rows in a partition is not written, but is understood; the new box $\square$ can not be in the $(n+1)$-st row.

\begin{defn}\label{def-p-dg-sym}
In what follows, we will always use $\sym_n$, $n\in \N$, to stand for the \emph{$p$-DG algebra} of symmetric functions on $n$-letters, equipped with the $p$-differential defined by equation \eqref{difofpilambdalonghand} or \eqref{difpil}.
\end{defn}

Let $e_k$ (resp. $h_k$) denote the $k$-th elementary (resp. complete) symmetric polynomial in $n$ variables. These are specific examples of Schur polynomials, associated to a single column (resp. row). Thus \eqref{difpil} implies:
\begin{equation} \label{difek}
\dif(e_k) = e_1 e_k - (k+1) e_{k+1},
\end{equation}
\begin{equation} \label{difhk}
\dif(h_k) = (k+1) h_{k+1} - h_1 h_k.
\end{equation}
By definition, $e_k=0$ for $k>n$. These formulas hold for all $k \in \N$, even when $k \ge n$.

Letting $n$ grow to infinity, we see that the formula \eqref{difpil} also applies to Schur functions $\pi_\l$ inside the ring $\Lambda$ of symmetric functions; one omits any constraints
on the number of rows in $\a + \square$.

%
\subsection{\texorpdfstring{$p$}{p}-DG symmetric polynomials: part I}
%

Now we investigate the underlying $p$-complex of the $p$-DG algebra $\sym_n$, and its rank one free modules. We begin with the case $n<p$. Part of the discussion below has already appeared in \cite{KQ, EQ1}, and we repeat it here for the sake of completeness.

\begin{lemma} \label{lemma-cohomology-of-sym-less-than-p}
When $n < p$, $\sym_n$ is quasi-isomorphic to $\Bbbk$.
\end{lemma}
\begin{proof} There is clearly an injection sending $\Bbbk \to \Bbbk \cdot 1$, with cokernel the augmentation ideal $\sym_n^\prime$. When $n < p$, the representations of $S_n$ over
$\Bbbk$ are actually semisimple. Therefore $\sym_n$ is not just a submodule of $\Bbbk[x_1,x_2,\ldots,x_n]$ but actually a summand, under an idempotent which
commutes with $\dif$. In particular, the positive degree symmetric polynomials $\sym_n^\prime$ are a summand of the contractible $p$-complex
$\Bbbk[x_1,\ldots,x_n]^\prime$, and are therefore contractible.
\end{proof}

This lemma does rely on the fact that $n<p$. When $n \ge p$, the augmentation ideal $\sym_n^\prime$ is still a submodule of an acyclic $p$-complex, but that does not mean it is acyclic itself.

\begin{example}
Suppose that $n=p$. Then $\sym_p$ is not quasi-isomorphic to $\Bbbk$, but is actually quasi-isomorphic to $\Bbbk[e_p^p]$ with the trivial differential. We will prove this in Corollary \ref{cor-cohomology-rank-one-modules-p}. Note for now that $\Bbbk[e_1,e_2,\ldots,e_{p-1}]$ is a $p$-DG subalgebra of $\sym_p$, since $\dif(e_{p-1}) = e_1 e_{p-1}- pe_p = e_1 e_{p-1}$.
\end{example}

Now we want to understand various $p$-DG module structures on the rank-one free module over $\sym_n$.

\begin{defn}\label{def-rank-one-mod}
Let $a\in \Bbbk$. Define a differential structure $\dif_a$ on the rank-one free module $\s_n (a):=\sym_n\cdot v_a$ over $\sym_n$, determined by
\begin{equation}
 \dif_a(v_a)=a e_1v_a.
\end{equation}
The degree of the generator $v_a$ is $n(1-n)/2$, to agree with standard conventions. The $p$-nilpotency condition $\dif_a^p=0$ holds if and only if $a \in \F_p$ (by an easy computation). When this happens, $\s_n(a)$ is isomorphic, up to a grading shift, to the $\dif$-closed ideal $(e_n^k) \subset \sym_n$ inside the algebra $\sym_n$, for any $k \in \N$ having residue class modulo $p$ equal to $a$.
\end{defn}

\begin{lemma}\label{lemma-contractible-rank-one-ideals-less-than-p}
Suppose that $1 \le a \le n < p$. Then the ideal generated by $e_n^a$ inside $\sym_n$ is acyclic.
\end{lemma}

Equation \eqref{difek} and the Lemma immediately implies the following.

\begin{cor} \label{cor-contractible-rank-one-modules-less-than-p}
For $1 \le a \le n < p$, the $p$-DG module $\s_n(a)$ is acyclic. \hfill$\square$
\end{cor}
\begin{proof}[Proof of Lemma \ref{lemma-contractible-rank-one-ideals-less-than-p}] We show this by induction on $n$, and within each fixed $n$, by induction on $a$. For $n=a=1$, the result is already known.
Suppose that $a=1$. There is a short exact sequence of $p$-DG modules
\begin{align*}
0 \to (e_n) \to \sym_n \to \sym_{n-1} \to 0.
\end{align*}
The first map is the inclusion of the ideal, while the second is a quotient. The quotient is a quasi-isomorphism between two $p$-DG modules which are both quasi-isomorphic to $\Bbbk$. Therefore the ideal is acyclic. Now suppose that $1 < a \le n$. There is a short exact sequence
\[
0 \to (e_n^a) \to (e_n^{a-1}) \to \s_{n-1}(a-1) \to 0.
\]
The first map is an inclusion of ideals. Here, $\s_{n-1}(a-1)$ is realized as the classes modulo $(e_n^a)$ of the polynomials $f e_n^{a-1}$ for $f \in \Bbbk[e_1,e_2,\ldots,e_{n-1}] \subset \sym_n$. Now induction implies that the middle and right terms are acyclic, and therefore so is the left term.
\end{proof}

Our next goal is to compute the cohomology of $\s_n(a)$ for the remaining cases $1 \le n < a \le p$. Although this can be done without leaving the realm of $p$-DG submodules, it is useful
pedagogically to consider certain sub-vector-spaces of $\s_n(a)$ which are preserved by $\dif$, but not by $\sym_n$; i.e., they are subcomplexes, not submodules.

First, we define the \emph{d-degree} $\mathrm{deg}(f)$ of symmetric monomials as follows.
\begin{equation}
\mathrm{deg}(e_1^{i_1} \cdots e_n^{i_n}) = \sum_{k=1}^n i_k.
\end{equation}
The d-degree pretends that
all elementary symmetric polynomials have the same degree, whereas in the usual grading $e_k$ has degree $2k$; it is not defined for homogeneous symmetric polynomials, only for monomials.

\begin{lemma}
Let $f$ be a symmetric monomial. Then $\dif(f) = f' + \mathrm{deg}(f) e_1 f$, where $f'$ is a linear combination of symmetric monomials having the same d-degree as $f$.
\end{lemma}

\begin{proof}
This follows immediately from \eqref{difek} and the Leibniz rule.
\end{proof}

For an interval $[a,b] \subset \NM$, let $N_{[a,b]} \subset \sym_n$ denote the $\Bbbk$-span of all symmetric monomials in $\sym_n$ having d-degree in $[a,b]$. Fix $k \in \NM$. By the above
claim, $N_{[kp+1, kp+p]}$ is a $p$-DG subcomplex of $\sym_n$, because the d-degree can not be lowered by $\dif$, nor can it be raised from a monomial whose d-degree is a multiple of $p$.
In similar fashion, the subspace $N_{[kp-a+1, kp-a+p]} v_a \subset \s_n(a)$ is a $p$-DG subcomplex.

\begin{cor}\label{cor-cohomology-other-rank-one-mod-less-than-p} For any $n\in \{1, \dots , p-1\}$ and $a\in \{n+1, \dots , p\}$, let $I_a = N_{[p-a+1,\infty]} v_a$, a $p$-DG submodule of
$\s_n(a)$. In other words, $$I_a:=\Bbbk\langle e_1^{i_1}\cdots e_n^{i_n}v_a|\sum_{k=1}^n i_{k}>p-a \rangle.$$ Then the natural quotient map $\s_n(a)\twoheadrightarrow \s_n(a)/I_a$, is a
quasi-isomorphism of p-DG $\sym_n$ modules. Equivalently, the inclusion of the subcomplex $N_{[0,p-a]} v_a$ is a quasi-isomorphism of $p$-complexes. \end{cor}

The case $a = p$ is none other than Lemma \ref{lemma-cohomology-of-sym-less-than-p}, because $\s_n(0) \cong \sym_n$ as $p$-DG modules (up to grading shift), and $I_a$ is the
augmentation ideal.

\begin{proof}Note that in the module $\s_n(a)=\sym_n\cdot v_a$, we have $\dif(e_n^{p-a}v_a)=0$. Therefore this element generates a $(\sym_n,\dif)$-stable submodule $\sym_n\cdot (e_n^{p-a}v_a)$. By Lemma \ref{lemma-cohomology-of-sym-less-than-p}, this submodule has one-dimensional cohomology, generated by $e_n^{p-a}v_a$. There is a short exact sequence of $p$-DG $\sym_n$-modules
$$0\lra \sym_n^\prime\cdot e_n^{p-a}v_a \lra \sym_n\cdot v_a \lra Q\lra 0,$$
where $\sym_n^\prime$ is the augmentation ideal, and the quotient module $Q$ is \[ Q:=\sum_{i=0}^{p-a}\sym_{n-1}\cdot e_n^i v_a. \] Thus $\s_n(a)$ is quasi-isomorphic to $Q$.

Now $Q$ has a decreasing filtration $F^\bullet$ by $\sym_{n-1}$-modules, given by
\[F^l:=\sum_{i=l}^{p-a}\sym_{n-1}\cdot e_n^i v_a\]
for $0 \le l \le p-a$.
The $l$-th subquotient is isomorphic to the $p$-DG module $\s_{n-1}(b)$ for $b = l+a$. By induction, $\s_{n-1}(b)$ is quasi-isomorphic to
$$\Bbbk\langle e_1^{i_1}\cdots e_{n-1}^{i_{n-1}}|\sum_{k=1}^{n-1}i_k\leq p-b\rangle.$$
Multiplying these monomials by $e_n^l$, we obtain monomials of degree $\le p-a$, as desired.
The result follows.
\end{proof}

%
\subsection{\texorpdfstring{$p$}{p}-DG symmetric polynomials: part II}
%

We next move on to investigate the cohomology of $\sym_n$ and its rank-one modules when $n\geq p$. The central theme will be a modification of the d-degree arguments above. For a monomial
$e_1^{i_1} \cdots e_n^{i_n}$, one considers the sum $\sum_{k=p}^n i_k$, which could be thought of as the relative d-degree over $\sym_{p-1}$. We use the relative d-degree to study
restrictions of $\sym_n$ modules to $\sym_{p-1}$.

The first immediate result is the following.

\begin{cor}\label{cor-cohomology-rank-one-modules-p} The natural inclusion map of $\Bbbk[e_p^p]$ with the zero differential into $\sym_p$ is a quasi-isomorphism of $p$-DG algebras. The rank-one module $\s_p(a)$ is quasi-isomorphic to the $p$-complex $\Bbbk[e_p^p]\cdot (e_p^{p-a}v_a)$ with the zero differential.
\end{cor}
\begin{proof}Since $\sym_{p-1} \subset \sym_p$ is an inclusion $p$-DG algebras, we may regard $\sym_p$ by restriction as a $p$-DG module over $\sym_{p-1}$. In this way $\sym_p$ splits into a direct sum of rank-one $\sym_{p-1}$-modules
\[
\sym_{p}\cong \bigoplus_{k=0}^{\infty}\sym_{p-1} \cdot e_p^k,
\]
and $\dif$ acts on each summand generator $e_p^k$ by $\dif(e_p^k)=ke_1e_p^k$. Therefore $\sym_{p-1}e_p^k \cong \s_{p-1}(k)$ as $p$-DG modules over $\sym_{p-1}$, and Corollary \ref{cor-contractible-rank-one-modules-less-than-p} shows that they are acyclic unless $p$ divides $k$. Moreover, when $p$ does divide $k$, $\sym_{p-1}e_p^{k}\cong \s_{p-1}(0)$, and by Lemma \ref{lemma-cohomology-of-sym-less-than-p} the inclusion $\Bbbk \cdot e_p^k \hookrightarrow \sym_{p-1}e_p^{k}$ is a quasi-isomorphism. The first claim follows.

The last statement follows from a similar consideration, and the fact that $\s_p(a)$ is isomorphic to the $p$-DG ideal generated by $e_p^a$ in $\sym_p$.
\end{proof}

Now we look at $\sym_{p+1}=\Bbbk[e_1,\dots, e_p, e_{p+1}]$. Again $\sym_{p-1}$ sits in it as a $p$-DG subalgebra, and by restriction $\sym_{p+1}$ decomposes into blocks of $p$-DG $\sym_{p-1}$-modules
\[
\sym_{p+1}\cong \bigoplus_{d \in \N} \sym_{p-1} \cdot E_d,
\]
where $E_d$ is the spanned by a basis $\{e_p^re_{p+1}^s|r, s \in \N , r+s=d\}$. This can be seen readily from the differential action on the basis elements:
\begin{equation}\label{eqn-dif-on-basis-of-blocks-p-plus-1}
\dif(e_p^re_{p+1}^s)=(r+s)e_1e_p^re_{p+1}^s-(p+1)re_p^{r-1}e_{p+1}^{s+1}.
\end{equation}
We analyze the $p$-DG structure of each block $\sym_{p-1} \cdot E_d$. By equation (\ref{eqn-dif-on-basis-of-blocks-p-plus-1}), there exists on it an increasing $\dif$-stable filtration $F_\bullet$ with $F_{-1}=0$ and $F_j:=\sum_{i=0}^j\sym_{p-1}\cdot e_p^{i}e_{p+1}^{d-i}$. The subquotients of this filtration, as $p$-DG modules over $\sym_{p-1}$, are all isomorphic to $\s_{p-1}(d)$. Therefore by Corollary \ref{cor-contractible-rank-one-modules-less-than-p}, this block will be acyclic unless $d$ is a multiple of $p$. When $d=kp$ for some $k\in \N$, the subquotients in this more interesting block are all isomorphic to $\s_{p-1}(0)\cong\sym_{p-1}$. Hence the quotient map
\begin{equation}
\sym_{p-1} \cdot E_d \lra \sym_{p-1} \cdot E_d/(\sym_{p-1}^\prime \cdot E_d),
\end{equation}
is a surjective quasi-isomorphism. The right hand side splits into a direct sum of $p$-complexes
\begin{equation}\label{eqn-decomp-interesting-block}
 (e_p^d) \bigoplus \left(\bigoplus_{l=0}^{k-1}(e_p^{d-lp-1}e_{p+1}^{lp+1} \stackrel{1}{\lra} e_p^{d-lp-2}e_{p+1}^{lp+2} \stackrel{2}{\lra} \cdots \stackrel{p-1}{\lra} e_p^{d-lp-p}e_{p+1}^{lp+p})\right),
\end{equation}
where the differential action on the residue basis elements is indicated by multiplication with the corresponding number on the arrow. It follows that the inclusion $\Bbbk \cdot e_p^{d} \hookrightarrow \sym_{p-1}E_d$ is a quasi-isomorphism in this case.

\begin{lemma}\label{lemma-cohomomology-rank-one-modules-p-plus-1}
\begin{itemize}
\item[(i)] The $p$-DG algebra $\sym_{p+1}$ is quasi-isomorphic to its subalgebra $\Bbbk[e_p^p]$ with trivial differential.
\item[(ii)] The rank-one $p$-DG module $\s_{p+1}(1)$ is acyclic.
\item[(iii)] The rank-one $p$-DG modules $\s_{p+1}(a)=\sym_{p+1}\cdot v_a$, where $\dif_a(v_a)=ae_1 \cdot v_a$ and $a\in \{2, \dots, p-1\}$, have cohomologies isomorphic to the $p$-complex
    $$\Bbbk[e_p^p]\cdot \left(e_p^{p-a}e_{p+1}^0v_a\stackrel{a}{\lra}e_p^{p-a-1}e_{p+1}^{1}v_a\stackrel{a-1}{\lra}\cdots\stackrel{1}{\lra}e_p^0e_{p+1}^{p-a}v_a\right)$$
\end{itemize}
\end{lemma}
\begin{proof}
The first statement follows from the discussion before the lemma. The second one follows by a similar consideration as with the first. The only difference is that, using the same argument, the more interesting block is spanned by $\{e_p^ae_{p+1}^b|a, b \in \N , a+b+1=kp, k\in \N\}$. The analogous decomposition of (\ref{eqn-decomp-interesting-block}) in this case becomes
\begin{equation}\label{eqn-decomp-interesting-block-2}
\bigoplus_{l=0}^{k-1}(e_p^{(k-l)p-1}e_{p+1}^{lp} \stackrel{1}{\lra} e_p^{(k-l)p-2}e_{p+1}^{lp+1} \stackrel{2}{\lra} \cdots \stackrel{p-1}{\lra} e_p^{(k-l-1)p}e_{p+1}^{(l+1)p-1}),
\end{equation}
which is a direct sum of contractible $p$-complexes. The last case follows by a similar consideration as the first two cases, and we leave it to the reader as an exercise.
\end{proof}

\begin{lemma}\label{lemma-cohomology-rank-one-modules-greater-than-p}Let $n$ be an integer in $\{p+1, \dots, 2p-1\}$.
\begin{itemize}
\item[(i)]The natural inclusion $\Bbbk[e_p^p] \hookrightarrow \sym_n$, where $\Bbbk[e_p^p]$ is equipped with the zero differential, is a quasi-isomorphism of $p$-DG algebras.
\item[(ii)] When $1\leq a \leq n-p$, the rank-one $p$-DG module $\s_n(a)$  is acyclic.
\item[(iii)] When $n-p+1\leq a \leq p-1$, the $p$-DG module $\s_n(a)$ has its cohomology isomorphic to the $p$-complex
    $$\Bbbk[e_p^p]\cdot\left(\bigoplus_{i_p+i_{p+1}+\cdots + i_n= p-a}\Bbbk e_p^{i_p}e_{p+1}^{i_{p+1}}\cdots e_{n}^{i_n}v_a\right),$$
    where the $\dif$-stable $\Bbbk$-space in the bracket is equipped with the inherited differential from \eqref{difek}.
\end{itemize}
\end{lemma}
\begin{proof}
The proof is analogous to that of Corollary \ref{cor-contractible-rank-one-modules-less-than-p}. One first shows part $(ii)$ using an induction argument based on Lemma \ref{lemma-cohomomology-rank-one-modules-p-plus-1} and the short exact sequence of $p$-DG modules (ideals)
\[
0 \lra (e_n^a) \lra (e_n^{a-1}) \lra \s_{n-1}(a-1) \lra 0.
\]
Then to prove $(i)$, one starts from Corollary \ref{cor-cohomology-rank-one-modules-p} and inducts on $n$ using the short exact sequence
\[
0\lra (e_n) \lra \sym_n \lra \sym_{n-1}\lra 0.
\]
Here the ideal $(e_n)=\sym_n\cdot e_n \cong \s_n(1)$ is acyclic by part $(ii)$. To prove part $(iii)$, one uses an induction argument similar as done for Corollary \ref{cor-cohomology-other-rank-one-mod-less-than-p}. We leave this as an exercise to the reader.
\end{proof}

The above method generalizes immediately to the case of arbitrary $\sym_n$.

\begin{prop}\label{prop-cohomology-of-sym-n}Let $n$ be a natural number such that $kp\leq n \leq (k+1)p-1$ for some $k\in \N$, and $\s_n(a)$ be the rank-one $p$-DG module over $\sym_n$, which is generated by $v_a$ with $\dif_a(v_a)=ae_1v_a$.
\begin{itemize}
\item[(i)] The inclusion $\Bbbk[e_p^p,\cdots, e_{kp}^p]\lra \sym_n$, where $\Bbbk[e_p^p,\cdots, e_{kp}^p]$ is equipped with the trivial differential, is a quasi-isomorphism of $p$-DG algebras.
\item[(ii)] When  $1\leq a \leq n-kp$, the rank-one $p$-DG module $\s_n(a)$ is acyclic.
\item[(iii)]When $n-kp+1 \leq a \leq p-1$, the $p$-DG module $\s_n(a)$ has its cohomology isomorphic to the $p$-complex
    $$\Bbbk[e_p^p,\dots, e_{kp}^p]\cdot\left(\bigoplus_{i_{kp}+i_{kp+1}+\cdots + i_n= p-a}\Bbbk e_{kp}^{i_{kp}}e_{kp+1}^{i_{kp+1}}\cdots e_{n}^{i_n}v_a\right),$$
    where the $\dif$-stable $\Bbbk$-space in the bracket is equipped with the inherited differential from \eqref{difek}.
\end{itemize}
\end{prop}
\begin{proof}Exercise.
\end{proof}

\section{Grassmannian modules for symmetric polynomials}\label{sec-grassmannian}


\subsection{Definition}

In this chapter, we will be interested in the space $\sym_{a,b}:=\sym_{a}\boxtimes \sym_b$, viewed as a $p$-DG algebra or as a module over the $p$-DG algebras $\sym_{a,b}$ and
$\sym_{a+b}$. We refer to it as a Grassmannian module, because it is isomorphic to the equivariant cohomology ring $\mH^*_{GL(a+b)}(Gr(a,a+b))$ of the Grassmannian variety of $a$-planes in $\C^{a+b}$. As in the previous chapter, we write
\begin{equation}
\s_{a,b}:=\sym_a\boxtimes\sym_b=\sym_{a,b} \cdot v_{a,b},
\end{equation}
to indicate the free rank-one module over the algebra $\sym_{a,b}$ with generator $v_{a,b}$, viewed as a bimodule over $\sym_{a,b}$ and $\sym_{a+b}$. The degree of $v_{a,b}$ is $-ab$.

As a module over $\sym_{a+b}$, $\s_{a,b}$ is free (its graded rank agrees with the graded dimension of the non-equivariant cohomology ring $\mH^*(Gr(a,a+b))$). A pair of convenient bases over $\sym_{a+b}$ is given by Schur functions:
\begin{equation}\label{eqn-Schur-basis}
\left\{(\pi_{\lambda}\boxtimes 1)|\lambda\in P(a,b)\right\} ~\textrm{or}~ \left\{(1\boxtimes \pi_{\lambda})|\lambda\in P(b,a)\right\},
\end{equation}
where $P(a,b)$ contains those Young diagrams that fit into a height $a$, width $b$ rectangle.

\begin{defn}\label{def-dg-bimod}Consider the rank-one $\sym_{a,b}$-module $\s_{a,b}:=\sym_{a,b}\cdot v_{a,b}$. Given $k,l\in \Bbbk$, let $\dif_{k,l}$ be the differential on $\s_{a,b}$ which acts on the generator $v_{a,b}$ by
\[
\dif_{k,l}(v_{a,b}):=ke_1(x_1,\dots, x_a)v_{a,b}+le_1(x_{a+1},\dots, x_{a+b})v_{a,b},
\]
extended to the whole space via the Leibnitz rule
\[
\dif_{k,l}(fv_{a,b}):=\dif(f)v_{a,b}+f\dif_{k,l}(v_{a,b}),
\]
for any $f\in \sym_{a,b}$. This $p$-DG module will be denoted by $\s_{a,b}(k,l)$ in what follows. (It is assumed that the $p$-differential $\dif$ on $\sym_{a+b}$ or $\sym_{a,b}$ is specified as in Lemma \ref{lemma-difofpilambdalonghand}.)

We will abbreviate the elementary symmetric functions in the first $a$-variables (resp.~the last $b$-variables) by $e_1^{(a)}$ (resp.~$e_1^{(b)}$). Thus we have
\[
\dif_{k,l}(v_{a,b}):=ke_1^{(a)}v_{a,b}+le_1^{(b)}v_{a,b}.
\]
\end{defn}

It is clear that the differential on $\s_{a,b}(k,l)$ is $p$-nilpotent if and only if $k,l\in \F_p$, which we will assume from now on. Moreover, $\s_{a,b}(k,l)$ is a $p$-DG bimodule over the
$p$-DG algebras $\sym_{a,b}$ and $\sym_{a+b}$ (see \S\ref{subsec-bimodules} for more on $p$-DG bimodules and their endomorphism rings).

\subsection{Cofibrance}

The most basic question to ask is for which parameters $k,l \in \F_p$ will the bimodules $\s_{a,b}(k,l)$ be cofibrant as $p$-DG modules over $\sym_{a+b}$. Because $\sym_{a+b}$ is a positively
graded local ring, a finitely-generated $p$-DG module will be cofibrant if and only if it is a finite-cell module (see \cite[Definition 2.4]{EQ1}).

\begin{defn} \label{def-finite-cell} Let $A$ be a $p$-DG algebra which is graded local, and $M$ a $p$-DG (left or right) module over $A$. Then $M$ is called \emph{finite-cell} if there is a
finite, exhaustive, increasing $p$-DG filtration $F_\bullet$ on $M$ whose subquotients are isomorphic to $A$ with the natural differential (up to grading shifts). \end{defn}

\begin{prop}\label{prop-finite-cell-parameters}The $p$-DG module $\s_{a,b}(k,l)$ in Definition \ref{def-dg-bimod}, when considered as a $p$-DG right module over $\sym_{a+b}$, is finite-cell if either of the following conditions hold
\[
\left\{
\begin{array}{ll}
k\equiv -b & (\mathrm{mod}~p),\\
l \equiv 0 & (\mathrm{mod}~p).
\end{array}
\right.
\quad
\quad
\left\{
\begin{array}{ll}
k\equiv 0 & (\mathrm{mod}~p),\\
l \equiv -a & (\mathrm{mod}~p).
\end{array}
\right.
\]
\end{prop}
\begin{proof} First consider the $p$-DG module $\s_{a,b}(-b,0)$ together with the basis
$
\{(\pi_{\lambda}\boxtimes 1)v_{a,b}| \lambda \in P(a,b)\}
$ over $\sym_{a+b}$ (see equation \eqref{eqn-Schur-basis}). We claim that the $\Bbbk$-span of this basis is closed with respect to the module differential $\dif_{-b,0}$. By the differential action on Schur functions (Lemma \ref{lemma-difofpilambdalonghand}),
\begin{eqnarray*}
\dif((\pi_{\lambda}\boxtimes 1)v_{a,b}) & = &\sum_{\lambda+\square}C(\square)(\pi_{\lambda+\square}\boxtimes 1)v_{a,b}-b(e_1^{(a)}\pi_{\lambda}\boxtimes 1)v_{a,b}\\
& = & \sum_{\lambda+\square}(C(\square)-b)(\pi_{\lambda+\square}\boxtimes 1)v_{a,b},
\end{eqnarray*}
where we have used the Pieri rule in the second equality. Note that, because these Schur polynomials live in $\sym_a$, the partitions in this sum have at most $a$ rows, but nothing constrains them to have at most $b$ columns. However, whenever $\l + \square$ has more than $b$ columns, it must be the case that $\l_1 = b$ and $C(\square)=b$; such a partition appears with coefficient zero, as desired.

Thus the $\Bbbk$-span of this basis forms a $p$-complex, and one can choose a $\dif_{-b,0}$-stable increasing filtration with one-dimensional subquotients. Inducing this filtration from
$\Bbbk$ to $\sym_{a+b}$ yields the desired filtration of $\s_{a,b}(-b,0)$.

A similar argument works for $\s_{a,b}(0,-a)$, using the second basis from \eqref{eqn-Schur-basis} instead. \end{proof}

\begin{rmk} It seems as though the converse to Proposition \ref{prop-finite-cell-parameters} also holds for large characteristic. In other words, for a given $a,b$, there may be small
primes for which some other pair $(k,l)$ gives rise to a cofibrant $\sym_{a+b}$-module $\s_{a,b}(k,l)$, but for large enough primes, these two seem to be the only cofibrant possibilities.
We have checked this for small values of $a$ and $b$, and conjecture that it is the case in general. \end{rmk}

\subsection{Duality}

There is a non-degenerate, $\sym_{a+b}$-bilinear pairing on $\sym_{a,b}$ given as follows. Identify $\sym_{a,b}$ as the space of polynomials that are symmetric in the first $a$-variables as well as the last $b$-variables. Consider the following element $D_{a,b}$ of the nilHecke algebra on $a+b$ strands.
\begin{equation}\label{eqn-trace}
D_{a,b}:=
\begin{DGCpicture}
\DGCstrand(0,0)(2,3)[$^1$`$\empty$]
\DGCstrand(0.5,0)(2.5,3)[$^2$`$\empty$]
\DGCstrand(1.5,0)(3.5,3)[$^a$`$\empty$]
\DGCstrand(2.5,0)(0,3)[$^{a+1}$`$\empty$]
\DGCstrand(3.5,0)(1,3)[$^{a+b}$`$\empty$]
\DGCcoupon*(0.65,0.1)(1.45,0.25){$\dots$}
\DGCcoupon*(2.65,0.1)(3.45,0.25){$\dots$}
\end{DGCpicture}~.
\end{equation}
It is not hard to observe that $D_{a,b}$ sends any polynomial in $\sym_{a,b}$ to a polynomial in $\sym_{a+b}$. Then for any elements $f,g\in \sym_{a,b}$, the bilinear form is defined by
\begin{equation}
\langle-,-\rangle:\sym_{a,b}\times \sym_{a,b}\lra \sym_{a+b}, \quad \langle f,g\rangle:= D_{a,b}(fg).
\end{equation}
One can also view this as a pairing on $\s_{a,b}$, pairing $f v_{a,b}$ and $g v_{a,b}$ to yield $D_{a,b}(fg) v_{a,b}$.

\begin{lemma} Under this bilinear pairing, the Schur bases
\begin{equation}\label{eqn-dual-basis}
\left\{(\pi_{\lambda}\boxtimes 1)\cdot v_{a,b}|\lambda\in P(a,b)\right\} ~\textrm{and}~ \left\{(1\boxtimes \pi_{\lambda})\cdot v_{a,b}|\lambda\in P(b,a)\right\},
\end{equation}
are orthogonal to each other:
\begin{equation}\label{eqn-Schur-dual-basis}
\langle(\pi_{\lambda}\boxtimes 1),(1\boxtimes \pi_{\hat{\mu}})\rangle=(-1)^{|\hat{\mu}|}\delta_{\lambda,\mu}.
\end{equation}
Here, $\l$ and $\mu$ are both in $P(a,b)$, and $\hat{\mu}$ is the partition obtained as follows: take the complement of $\mu$ inside an $a\times b$-box, rotate it by 180 degrees to get a Young diagram, and take the transpose. The absolute value indicates the number of boxes in $\hat{\mu}$. \end{lemma}

\begin{proof} If $\lambda\in P(a,b)$ and $\hat{\mu}\in P(b,a)$, we have
\begin{align*}
D_{a,b}(\pi_{\lambda}\boxtimes\pi_{\hat{\mu}}) & = D_{a,b}(D_a(x_1^{\lambda_1+a-1}\dots x_a^{\lambda_a})D_b(x_{a+1}^{\hat{\mu}_1+b-1}\dots x_{a+b}^{\hat{\mu}_b}))\\
&=D_{a+b}(x_1^{\lambda_1+a-1}\dots x_a^{\lambda_a}x_{a+1}^{\hat{\mu}_1+b-1}\dots x_{a+b}^{\hat{\mu}_b}).
\end{align*}
The last term is non-zero if and only if the sequence of numbers $(a+b-1,a+b-2,\dots, 1,0)$ can be obtained from the sequences $(\lambda_1+(a-1),\dots,\lambda_a)$ and $(\hat{\mu}_1+(b-1),\dots, \hat{\mu}_{b})$ by a shuffling, which in turn happens if and only if $\lambda=\mu$.  It is not hard to check that
the minimal permutation corresponding to the shuffling has length $|\hat{\mu}|$.
\end{proof}

Now we investigate how this form interacts with the $p$-DG structure.

\begin{prop}\label{prop-bilinear-p-dg-pairing}The bilinear form $\langle-,-\rangle:\s_{a,b}\times\s_{a,b}\lra \sym_{a+b}$ extends to a $p$-DG pairing \[
\langle-,-\rangle:\s_{a,b}(k,l)\times \s_{a,b}(-b-k,-a-l)\lra \sym_{a+b},
\]
which is $\dif$-invariant, $\sym_{a+b}$-linear and non-degenerate.
\end{prop}
\begin{proof}The linearity and non-degeneracy properties are clear.
The $\dif$-invariance condition means that for any $f\in \s_{a,b}(k,l)$ and $g\in \s_{a,b}(-b-k,-a-l)$, we have
\begin{equation}\label{eqn-dif-invariance}
\dif(\langle f,g\rangle)=\langle\dif_{k,l}(f),g\rangle+\langle f,\dif_{-b-k,-a-l}(g)\rangle.
\end{equation}
It suffices to check \eqref{eqn-dif-invariance} on our dual bases above. We wish to show that
\begin{align}\label{eqn-dif-invariance-dual-bases}
\langle(\dif_{k,l}(\pi_{\lambda}\boxtimes 1)v),(1\boxtimes\pi_{\gamma})u\rangle +\langle(\pi_{\lambda}\boxtimes 1)v,\dif_{-b-k,-a-l}((1\boxtimes\pi_{\gamma})u)\rangle =0
\end{align}
holds for all $\lambda\in P(a,b)$ and $\gamma\in P(b,a)$, where $v$, $u$ are respectively the module generators of $\s_{a,b}(k,l)$ and $\s_{a,b}(-b-k, -a-l)$.

First off, we compute the differential action on an element of the form $(\pi_{\lambda}\boxtimes 1)v\in \s_{a,b}(k,l)$,
\begin{align*}
\dif_{k,l}((\pi_{\lambda}\boxtimes 1)v) & = (\dif(\pi_{\lambda})\boxtimes 1)v+ (\pi_{\lambda}\boxtimes 1)\dif_{k,l}(v)\\
& = \sum_{\lambda+\square}C(\square) (\pi_{\lambda+\square}\boxtimes 1)v+ k(e_1^{(a)}\pi_{\lambda}\boxtimes 1)v+l(\pi_{\lambda}\boxtimes e_1^{(b)})v\\
& = \sum_{\lambda+\square}(C(\square)+k-l)(\pi_{\lambda+\square}\boxtimes 1)+le_1(\pi_\lambda\boxtimes 1)v.
\end{align*}
In the second equality, we have used the differential action on $\pi_{\lambda}$ as in Lemma \ref{lemma-difofpilambdalonghand}, and we have abbreviated the elementary symmetric functions as in Definition \ref{def-dg-bimod}. Likewise, we have that
\begin{align*}
\dif_{-b-k,-a-l}((1\boxtimes \pi_{\gamma})u) & = \sum_{\gamma+\square}(C(\square)-a-l+b+k)(1\boxtimes \pi_{\gamma+\square})u -(b+k)e_1(1\boxtimes \pi_{\gamma})u.
\end{align*}

Now we show that equation \eqref{eqn-dif-invariance-dual-bases} holds. From \eqref{eqn-Schur-dual-basis} we see that, for the pairing to be non-zero, there are only two possibilities:
\begin{enumerate}
\item[(i)]When $\lambda$ is obtained from adding one box in the $r$-th row to $\hat{\gamma}$. In this situation, there are two non-zero terms in the sum. The first one comes from adding a box to $\lambda$ (differentiating $\pi_{\lambda}$), and pairing with $\pi_{\hat{\gamma}}$. The second term arises from adding a box to $\pi_{\gamma}$ and pairing with $\pi_{\lambda}$. In both cases, the box added is the unique one missing from the $a\times b$-rectangle with $\lambda$ in the upper left and $\hat{\gamma}$ in the lower right. This is clear in the diagram below, where $\lambda = (9^2,8,4^2,3)$ is shaded red, and $\gamma = (7,5,4^3,3,2,1^3)$ is shaded green.

\drawing{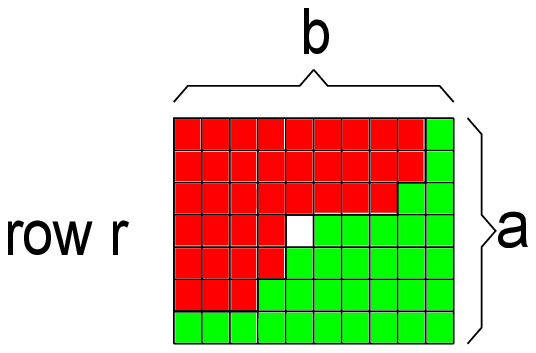}

    These two coefficients add up to
    \begin{align*}
    \left(\lambda_r-r+1+k-l\right)(-1)^{|\hat{\gamma}|}+\left(\gamma_{b-\lambda_r}-(b-\lambda_r)+1-a-l+b+k\right)(-1)^{|\hat{\gamma}|+1} = 0,
    \end{align*}
    since $\gamma_{b-\lambda_r}+r=a$.
\item[(ii)]When $\lambda=\hat{\gamma}$. In this case, either $\lambda_1=b$ or $\gamma_1=a$. Assuming the first, the non-zero coefficients combine to give
    \begin{align*}
    \langle (b+k-l)(\pi_{\lambda+\square_1}\boxtimes 1)v,(1\boxtimes\pi_{\hat{\lambda}})u\rangle+\langle le_1(\pi_{\lambda}\boxtimes 1)v,(1\boxtimes\pi_{\hat{\lambda}})u\rangle
    \\+\langle (\pi_{\lambda}\boxtimes 1)v,(-b-k)e_1(1\boxtimes\pi_{\hat{\lambda}})u \rangle  =  (b+k-l+l-b-k)e_1 = 0,
    \end{align*}
    where $\lambda+\square_1$ denotes $\lambda$ with one box added to the first row, of which the content equals $b$. Likewise, a similar computation shows that, when $\gamma_1=a$, the non-trivial coefficients also add up to zero.
\end{enumerate}
This finishes the proof of the proposition.
\end{proof}

Notice that the two system of parameters in Proposition \ref{prop-finite-cell-parameters} are in fact dual to each other in the sense of Proposition \ref{prop-bilinear-p-dg-pairing}. Thus,
up to duality, there is a unique cofibrant $p$-DG $\sym_{a+b}$-module structure on $\s_{a,b}$.

\begin{rmk}[A notational simplification]\label{rmk-Sab-notation} Because of Propositions \ref{prop-finite-cell-parameters} and \ref{prop-bilinear-p-dg-pairing}, we will abuse notation and
denote $\s_{a,b}(-b,0)$ by $\s_{a,b}$, unless otherwise specified. \end{rmk}

\subsection{Diagrammatic notation}

Following \cite{KLMS}, we develop diagrammatic notation for elements of $\s_{a,b}$ and its endomorphism ring. Alternatively, the bimodule $\s_{a,b}$ gives rise to an exact tensor product functor
\begin{equation}\label{eqn-bimod-giving-tensor-functor}
(\s_{a,b})\otimes_{\sym_{a+b}}(-): \sym_{a+b}\dmod\lra \sym_{a,b}\dmod,
\end{equation}
and its bimodule endomorphisms can be viewed as natural transformations of this functor.

We depict an element $x \in \sym_n$ as a thick line of width $n$ carrying the label $x$. We view such an element as an endomorphism of a $\sym_n$-module, so it has degree equal to the degree of $x$. Multiplication of elements is given by vertically concatenating pictures:
\[
\begin{DGCpicture}
\DGCstrand[thick](0,-0.35)(0,1.35)[$^n$]
\DGCcoupon(-0.3,0.25)(0.3,0.75){$x$}
\end{DGCpicture}\ , \quad \quad \quad \quad
\begin{DGCpicture}
\DGCstrand[thick](0,-0.35)(0,1.35)[$^n$]
\DGCcoupon(-0.3,0.25)(0.3,0.75){$xy$}
\end{DGCpicture}
~=~
\begin{DGCpicture}
\DGCstrand[thick](0,-0.35)(0,1.35)[$^n$]
\DGCcoupon(-0.3,0.6)(0.3,1.1){$x$}
\DGCcoupon(-0.3,-0.15)(0.3,0.35){$y$}
\end{DGCpicture}
\ .
\]

We depict the bimodule generator $v_{a,b}$ of $\s_{a,b}(-b,0)$, unique up to a non-zero constant, by
\[
v_{a,b}:=
\begin{DGCpicture}
\DGCPLstrand[thick](1,0)(1,1)[$^{a+b}$]
\DGCPLstrand[thick](1,1)(0,2)[`$_a$]
\DGCPLstrand[thick](1,1)(2,2)[`$_b$]
\end{DGCpicture},
\]
and we declare the degree of the element equals $-ab$. Then $\s_{a,b}$ can be identified with the span of diagrams
\[
\s_{a,b}\cong \left\{
\begin{DGCpicture}
\DGCPLstrand[thick](1,0)(1,1)[$^{a+b}$]
\DGCPLstrand[thick](1,1)(0,2)[`$_a$]
\DGCPLstrand[thick](1,1)(2,2)[`$_b$]
\DGCcoupon(0.2,1.25)(0.8,1.75){$f$}
\DGCcoupon(1.2,1.25)(1.8,1.75){$g$}
\end{DGCpicture}\Bigg| f\in \sym_a, g\in \sym_b\right\},
\]
while placing a box on the bottom strand indicates the action of $\sym_{a+b}$. These diagrams are
subject to the \emph{Grassmannian sliding relation}:
\[
\begin{DGCpicture}
\DGCPLstrand[thick](1,0)(1,1)[$^{a+b}$]
\DGCPLstrand[thick](1,1)(0,2)[`$_a$]
\DGCPLstrand[thick](1,1)(2,2)[`$_b$]
\DGCcoupon(0.625,.25)(1.375,.75){$e_k$}
\end{DGCpicture}
=\sum_{l=0}^k
\begin{DGCpicture}
\DGCPLstrand[thick](1,0)(1,1)[$^{a+b}$]
\DGCPLstrand[thick](1,1)(0,2)[`$_a$]
\DGCPLstrand[thick](1,1)(2,2)[`$_b$]
\DGCcoupon(0.125,1.25)(0.875,1.75){$e_l$}
\DGCcoupon(1.125,1.25)(1.875,1.75){$e_{k-l}$}
\end{DGCpicture}~,
\]
where it is understood that $e_m(x_1,\dots,x_n)=0$ if $m> n$.
The differential is given by the formula
\begin{equation}\label{eqn-d-action-mod-generator}
\dif\left(~
\begin{DGCpicture}
\DGCPLstrand[thick](1,0)(1,1)[$^{a+b}$]
\DGCPLstrand[thick](1,1)(0,2)[`$_a$]
\DGCPLstrand[thick](1,1)(2,2)[`$_b$]
\end{DGCpicture}
~\right)
=
-b\begin{DGCpicture}
\DGCPLstrand[thick](1,0)(1,1)[$^{a+b}$]
\DGCPLstrand[thick](1,1)(0,2)[`$_a$]
\DGCPLstrand[thick](1,1)(2,2)[`$_b$]
\DGCcoupon(0.2,1.25)(0.8,1.75){$e_1$}
\end{DGCpicture}~.
\end{equation}

The dual bimodule $\s_{a,b}(-b,0)^\vee = \s_{a,b}(0,-a)$ has generator $v_{a,b}^\vee$ in degree $-ab$, which we denote by
\[
v_{a,b}^\vee:=
\begin{DGCpicture}
\DGCPLstrand[thick](0,0)(1,1)[$^a$]
\DGCPLstrand[thick](2,0)(1,1)[$^b$]
\DGCPLstrand[thick](1,1)(1,2)[`$_{a+b}$]
\end{DGCpicture}~.
\]
Placing polynomials on the strands has the obvious meaning, and the dual Grassmannian sliding relation is evident. The differential on the dual module is given by
\begin{equation}\label{eqn-d-action-dual-mod-generator}
\dif\left(~
\begin{DGCpicture}
\DGCPLstrand[thick](0,0)(1,1)[$^a$]
\DGCPLstrand[thick](2,0)(1,1)[$^b$]
\DGCPLstrand[thick](1,1)(1,2)[`$_{a+b}$]
\end{DGCpicture}
~\right)=
-a
\begin{DGCpicture}
\DGCPLstrand[thick](0,0)(1,1)[$^a$]
\DGCPLstrand[thick](2,0)(1,1)[$^b$]
\DGCPLstrand[thick](1,1)(1,2)[`$_{a+b}$]
\DGCcoupon(1.2,0.25)(1.8,0.75){$e_1$}
\end{DGCpicture}
\end{equation}

The endomorphism algebra of the module $\s_{a,b}(-b,0)$ over $\sym_{a+b}$, which is identified with
\begin{eqnarray}\label{eqn-endo-iso-to-matrix}
\END_{\sym_{a+b}}(\s_{a,b}(-b,0)) & \cong & \s_{a,b}(-b,0)\otimes_{\sym_{a+b}}\s_{a,b}(-b,0)^\vee \nonumber\\
 & \cong & \s_{a,b}(-b,0)\otimes_{\sym_{a+b}}\s_{a,b}(0,-a)
\end{eqnarray}
under the bilinear pairing of Proposition \ref{prop-bilinear-p-dg-pairing}, inherits a natural $p$-DG algebra structure. Elements of this $p$-DG algebra, generated by $v_{a,b} \otimes v_{a,b}^\vee$ can be viewed diagrammatically as polynomials on the gluing of the two diagrams along the common edge $\sym_{a+b}$ (see below for an example).  If we choose a filtered $p$-DG basis for $\s_{a,b}(-b,0)$ as in the proof of Proposition \ref{prop-finite-cell-parameters}, we may identify the endomorphism algebra with a graded-matrix algebra of size ${a+b \brack a}^2$ with coefficients in $\sym_{a+b}$. A basis for the matrix algebra is given by
\begin{equation}\label{eqn-Grass-matrix-basis}
\left\{
(-1)^{|\hat{\mu}|}
\begin{DGCpicture}
\DGCPLstrand[thick](0,0)(1,1.15)[$^a$]
\DGCPLstrand[thick](2,0)(1,1.15)[$^b$]
\DGCPLstrand[thick](1,1.15)(1,1.85)
\DGCPLstrand[thick](1,1.85)(0,3)[`$_a$]
\DGCPLstrand[thick](1,1.85)(2,3)[`$_b$]
\DGCcoupon(1.15,0.35)(1.85,0.85){$\pi_{\hat{\mu}}$}
\DGCcoupon(0.15,2.15)(0.85,2.65){$\pi_\lambda$}
\end{DGCpicture}~\Bigg|
\lambda, \mu \in P(a,b)
\right\}.
\end{equation}
The inherited differential, under the isomorphism \eqref{eqn-endo-iso-to-matrix}, acts on the lowest degree basis element as follows,
\begin{equation}\label{eqn-d-action-matrix-generator}
\dif\left(~
\begin{DGCpicture}
\DGCPLstrand[thick](0,0)(1,1.15)[$^a$]
\DGCPLstrand[thick](2,0)(1,1.15)[$^b$]
\DGCPLstrand[thick](1,1.15)(1,1.85)
\DGCPLstrand[thick](1,1.85)(0,3)[`$_a$]
\DGCPLstrand[thick](1,1.85)(2,3)[`$_b$]
\end{DGCpicture}
~\right)=
-b
\begin{DGCpicture}
\DGCPLstrand[thick](0,0)(1,1.15)[$^a$]
\DGCPLstrand[thick](2,0)(1,1.15)[$^b$]
\DGCPLstrand[thick](1,1.15)(1,1.85)
\DGCPLstrand[thick](1,1.85)(0,3)[`$_a$]
\DGCPLstrand[thick](1,1.85)(2,3)[`$_b$]
\DGCcoupon(0.15,2.15)(0.85,2.65){$e_1$}
\end{DGCpicture}
-a
\begin{DGCpicture}
\DGCPLstrand[thick](0,0)(1,1.15)[$^a$]
\DGCPLstrand[thick](2,0)(1,1.15)[$^b$]
\DGCPLstrand[thick](1,1.15)(1,1.85)
\DGCPLstrand[thick](1,1.85)(0,3)[`$_a$]
\DGCPLstrand[thick](1,1.85)(2,3)[`$_b$]
\DGCcoupon(1.15,0.35)(1.85,0.85){$e_1$}
\end{DGCpicture}~,
\end{equation}
a juxtaposition of the differential action on the generators of $\s_{a,b}(-b,0)$ and $\s_{a,b}(0,-a)$. This formula can be seen as a generalization of the differential action \eqref{eqn-dif-on-nilHecke-generator} on the local divided difference operator, the latter being the special case when $a=b=1$.

To compose two endomorphisms, one may vertically stack their diagrams, resulting in an element of $\s_{a,b} \otimes \s_{a,b}^\vee \otimes \s_{a,b} \otimes \s_{a,b}^\vee$, and then pair the
middle factors against each other using the bilinear pairing $\s_{a,b}^\vee \otimes \s_{a,b} \to \sym_{a+b}$. Diagrammatically, this corresponds to the \emph{duality relation}:
\begin{equation} \label{eq-duality}
\begin{DGCpicture}
\DGCPLstrand[thick](1,0)(1,0.5)[$^{a+b}$]
\DGCstrand[thick](1,0.5)(0,1.5)
\DGCstrand[thick](1,0.5)(2,1.5)
\DGCstrand[thick](0,1.5)(1,2.5)\DGCdot.{2.15}[l]{$^a$}
\DGCstrand[thick](2,1.5)(1,2.5)\DGCdot.{2.15}[r]{$^b$}
\DGCPLstrand[thick](1,2.5)(1,3)
\DGCcoupon(-0.3,1.25)(0.3,1.75){$\pi_{\alpha}$}
\DGCcoupon(1.7,1.25)(2.3,1.75){$\pi_{\beta}$}
\end{DGCpicture}
=\left\{
\begin{array}{ll}
(-1)^{|\hat{\alpha}|}~
\begin{DGCpicture}
\DGCstrand[thick](0,0)(0,2)[$^{a+b}$]
\end{DGCpicture}& \textrm{if $\beta=\hat{\alpha}$,}\\
& \\
0 & \textrm{otherwise.}
\end{array}
\right.
\end{equation}

Because $\END_{\sym_{a,b}}(\s_{a,b}) \cong \sym_{a,b}$ as a $\sym_{a,b}$-bimodule, one can view this endomorphism algebra representing the identity functor on $\sym_{a,b}$-modules.  Under this identification, one has the \emph{identity decomposition relation}:
\begin{equation} \label{eq-identitydecomp}
\begin{DGCpicture}
\DGCPLstrand[thick](0,0)(0,3)[$^a$`$_a$]
\DGCPLstrand[thick](2,0)(2,3)[$^b$`$_b$]
\end{DGCpicture}
=\sum_{\alpha\in P(a,b) }(-1)^{|\hat{\alpha}|}
\begin{DGCpicture}
\DGCPLstrand[thick](0,0)(1,1.15)[$^a$]
\DGCPLstrand[thick](2,0)(1,1.15)[$^b$]
\DGCPLstrand[thick](1,1.15)(1,1.85)
\DGCPLstrand[thick](1,1.85)(0,3)[`$_a$]
\DGCPLstrand[thick](1,1.85)(2,3)[`$_b$]
\DGCcoupon(1.15,0.35)(1.85,0.85){$\pi_{\hat{\alpha}}$}
\DGCcoupon(0.15,2.15)(0.85,2.65){$\pi_\alpha$}
\end{DGCpicture}~.
\end{equation}

Recall from Remark \ref{rmk-Sab-notation} that we have abbreviated $\s_{a,b}(-b,0)$ just by $\s_{a,b}$.
More generally, if $\underline{a}:=(a_1,a_2,\dots,a_k)\in \N^k$ is a decomposition of an integer $n\in \N$, so that $a_1+a_2+\cdots+a_k=n$, we set $$\sym_{\underline{a}}:=\sym_{a_1}\boxtimes \sym_{a_2}\boxtimes \cdots \boxtimes \sym_{a_k},$$
regarded as a $p$-DG algebra, and define the \emph{generalized $p$-DG Grassmannian module} (or partial flag variety module)
\[
\s_{\underline{a}}:=\sym_{\underline{a}}\cdot v_{\underline{a}},
\]
where the differential acts on the generator by
\[
\dif(v_{\underline{a}}):=\sum_{i=1}^{k-1}-(a_{i+1}+\dots+a_{k})e_1(x_{a_1+\dots+a_i}, \dots, x_{a_1+\dots+a_{i+1}-1})v_{\underline{a}}.
\]
This module generator is depicted by
\[
v_{\underline{a}}=
\begin{DGCpicture}
\DGCPLstrand[thick](1,0.5)(1,1)[$^{n}$]
\DGCPLstrand[thick](1,1)(0,2)[`$_{a_1}$]
\DGCPLstrand[thick](1,1)(0.5,2)[`$_{a_2}$]
\DGCPLstrand[thick](1,1)(2,2)[`$_{a_k}$]
\DGCcoupon*(0.8,1.8)(1.7,2){$\dots$}
\end{DGCpicture}=
\begin{DGCpicture}
\DGCPLstrand[thick](1,0.5)(1,1)[$^{n}$]
\DGCPLstrand[thick](1,1)(0,2)[`$_{a_1}$]
\DGCPLstrand[thick](0.5,1.5)(1,2)[`$_{a_2}$]
\DGCPLstrand[thick](1,1)(2,2)[`$_{a_k}$]
\DGCcoupon*(1,1.8)(1.9,2){$\dots$}
\end{DGCpicture},
\]
where the second equality follows from the associativity of tensor product:
\[
(\s_{a_1,a_2}\boxtimes \s_{a_3}\boxtimes\cdots\boxtimes\s_{a_n})\otimes_{\sym_{(a_1+a_2,\dots, a_n)}}\s_{a_1+a_2,\dots,a_n}\cong \s_{a_1,a_2,\dots, a_n}.
\]
The differential action on the generator is then depicted by
\begin{equation}\label{eqn-d-action-generalized-Grassmann}
\dif
\left(
\begin{DGCpicture}
\DGCPLstrand[thick](1,0.5)(1,1)[$^{n}$]
\DGCPLstrand[thick](1,1)(0,2)[`$_{a_1}$]
\DGCPLstrand[thick](1,1)(1,2)[`$_{a_i}$]
\DGCPLstrand[thick](1,1)(2,2)[`$_{a_k}$]
\DGCcoupon*(0.2,1.8)(0.8,2){$\dots$}
\DGCcoupon*(1.2,1.8)(1.8,2){$\dots$}
\end{DGCpicture}
\right)=-\sum_{i=1}^{k-1}(a_{i+1}+\dots+a_k)
\begin{DGCpicture}
\DGCPLstrand[thick](1,0.5)(1,1)[$^{n}$]
\DGCPLstrand[thick](1,1)(0,2)[`$_{a_1}$]
\DGCPLstrand[thick](1,1)(1,2)[`$_{a_i}$]
\DGCcoupon(0.75,1.35)(1.25,1.75){$e_1$}
\DGCPLstrand[thick](1,1)(2,2)[`$_{a_k}$]
\DGCcoupon*(0.2,1.8)(0.8,2){$\dots$}
\DGCcoupon*(1.2,1.8)(1.8,2){$\dots$}
\end{DGCpicture}.
\end{equation}
The thick strand on the bottom is utilized, as before, to emphasize that we are regarding $\s_{\underline{a}}$ also as a right $p$-DG module over $\sym_{n}$.

Starting with the case when the decomposition $\underline{a}$ has only two parts, an easy inductive argument on the number $k$ of the parts of the decomposition establishes the following result.

\begin{lemma}\label{lemma-generalized-Grassmann-basis}
If $\underline{a}=(a_1,\dots,a_k)$ is a decomposition of $n$, then the generalized $p$-DG Grassmannian module has a $\dif$-stable basis over the algebra $\sym_n$ consisting of elements of the form
\[
\left\{
(\pi_{\lambda_1}(x_1,\dots, x_{a_1})\pi_{\lambda_2}(x_1,\dots, x_{a_1+a_2})\cdots\pi_{\lambda_{k-1}}(x_1,\dots, x_{a_1+\dots+a_{k-1}}))v_{\underline{a}}
\right\},
\]
where each $\lambda_i$ ranges over all partitions that fit into an $(a_1+\dots+a_i)\times a_{i+1}$-box. \hfill$\square$
\end{lemma}

The top-degree $\sym_n$ summand is generated by the product of all Young diagrams $\lambda_i$ that are full rectangles. The canonical projection map from $\s_{\underline{a}}$ to its top-degree summand defines a non-degenerate trace form on $\s_{\underline{a}}$. This projection can be obtained as an iteration of the rank-two traces as follows.
\[
D_{\underline{a}}:=D_{a_1,a_2+\dots+a_k}\circ D_{a_2,a_3+\dots+a_n}\circ\cdots\circ D_{a_{n-1},a_n},
\]
where each term on the right hand side is given by \eqref{eqn-trace} with appropriate indices.
Then, with respect to the trace form, it is easy to check that the dual module
\[
\s_{\underline{a}}^\vee:=\HOM_{\sym_n}(\s_{\underline{a}},\sym_n),
\]
with the natural induced differential, is isomorphic to the $p$-DG $(\sym_n,\sym_{\underline{a}})$-bimodule
$\sym_{\underline{a}}\cdot v_{\underline{a}}^\vee$,
where $v_{\underline{a}}^\vee$ is acted upon the differential by
\[
\dif(v_{\underline{a}}^\vee)=\sum_{i=2}^k-(a_1+\dots+a_{i-1})v_{\underline{a}}^\vee e_1(x_{a_1+\dots+a_{i-1}+1},\dots,x_{a_{1}+\dots+a_{i}}).
\]
As for $\s_{\underline{a}}$, we depict the generator for the dual module by

\begin{equation}
v^\vee_{\underline{a}}:=
\begin{DGCpicture}
\DGCPLstrand[thick](0,0)(1,1)[$^{a_1}$]
\DGCPLstrand[thick](.5,0)(1,1)[$^{a_2}$]
\DGCPLstrand[thick](2,0)(1,1)[$^{a_k}$]
\DGCPLstrand[thick](1,1)(1,1.5)[`$_{n}$]
\DGCcoupon*(0.75,0)(1.35,0.2){$\dots$}
\end{DGCpicture}
=
\begin{DGCpicture}
\DGCPLstrand[thick](0,0)(1,1)[$^{a_1}$]
\DGCPLstrand[thick](.5,0)(1.3,0.75)[$^{a_2}$]
\DGCPLstrand[thick](2,0)(1,1)[$^{a_k}$]
\DGCPLstrand[thick](1,1)(1,1.5)[`$_{n}$]
\DGCcoupon*(0.75,0)(1.9,0.2){$\dots$}
\end{DGCpicture},
\end{equation}
so that the differential action will be diagrammatically written as
\begin{equation}\label{eqn-d-action-dual-generalized-Grassmann}
\dif\left(~
\begin{DGCpicture}
\DGCPLstrand[thick](0,0)(1,1)[$^{a_1}$]
\DGCPLstrand[thick](1,0)(1,1)[$^{a_i}$]
\DGCPLstrand[thick](2,0)(1,1)[$^{a_k}$]
\DGCPLstrand[thick](1,1)(1,1.5)[`$_{n}$]
\DGCcoupon*(0.25,0)(0.75,0.2){$\dots$}
\DGCcoupon*(1.25,0)(1.75,0.2){$\dots$}
\end{DGCpicture}
~\right)=\sum_{i=2}^k
-(a_1+\dots+a_{i-1})
\begin{DGCpicture}
\DGCPLstrand[thick](0,0)(1,1)[$^{a_1}$]
\DGCPLstrand[thick](1,0)(1,1)[$^{a_i}$]
\DGCcoupon(0.75,0.25)(1.25,0.6){$e_1$}
\DGCPLstrand[thick](2,0)(1,1)[$^{a_k}$]
\DGCPLstrand[thick](1,1)(1,1.5)[`$_{n}$]
\DGCcoupon*(0.25,0)(0.75,0.2){$\dots$}
\DGCcoupon*(1.25,0)(1.75,0.2){$\dots$}
\end{DGCpicture} \ .
\end{equation}
In particular, when each $a_i=1$, we have, from equations \eqref{eqn-d-action-generalized-Grassmann} and \eqref{eqn-d-action-dual-generalized-Grassmann}, the following special cases:
\begin{subequations}
\begin{equation}\label{eqn-d-full-flag}
\dif
\left(
\begin{DGCpicture}
\DGCPLstrand[thick](1,0.5)(1,1)[$^{n}$]
\DGCPLstrand(1,1)(0,2)
\DGCPLstrand(1,1)(1,2)[`$_i$]
\DGCPLstrand(1,1)(2,2)
\DGCcoupon*(0.2,1.8)(0.8,2){$\dots$}
\DGCcoupon*(1.2,1.8)(1.8,2){$\dots$}
\end{DGCpicture}
\right)=-\sum_{i=1}^{n-1}(n-i)
\begin{DGCpicture}
\DGCPLstrand[thick](1,0.5)(1,1)[$^{n}$]
\DGCPLstrand(1,1)(0,2)
\DGCPLstrand(1,1)(1,2)[`$_i$]
\DGCdot{1.5}
\DGCPLstrand(1,1)(2,2)
\DGCcoupon*(0.2,1.8)(0.8,2){$\dots$}
\DGCcoupon*(1.2,1.8)(1.8,2){$\dots$}
\end{DGCpicture}.
\end{equation}
Here the $i$ in the diagram indicates the $i$th strand of thickness one, rather than the thickness. Likewise, the differential acts on the dual module generator by
\begin{equation}\label{eqn-d-dull-full-flag}
\dif\left(~
\begin{DGCpicture}
\DGCPLstrand(0,0)(1,1)
\DGCPLstrand(1,0)(1,1)[$^{i}$]
\DGCPLstrand(2,0)(1,1)
\DGCPLstrand[thick](1,1)(1,1.5)[`$_{n}$]
\DGCcoupon*(0.25,0)(0.75,0.2){$\dots$}
\DGCcoupon*(1.25,0)(1.75,0.2){$\dots$}
\end{DGCpicture}
~\right)=
-\sum_{i=2}^k
(i-1)
\begin{DGCpicture}
\DGCPLstrand(0,0)(1,1)
\DGCPLstrand(1,0)(1,1)[$^{i}$]
\DGCdot{0.5}
\DGCPLstrand(2,0)(1,1)
\DGCPLstrand[thick](1,1)(1,1.5)[`$_{n}$]
\DGCcoupon*(0.25,0)(0.75,0.2){$\dots$}
\DGCcoupon*(1.25,0)(1.75,0.2){$\dots$}
\end{DGCpicture}.
\end{equation}
\end{subequations}
Gluing these differentials together recovers the differential on $\NH_n$ defined in \cite[Chapter 3]{KQ}.

Dualizing the basis constructed in  Lemma \ref{lemma-generalized-Grassmann-basis}, we have obtained the next result.

\begin{lemma}\label{lemma-dual-generalized-Grassmann-basis}
If $\underline{a}$ is a decomposition of $n$, then the dual generalized Grassmannian module $\s_{\underline{a}}^\vee$ has a $\dif$-stable basis over $\sym_n$ given by
\[
\{v_{\underline{a}}^\vee \pi_{\mu_2}(x_{a_1+1},\dots, x_n)\pi_{\mu_3}(x_{a_1+a_{2}+1},\dots, x_n)\cdots \pi_{\mu_k}(x_{a_1+\dots +a_{k-1}+1},\dots, x_n)\},
\]
where $\mu_i$'s range over partitions that fit into a rectangle of size $a_i\times (a_{i+1}+\dots+a_k)$ for each $i=2,\dots,k$.
\end{lemma}

Finally, we can consider the space of $\sym_n$-linear morphisms between two different generalized Grassmannian modules.

\begin{prop}\label{prop-morphism-space-half}
Let $\underline{a}=(a_1,\dots,a_k)$ and $\underline{b}=(b_1,\dots, b_l)$ be two decompositions of $n\in \N$. Then the space of $\sym_n$-linear $p$-DG homomorphisms from $\s_{\underline{b}}$ to $\s_{\underline{a}}$ is canonically isomorphic to
\[
\HOM_{\sym_n}(\s_{\underline{b}},\s_{\underline{a}})\cong \s_{\underline{a}}\o_{\sym_n}\s^\vee_{\underline{b}} \cong \HOM_{\sym_n}((\s_{\underline{a}}^\vee)^\vee, \s_{\underline{b}}^\vee).
\]
The space in the middle has a $\dif$-stable basis over $\sym_n$, which is given, diagrammatically, by
\[
\left\{
\begin{DGCpicture}
\DGCPLstrand[thick](0,0)(1,1.15)[$^{b_1}$]
\DGCPLstrand[thick](2,0)(1,1.15)[$^{b_l}$]
\DGCPLstrand[thick](1,0)(1.5,0.65)[$^{b_2}$]
\DGCPLstrand[thick](1,1.15)(1,1.85)
\DGCPLstrand[thick](1,1.85)(0,3)[`$_{a_1}$]
\DGCPLstrand[thick](1,1.85)(2,3)[`$_{a_k}$]
\DGCPLstrand[thick](0.5,2.45)(1,3)[`$_{a_2}$]
\DGCcoupon(1.5,0.1)(2.1,0.4){$_{{{\mu_l}}}$}
\DGCcoupon(1,0.7)(1.6,1){$_{{{\mu_2}}}$}
\DGCcoupon*(1.2,0)(1.8,0.1){$\dots$}
\DGCcoupon(0.3,2)(1.0,2.3){$_{_{{\lambda_{k-1}}}}$}
\DGCcoupon(-0.1,2.6)(0.6,2.9){$_{{{\lambda_1}}}$}
\DGCcoupon*(1.2,2.9)(1.8,3){$\dots$}
\end{DGCpicture}
\Bigg|
\begin{array}{l}
\lambda_i\in P((a_1+\dots+a_i), a_{i+1}), \quad i=1,\dots, k-1\\
\mu_j\in P(b_i, (b_{i+1}+\dots+b_l)), \quad j=2,\dots, l.
\end{array}
\right\}.
\]
Here, each $\lambda_i$ stands for a Schur function whose Young diagram fits into an $(a_1+\dots+a_i)\times a_{i+1}$-box, while each $\mu_j$ is a Schur function whose Young diagram fits into a rectangle of size $b_i\times (b_{i+1}+\dots+b_l)$.
\end{prop}
\begin{proof}
The result follows directly from Lemma \ref{lemma-generalized-Grassmann-basis} and Lemma \ref{lemma-dual-generalized-Grassmann-basis}.
\end{proof}
Notice that, if all $a_i$ and $b_j$'s are equal to $1$, this diagrammatic basis recovers the basis in \cite[Propostion 2.16]{KLMS}, up to a sign, in the matrix decomposition of $\NH_n$ over its center.

\subsection{Categorification of \texorpdfstring{$\dot{U}^+$}{U plus}}

We recall the definition of $\dot{U}^+$, the positive half of quantum $\mathfrak{sl}(2)$ a la Beilinson-Lusztig-MacPherson, over the ground ring
\[
\mathbb{O}_p:=K_0(\DC^c(\Bbbk))\cong \Z[q]/(1+q^2+\dots+q^{2(p-1)}).
\]
As a free $\mathbb{O}_p$-module, it is generated by symbols $E^{(a)}$ for any $a\in \N$:
\[
\dot{U}^+=\bigoplus_{a\in \N}\mathbb{O}_p\cdot E^{(a)}.
\]
The algebra structure is given by
\[
E^{(a)}E^{(b)}={a+b \brack a} E^{(a+b)}.
\]
It is understood that the quantum integers in the above formula are specialized in $\mathbb{O}_p$. For instance, it is easy to see that, although $E^{(p)}\neq 0$,
\[
E^p=
[p]!E^{(p)}=0
\]
in this algebra, since $[p]=q^{-p+1}(1+q^2+\dots+q^{2(p-1)})=0$ in $\mathbb{O}_p$.

The algebra $\dot{U}^+$ is a \emph{twisted bialgebra} over $\mathbb{O}_p$. The coproduct structure $r:\dot{U}^+\lra \dot{U}^+\otimes_{\mathbb{O}_p}\dot{U}^+$ is given by
\begin{equation}\label{eqn-comul}
r(E^{(a)})=\sum_{k=0}^a q^{-k(a-k)}E^{(k)}\otimes E^{(a-k)}.
\end{equation}
The multiplication on $\dot{U}^+\otimes_{\mathbb{O}_p}\dot{U}^+$ is a twist of the naive multiplication. It is determined on basis elements by
\[
(E^{(a)}\otimes 1)(1\otimes E^{(b)})=E^{(a)}\o E^{(b)},\quad (1\otimes E^{(a)})(E^{(b)}\o 1)=q^{ab}E^{(b)}\o E^{(a)}.
\]
This twisted multiplication is chosen so that $r$ is an algebra homomorphism from $\dot{U}^+$ to $\dot{U}^+\otimes_{\mathbb{O}_p}\dot{U}^+$.

The \emph{(positive) half of the small quantum} $\mathfrak{sl}(2)$, written as $\dot{u}^+$, sits inside $\dot{U}^+$ as the subalgebra generated by $E$. Our goal in this section is to give a $p$-DG monoidal categorification of the bialgebra structure on $\dot{U}^+$, as well as the embedding of $\dot{u}^+$ inside it.

\begin{defn} \label{def-half-U-thick}
Let $\DC(\sym)$ denote the $p$-DG derived category defined as follows.
 \begin{itemize}
\item The category $\DC(\sym)$ is the direct sum of the $p$-DG derived categories $\sym_a$, one for each $a\in \N$:
\[
\DC(\sym):=\bigoplus_{a\in \N}\DC(\sym_a).
\]
Here by convention, we set $\sym_0$ to be the ground field $\Bbbk$ with the trivial $p$-differential action.
\item Let $\MC:\DC(\sym) \otimes \DC(\sym)\lra \DC(\sym)$ be the functor given by derived tensoring with the bimodule
$\s^\vee = \oplus_{a,b\in \N}{\s_{a,b}^\vee}$. That is, componentwise, this functor takes
\[
{}_{a+b}\MC_{a,b}:\DC(\sym_{a,b})\lra \DC(\sym_{a+b}), \quad
(M,N)\mapsto \s_{a,b}^\vee\otimes_{\sym_{a,b}}^\mathbf{L}(M\boxtimes N).
\]
\end{itemize}
\end{defn}

\begin{rmk}\label{rmk-U-plus-diagram}
Alternatively, we may regard $\DC(\sym)$ as a $p$-DG monoidal category as follows. Generating 1-morphisms will be thick strands of thickness $a \in \NM$, which we call $\EC^{(a)}$.  Generating 2-morphisms will be elements of $\sym_a$ attached to a strand of thickness $a$, and trivalent vertices $v_{a,b}$ and $v_{a,b}^\vee$. These 2-morphisms satisfy the relations \eqref{eq-duality} and \eqref{eq-identitydecomp}, as well as the \emph{associativity relation}:
\begin{equation}
\begin{DGCpicture}
\DGCPLstrand[thick](1,0)(1,1)[$^{a+b+c}$]
\DGCPLstrand[thick](1,1)(0,2)[`$_a$]
\DGCPLstrand[thick](0.5,1.5)(1,2)[`$_b$]
\DGCPLstrand[thick](1,1)(2,2)[`$_c$]
\end{DGCpicture}
=
\begin{DGCpicture}
\DGCPLstrand[thick](1,0)(1,1)[$^{a+b+c}$]
\DGCPLstrand[thick](1,1)(0,2)[`$_a$]
\DGCPLstrand[thick](1.5,1.5)(1,2)[`$_b$]
\DGCPLstrand[thick](1,1)(2,2)[`$_c$]
\end{DGCpicture},
\quad \quad \quad
\begin{DGCpicture}
\DGCPLstrand[thick](0,0)(1,1)[$^a$]
\DGCPLstrand[thick](2,0)(1,1)[$^c$]
\DGCPLstrand[thick](1,0)(0.5,0.5)[$^b$]
\DGCPLstrand[thick](1,1)(1,2)[`$_{a+b+c}$]
\end{DGCpicture}
=
\begin{DGCpicture}
\DGCPLstrand[thick](0,0)(1,1)[$^a$]
\DGCPLstrand[thick](2,0)(1,1)[$^c$]
\DGCPLstrand[thick](1,0)(1.5,0.5)[$^b$]
\DGCPLstrand[thick](1,1)(1,2)[`$_{a+b+c}$]
\end{DGCpicture}.
\end{equation}
The differential is defined on symmetric polynomials as in Lemma \ref{lemma-difofpilambdalonghand}, and on trivalent vertices as in \eqref{eqn-d-action-mod-generator} and \eqref{eqn-d-action-dual-mod-generator}.

One should think of this diagrammatic description of $\DC(\sym)$ as an ``upward-oriented" version of the category $\Uthick$ from \cite{KLMS}, with an equipped differential. This differential diagrammatic calculus gives an explicit way of describing the $p$-DG endomorphism algebra of the functor $\MC$ (see Proposition \ref{prop-morphism-space-half}).
\end{rmk}

Besides the bimodule $\s^\vee$ considered in Definition \ref{def-half-U-thick}, we will also utilize the following bimodule.

\begin{defn}
\label{def-comultiplication-bimod}
Let $\RC:\DC(\sym)\lra \DC(\sym)\otimes \DC(\sym)$ be the derived tensor with the following bimodule,
\[
\RC=\bigoplus_{a,b\in \N} \sym_{a,b}\cdot v_{a,b},
\]
where the generators $v_{a,b}$ have degree zero, equipped with the trivial differential action: $\dif(v_{a,b})=0$. Componentwise, each $\sym_{a,b} v_{a,b}$, considered as a $p$-DG bimodule over the $p$-DG algebras $(\sym_{a,b},\sym_{a+b})$, gives rise to a functor
\[
(\sym_{a,b}\cdot v_{a,b})\otimes_{\sym_{a+b}}^{\mathbf{L}}(-):\DC(\sym_{a+b})\lra \DC(\sym_a\otimes \sym_b), \quad M\mapsto (\sym_{a,b}v_{a,b})\otimes^{\mathbf{L}}_{\sym_{a+b}}M.
\]
\end{defn}

\begin{lemma}On the level of Grothendieck groups, the following K\"{u}nneth property holds:
\[
K_0(\sym_a\otimes \sym_b)\cong K_0(\sym_a)\otimes_{\mathbb{O}_p}K_0(\sym_{b}).
\]
\end{lemma}
\begin{proof}This is true because, by Proposition \ref{prop-cohomology-of-sym-n}, each $\sym_n$ is quasi-isomorphic to a polynomial algebra with the trivial differential. It follows that, for any $a\in \N$
\[
K_0(\sym_a)\cong \mathbb{O}_p[\sym_a],
\]
the right-hand-side meaning the free module generated by the class of $\sym_a$ as a left $p$-DG module over itself. Then both sides of the equation to be shown are isomorphic to $\mathbb{O}_p$. The result follows.
\end{proof}

\begin{thm} \label{thm-half-U-thick}
The $p$-DG Grothendieck group of $\DC(\sym)$ is isomorphic to the positive half $\dot{U}^+$ of the divided powers integral form of quantum $\sl_2$ as a bialgebra. The functors $\MC$ and $\RC$ categorify the multiplication and comultiplication structures on $\dot{U}^+$ respectively.
\end{thm}

\begin{proof}It is clear from the definition that
\[
K_0(\DC(\sym))\cong \bigoplus_{a\in \N}\mathbb{O}_p[\sym_a].
\]
To construct the desired isomorphism, we send the class of the rank-one free module $[\s_a(0)]$ (see Definition \ref{def-rank-one-mod}) to the generator $E^{(a)}$ of $\dot{U}^+$.

By the previous lemma, the functors $\MC$ and $\RC$ induce maps on the Grothendieck groups
\[
[\MC]:K_0(\sym)\otimes_{\mathbb{O}_p}K_0(\sym)\lra K_0(\sym),\quad
[\RC]:K_0(\sym)\lra K_0(\sym)\otimes_{\mathbb{O}_p}K_0(\sym).
\]
It then suffices to check their actions on the symbols of the free modules. The bimodule $\s^\vee_{a,b}$, as a left $\sym_{a+b}$-module, has a filtration of graded dimension ${a+b\brack a}$ whose subquotients are all isomorphic to the free module $\sym_{a+b}$ (this was shown in Proposition \ref{prop-finite-cell-parameters}). Therefore, when acting on the free modules $\sym_a\boxtimes \sym_b$, one has
\[
{[\s_{a,b}^\vee\otimes_{\sym_{a,b}}^\mathbf{L}(\sym_a\boxtimes\sym_b)]=[\s_{a,b}^\vee]={a+b\brack a}[\sym_{a+b}]}.
\]
Therefore the functor $\s^\vee$ categorifies the product structure. Similarly, one can check that $\RC$ categorifies the coproduct structure $r$ on $\dot{U}^+$.
\end{proof}

\begin{rmk}\label{rmk-why-twisting-comultiplication}
Because the comultiplication on $\dot{U}^+$ is twisted, we must use the bimodule $\RC$ rather than $\oplus_{a,b\in \N}\s_{a,b}$. The latter
would send the class of $[\sym_{a+b}]$, up to a grading shift, to ${a+b\brack a}[\sym_a\otimes \sym_b]$, again by Proposition \ref{prop-finite-cell-parameters}. A similar categorical twist was used in \cite{KQ}
when categorifying the comultiplication on the small quantum group. \end{rmk}

\begin{rmk}
The proof of this theorem can also be given purely diagrammatically in terms of idempotents and Fc-filtrations. In the next chapter, we will recall the concept of Fc-filtrations, as defined in \cite{EQ1}.

In \cite{KQ} it was shown that the decomposition of $\EC^a$ into $a!$ copies of $\EC^{(a)}$ is a DG-filtration (also reviewed in the next chapter), using the $p$-DG structure on the nilHecke algebra. This can be seen as an extension of that result.
\end{rmk}

In \cite{KQ}, the authors constructed a categorification of $\dot{u}^+$ by placing a $p$-DG structure on the nilHecke algebras. This is a family of $p$-DG algebras
\[
\NH=\bigoplus_{a\in \N}\NH_n,
\]
where the $p$-DG structure is described, on the local generators, via equation \eqref{eqn-dif-on-nilHecke-generator}. We now explain how $\DC(\NH)$ embeds inside $\DC(\sym)$, and this will provide a categorification of the inclusion of $\dot{u}^+$ into $\dot{U}^+$.

As shown in \cite[Proposition 3.14]{KQ}, the differential on $\NH_n$ appears naturally by identifying $\NH_n$ as an endomorphism algebra
\[
\NH_n\cong \END_{\sym_n}(\PC_n^+),
\]
where $\PC_n^+)$ is described in Remark \ref{rmk-dif-on-NH}. Therefore, one can regard $\PC_n^+$ as a bimodule over $(\NH_n,\sym_n)$. Its $\sym_n$-dual $\PC_n^\vee=\HOM_{\sym_n}(\PC_n^+,\sym_n)$ is therefore a $p$-DG bimodule over $(\sym_n, \NH_n)$, which in turn provides a derived tensor functor
\[
\PC_n^\vee\otimes^\mathbf{L}_{\NH_n}(-):\DC(\NH_n)\lra \DC(\sym_n), \quad M \mapsto \PC_n^\vee\otimes^{\mathbf{L}}_{\NH_n}M.
\]
Defining $\PC$ as the direct sum of all $\PC_n^\vee$, this gives rise to a functor
\[
\PC^\vee: \DC(\NH)\lra \DC(\sym).
\]

\begin{cor}\label{cor-categorical-embedding-half-u} The functor $\PC^\vee$ is a fully-faithful embedding of $\DC(\NH)$ into $\DC(\sym)$. On the Grothendieck group level, this categorical embedding categorifies the embedding of $\dot{u}^+$ into $\dot{U}^+$.
\end{cor}
\begin{proof}Since $\DC(\NH)=\oplus_{a\in \N}\DC(\NH_a)$, it suffices to show the fully-faithful embedding componentwise for each $a$. If $a\geq p$, it is shown in \cite[Corollary 3.17]{KQ} that $\DC(\NH_a)\cong 0$, and the claim above is trivially true. Thus we are reduced to proving it when $a\in\{0,1,\dots, p-1\}$. To this end, we recall from \cite[Proposition 3.26]{KQ} that $\PC_{a}^+$ is a compact cofibrant generator of $\DC(\NH_a)$. Therefore, to show that the above functor is fully-faithful, it suffices to compute the endomorphism spaces of $\PC_a^+$ before and after we apply the functor. On the one hand, the cofibrance of $\PC_a^+$ in $\DC(\NH_a)$ tells us that the endomorphism in this category can be computed using the usual $\HOM$-space
\[
\mathbf{R}\END_{\NH_a}(\PC_a^+)\cong \END_{\NH_a}(\PC_a^+)\cong \sym_a.
\]
On the other hand, after applying the functor, the new object we obtain equals
\[
\PC_a^\vee\otimes_{\NH_a}^\mathbf{L}\PC_a^+\cong \PC_a^\vee\otimes_{\NH_a}\PC_a^+\cong \END_{\NH_a}(\PC_a^+)\cong \sym_a,
\]
whose endomorphism space inside $\DC(\sym_a)$ also equals $\sym_a$. This proves the first part of our result.

On the level of $K_0$, we have known from \cite[Theorem 3.35]{KQ} that $\PC_a^+$ categorifies $E^{(a)}$ when $a\in \{0,1,\dots,p-1\}$. Hence, on $K_0$, the functor sends the element $E^{(a)}\in \dot{u}^+$
\[
[\PC_a]: E^{(a)}=[\PC_a^+]\mapsto [\sym_a]=E^{(a)}\in \dot{U}^+.
\]
The corollary now follows.
\end{proof}

\begin{rmk}One should notice that, if one considers homotopy category of $p$-DG modules over $\sym$ and $\NH$, the above functor actually provides an equivalence between the respective homotopy categories
\[
\PC^\vee:\HC(\NH)\lra \HC(\sym).
\]
The quasi-inverse is given by derived tensor product with the $(\NH_,\sym)$-bimodule
\[
\PC:=\bigoplus_{a\in \N} \PC_a^+.
\]
This is a special case of Lemma \ref{lemma-Morita-homotopy-category} to be proved in the next chapter.

However, the derived category $\DC(\NH_a)$ collapses to the zero category if $a\geq p$, so these functors fail to induce an equivalence of derived categories.
\end{rmk}

\section{Grothendieck groups and the Karoubi envelope}\label{sec-idempotents}

%
\subsection{Bimodules, endomorphisms, and differentials} \label{subsec-bimodules}
%

We gather here some basic discussion of differentials and bimodules.

Let $(A,\dif_A)$ and $(B,\dif_B)$ be two $p$-DG algebras, and $M$ be an $(A,B)$-bimodule. The bimodule structure can be lifted to the $p$-DG level if $M$ can be equipped with $p$-nilpotent differential $\dif_M$ such that the following \emph{compatibility} condition is satisfied: for any $a\in A$, $b\in B$ and $m\in M$, we have
\begin{equation}\label{eqn-compatibility-cond}
\dif_M(amb)=\dif_A(a)mb+a\dif_M(m)b+am\dif_B(b).
\end{equation}
We will drop the subscripts in the various differentials whenever no confusion can be caused.
Now if $(A,\dif_A)$ is a $p$-DG algebra, and $(M,\dif_M)$ is a $p$-DG left module over $A$, then the graded algebra $\END_A(M)$ of $A$-module endomorphisms is equipped with a natural differential $\dif_E$ as follows. If $\phi\in \END_{A}(M)^{\mathrm{op}}$, and $m\in M$, then
\[
\dif_E(\phi) (m):= \dif_M(\phi(m))-\phi(\dif_M(m)).
\]
It is easy to check that $\dif_E^p\equiv 0$, so that $(\END_A(M),\dif_E)$ is a $p$-DG algebra. Furthermore, if we regard $M$ as a bimodule over $(A,\END_A(M)^{\textrm{op}})$ via
\[
a\cdot(m)\cdot \phi:=\phi(am),
\]
where $\END_A(M)^{\textrm{op}}$ denotes the opposite algebra of $\END_A(M)$ with the differential action unchanged, then it is a simple exercise to check that the compatibility condition \eqref{eqn-compatibility-cond} is satisfied. In this way $M$ is then promoted to a $p$-DG bimodule.

More generally, the graded space $\HOM_A(M,N)$ of $A$-module morphisms between two $p$-DG modules also has a $p$-DG module structure with differential $\dif_H$, using the formula
\[ \dif_H(\phi)(m) = \dif_N(\phi(m)) - \phi(\dif_M(m)). \]

%
\subsection{The differential on the image of an idempotent} \label{subsec-differential-on-image-of-idempotent}
%

Let $\e$ be an idempotent in a $p$-DG algebra $A$. Applying the Leibniz rule to the equality $\e^2 = \e$, one obtains the following facts.
\begin{equation} \label{eq:de} \dif(\e) = \e \dif(\e) + \dif(\e) \e. \end{equation}
\begin{equation} \label{edeplusdee} \dif(\e) = 0 \textrm{ if and only if } \e \dif(\e) = 0 \textrm{ and } \dif(\e)\e = 0. \end{equation}
\begin{equation} \label{edee} \e \dif(\e) \e = 0. \end{equation}
In this section we investigate when the left $A$-module summand $A \e$ inherits a $p$-DG structure.

Suppose that $\e \in A$ is an idempotent with the property that $\dif(\e) \in A \e$, or equivalently $\e \dif(\e) = 0$. The subset $A \e \subset A$ is not just a summand of $A$ as a left $A$-module, but is also a $p$-DG submodule, and as such we call $\e$ a \emph{left submodule idempotent}. Let us write $\odif$ for the induced differential on the $p$-DG
submodule $A \e$. In particular, for any $a \in A$, one has the formula \begin{equation} \label{odifonAe} \odif(a \e) = \dif(a \e) \e, \end{equation} in which the final $\e$ is redundant.

Now, suppose that $\e \in A$ is an idempotent with the property that $\dif(1 - \e) \in A (1 - \e)$, or equivalently $\dif(\e) \e = 0$. Then $A \e$ can be viewed as a $p$-DG
quotient of $A$ by the $p$-DG submodule $A (1 - \e)$, so we call $\e$ a \emph{left quotient idempotent}. Let us write $\odif$ for the induced differential on $A \e$ as a $p$-DG quotient.
Once again, it is easy to see that \begin{equation*} \odif(a \e) = \dif(a \e) \e, \end{equation*} where now the final $\e$ is necessary.

\begin{defn} Let $\e \in A$ be an idempotent. We say that $\e$ is a \emph{(left) summand idempotent} if $\dif(\e)=0$. By \eqref{edeplusdee}, $A \e$ is both a left $p$-DG submodule and a left $p$-DG quotient (that is, its $A$-module complement is a $p$-DG submodule), so it is a $p$-DG summand. \end{defn}

\begin{defn} \label{defn:subquotsummand} Let $\e \in A$ be an idempotent. We say that $A \e$ is a \emph{(left) subquotient summand} or that $\e$ is a \emph{(left) subquotient idempotent}
if there exists an idempotent $\w$ dominating $\e$ (i.e. for which $\w \e = \e \w = \e$), such that $\w \dif(\w) = 0$, and $(\w - \e) \dif(\w - \e) = 0$. In other words, $A \e$ is a left
$p$-DG quotient of $A \w$, which is a left $p$-DG submodule of $A$, all of which are summands in the category of $A$-modules. In this case, we define the (left) differential $\odif$ on $A
\e$ by the formula \eqref{odifonAe}. \end{defn}

The differential $\odif$ on a subquotient summand $A \e$ agrees with its induced differential as a $p$-DG subquotient of $A$. From this, it is easy to see that $\odif^n(a \e) = \dif^n(a \e) \e$.

In our definition of a subquotient summand above, we gave conditions for $A \e$ to be a quotient of a submodule in the $p$-DG category (all of which are summands in the ordinary module
category). The condition for $A \e$ to be a submodule of a quotient appears slightly different: there exists an idempotent $\eta$ dominating $\e$ with $\dif(\eta) \eta = 0$ and $\e
\odif(\e) = \e \dif(\e) \eta = 0$, where $\odif$ was the induced differential on the quotient $A \eta$. However, we leave it as an exercise to prove that $(1-\eta + \e)$ satisfies the
conditions of $\w$ in Definition \ref{defn:subquotsummand}.

The notion of a \emph{right submodule idempotent} and a \emph{right quotient idempotent} are easy to intuit. Because $\e \dif(\e) = 0$ if and only if $\dif(1-\e) (1-\e) = 0$, it is easy to
see that a right submodule idempotent is a left quotient idempotent, and vice versa. Similarly, a right summand (resp. subquotient) idempotent is the same as a left summand (resp. subquotient) idempotent, which is why we omit the word left or right in these cases.

One may be tempted to formally define a differential on $A \e$ for any idempotent, by the formula \eqref{odifonAe}. However, when $\e$ is not a subquotient summand, the formula
$\odif^n(a \e) = \dif^n(a \e) \e$ may fail, and there is no guarantee that $\odif^p = 0$. There does not seem to be any way to place a natural $p$-DG module structure on a
general summand $A \e$.

\begin{rmk} \label{source-of-counterexamples} In the prequel \cite{EQ1} we defined a family of differentials on the Khovanov-Lauda category $\UC$. We demonstrated that, for two members of
this family, the idempotent decompositions used by Khovanov-Lauda are DG-filtrations (defined later in this chapter), and therefore the idempotents are subquotient idempotents. However, for
the other members of the family of differentials, these idempotents are not subquotient idempotents; in such a situation, if one defines $\odif$ by \eqref{odifonAe}, then $\odif^p \ne 0$. These faulty differentials can be used to construct several other counterexamples in this chapter. \end{rmk}

Now we consider the endomorphism ring of a subquotient summand $A \e$
\begin{equation}\label{eqn-endo-algebra-idempotent}
\END_A(A\e)\cong \e A \e ,\quad \phi \mapsto \phi(\e)=\e \phi(\e)
\end{equation}
which carries a natural differential $\dif_E$ by the discussion above. Let us change notation, and also refer to the induced differential on $\e A \e$ as $\odif$. Note that $\e A \e$ is not closed under $\dif_A$. Instead, given $\phi\in \END_A(A\e)$ which corresponds to $\e a \e \in \e A \e$ (i.e.~$\phi(x\e):=x\e a \e$), and $x\e \in A\e$, we have,
\begin{equation*}
\odif(\phi)(x\e) = x \e \dif_A(\e a \e) \e.
\end{equation*}
Comparing this with the corresponding element under the isomorphism \eqref{eqn-endo-algebra-idempotent}, we have proven the following.

\begin{lemma}\label{lemma-dif-endo-algebra-idempotent}Let $\e$ be a subquotient idempotent of a $p$-DG algebra $(A,\dif_A)$. Under the isomorphism $\END_A(A\e)\cong \e A \e$, the inherited differential on the algebra $\e A \e$, denoted $\odif$, is given by
\[
\odif(\e a \e)= \e \dif(\e a \e) \e,
\]
for any $a \in A$. \hfill$\square$
\end{lemma}

Note that when $\e$ is either a submodule or a quotient idempotent, the $\e$ on one side of the RHS will be redundant.

In similar fashion, we can consider subquotient idempotents of a left $p$-DG module $M$ over a $p$-DG algebra $A$. Given an idempotent $\e \in \END_A(M)$, one can call it a \emph{left
submodule idempotent} if its image is closed under $\dif_M$. This is equivalent to $\e$ being a right submodule idempotent inside $\END_A(M)$, i.e. $\dif_E(\e) \e = 0$. The switch
from left to right arises because $M$ is naturally a right module over $\END_A(M)$. In general, one can transfer notions about idempotents in the $p$-DG algebra $\END_A(M)^\op$ into notions
about the module $M$. We leave the details to the reader, and mention only this lemma.

\begin{lemma} \label{lemma-dif-hom-idempotent} Let $(A,\dif_A)$ be a $p$-DG algebra, and $(M,\dif_M)$ and $(N,\dif_N)$ be two $p$-DG modules. Let $\e \in \END_A(M)$ and $\eta \in \END_A(N)$ be subquotient idempotents, and let $(X,\dif_X)$ and $(Y,\dif_Y)$ be the induced $p$-DG modules on their respective images. Under the isomorphism
	\[ \HOM_A(X,Y) \cong \eta \HOM_A(M,N) \e, \]
the inherited differential on $\HOM_A(X,Y)$, denoted $\odif$, is given by
\[
\odif(\eta \phi \e) = \eta \dif_H(\eta \phi \e) \e,
\]
for any $\phi \in \HOM_A(M,N)$. \hfill$\square$
\end{lemma}

\subsection{A basic problem} \label{subsec-basic-problem}

Let us motivate the remainder of the chapter.

Given an algebra $A$, its Grothendieck group $K_0(A)$ is defined to be the split Grothendieck group of the additive category of finitely-generated projective left $A$-modules. It has a
basis given by the classes of indecomposable projectives. For a projective left $A$-module $M$, and a complete collection of primitive orthogonal idempotents $\{\e_i\} \subset \END_A(M)$
(i.e.~whose images are indecomposable projective summands of $M$), one has a relation in $K_0(A)$ that
\begin{equation} \label{whatwewant}
[M] = \sum_i [\Im \e_i].
 \end{equation}

Now suppose that $A$ is a $p$-DG algebra, and $M$ is a left $p$-DG module whose underlying $A$-module is finitely-generated projective, and $\{\e_i\}$ is a collection of idempotents as
above. In applications, we wish to show that the same relation \eqref{whatwewant} holds as an equality in the $p$-DG Grothendieck group of $A$. However, it may not, for three essential
reasons.

Firstly, the idempotents $\e_i$ may not be subquotient idempotents, and therefore $\Im \e_i$ may not even have a $p$-DG module structure.

Secondly, even when the $\e_i$ are each individually subquotient idempotents, they need not fit into a filtration of $M$ by $p$-DG submodules. Below, we will encapsulate the notion of a
filtration of $M$ by $p$-DG submodules for which the underlying $A$-modules of the subquotients are projective in the algorithmic definition of a DG-filtration.

Thirdly, even when $\Im \e_i$ are the subquotients in a DG-filtration, they need not be cofibrant! A reminder of the definition of the $p$-DG Grothendieck group is coming soon, but a key
point is that \eqref{whatwewant} is only guaranteed to hold when the subquotients are cofibrant. Unfortunately, $p$-DG algebras typically have very few cofibrant objects, as they correspond
roughly to summand idempotents (i.e. idempotents $\e$ with $\dif(\e)=0$). For this reason, we will pass from discussing $p$-DG algebras to discussing $p$-DG categories (i.e., $p$-DG
algebroids), which have many more cofibrant objects. The next few sections will review some $p$-DG homological algebra, rephrased in the language of $p$-DG categories.

In a $p$-DG category, one way to show that $\Im \e_i$ is cofibrant is to factor the idempotent through another object $N_i$ in the category (compatibly with the $p$-DG structure). Often,
the relation \begin{equation} \label{whatwewant2} [M] = \sum_i [N_i] \end{equation} is the one we desire in the first place. Below, we will encapsulate the notion of a filtration by $p$-DG
submodules whose subquotients are isomorphic to other objects in the algorithmic definition of a fantastic filtration, or Fc-filtration for short. When the idempotents $\e_i$ form a
Fc-filtration, the desired relation \eqref{whatwewant2} will hold in the $p$-DG Grothendieck group.

\subsection{Pre-additive categories}

We begin with a reminder about ordinary pre-additive categories (without any $p$-DG structure). These are generalizations of the concept of a $\Bbbk$-algebra.

A \emph{pre-additive category} or \emph{algebroid} $\AC$ is a (small) $\Bbbk$-linear category. It is convenient to regard such a category as an ``algebra'' $A = \oplus_{X,Y \in \Obj(\AC)}
\Hom_\AC(X,Y)$, with a collection of orthogonal idempotents $\1_X$ instead of a unit, one for each object of the category. Recall that a pre-additive category is called \emph{additive} if
it admits all finite direct sums on objects. For instance, the category with a single object whose endomorphism space equals $\Bbbk$ is pre-additive, while the category $\Bbbk-\mathrm{vec}$
of $\Bbbk$-vector spaces is additive. A pre-additive category with only one object is a $\Bbbk$-algebra.

\begin{defn}\label{def-C-mod}
\begin{enumerate}
\item[(i)] Let $\AC$ be a pre-additive category. A \emph{left (resp. right) module} $\MC$ over $\AC$ is a covariant functor
\[
\MC:\AC\lra \Bbbk\!-\!\mathrm{vec} \quad (\textrm{resp.}~\MC:\AC^{\mathrm{op}}\lra \Bbbk\!-\!\mathrm{vec}).
\]
Similarly, a \emph{bimodule} $\MC$ over $\AC$ is a covariant functor
\[
\MC:\AC \times \AC^{\mathrm{op}} \lra \Bbbk\!-\!\mathrm{vec}.
\]
\item[(ii)] A left (resp. right) module $\MC$ over $\AC$ is called \emph{representable} or \emph{principal} if it is isomorphic to a module of the form
\[
\Hom_{\AC}(X,-):\AC\lra \Bbbk\!-\!\mathrm{vec} \quad (\textrm{resp.}~\Hom_{\AC}(-,X):\AC\lra \Bbbk\!-\!\mathrm{vec}),
\]
for some object $X\in \AC$.
\end{enumerate}
\end{defn}

It is easy to check that the category of left (resp.~right) modules over a pre-additive category $\AC$ is abelian, which we will denote as $\AC\dmod$ (resp. $\AC^{\op}\dmod$). Via
the Yoneda embedding, the additive closure of $\AC$ is equivalent to the full subcategory in $\AC\dmod$ or $\AC^{\mathrm{op}}\dmod$ consisting of finite direct sums of representable
modules. As in the case of algebras, a projective module will be a summand of a principal module.

\begin{defn} \label{def-proj-C-mod}Let $\AC$ be an additive category.
\begin{enumerate}
\item[(i)] We often refer to a pair $(M,\e)$, consisting of an object $M\in \AC$ and an idempotent $\e \in \END_{\AC}(M)$, as an \emph{idempotent} in $\AC$. Given an idempotent $(M,\e)$, the \emph{left (resp.~right) projective module} $\AC\e$ (resp.~$\e\AC$) is the functor
$\AC\lra \Bbbk\!-\!\mathrm{vec}$, which assigns to an object $N\in \AC$ the vector space
$$\Hom_{\AC}(M,N) \cdot \e:=\{\phi \circ \e|\phi\in \Hom_{\AC}(M,N)\},$$
$$\left(\textrm{resp.}~\e \cdot \Hom_{\AC}(N,M):=\{\e \circ \psi|\psi\in \Hom_{\AC}(N,M)\},\right)$$
and to morphisms the natural composition of morphisms.
\item[(ii)] Given $\mathbb{X} = \{(M_i,\e_i)|i\in I\}$ a set of idempotents, define the \emph{partial idempotent completion category} $\AC(\mathbb{X})$ as follows. One has $\Obj (\AC(\mathbb{X})) = \Obj (\AC) \cup \mathbb{X}$, with $\AC$ included as a full subcategory in the natural way. Then,
    for any $X,Y \in \Obj (\AC)$ and $(M_i,\e_i), (M_j,\e_j) \in \mathbb{X}$, set
\begin{eqnarray*}
\Hom(X,(M_j,\e_j)) &=& \e_j \cdot \Hom(X,M_j), \\
\Hom((M_i,\e_i),Y) &=& \Hom(M_i,Y) \cdot \e_i, \\
\Hom((M_i, \e_i),(M_j,\e_j)) &=& \e_j \cdot
\Hom(M_i,M_j) \cdot \e_i.
\end{eqnarray*}
The composition maps are obvious.
\item[(iii)] An idempotent $(M,\e)$ \emph{factors} through an object $N$ if there are morphisms $u \co N \to M$ and $v \co M \to N$ such that $vu = \1_N$ and $uv = \e$. Note that any $(M,\e) \in \XM$ will factor in $\AC(\XM)$. A category in which every idempotent factors is called \emph{Karoubian} or \emph{idempotent-closed}.
\item[(iv)] When $\mathbb{X}$ contains every idempotent of $\AC$, the resulting category $\AC(\mathbb{X})$ is called the \emph{Karoubi envelope} or \emph{idempotent completion} $\Kar(\AC)$.
\end{enumerate}
\end{defn}

Our definition of the Karoubi envelope is slightly larger than the usual construction of $\Kar(\AC)$ in the literature, which ignores the objects coming from $\AC$ originally. These objects
are redundant, as $X \cong (X,1)$ in our construction of $\Kar(\AC)$.

In this paper, it will be important (once we introduce a $p$-DG structure) to distinguish between a partial idempotent completion category $\AC(\XM)$ which is Karoubian, and the true
Karoubi envelope $\Kar(\AC)$. As additive categories, however, there is no essential difference.

\begin{lemma}\label{lemma-classical-Morita}
Let $\mathbb{X}=\{(M_i,\e_i)|i\in I\}$ be a collection of idempotent endomorphisms in an additive category $\AC$. Then the category $\AC\dmod$ and $\AC(\mathbb{X})\dmod$ are Morita equivalent. In particular, there is an equivalence of abelian categories between $\AC\dmod$ and $\Kar(\AC)\dmod$.
\end{lemma}
\begin{proof}
Tensoring with the bimodules $\AC\oplus(\oplus_{i\in I})\AC\e_i$ and $\AC\oplus(\oplus_{i\in I}\e_i\AC)$ gives rise to exact functors between these module categories which are inverse to each other.
\end{proof}

From this discussion, it is clear that the notion of the (split) Grothendieck group of an additive category
$$K_0(\AC):=K_0(\Kar(\AC))$$
is an analogue of the usual Grothendieck group of finitely generated projective modules over an algebra. This is because, under the Yoneda embedding, $\Kar(\AC)$ is equivalent to the category of finitely generated projective modules inside $\AC\dmod$ or $\AC^{\mathrm{op}}\dmod$.

\subsection{\texorpdfstring{$p$}{p}-DG categories}

We recall some definitions from \cite[Section 2.5]{EQ1}.

\begin{defn}\label{def-p-dg-cat}Let $\Bbbk$ be a field of characteristic $p>0$.
\begin{enumerate}
\item[(i)]A graded $\Bbbk$-linear (pre-)additive category $\AC$ is called a \emph{$p$-DG category} if, for any object $X,Y\in \AC$, the graded morphism space $\HOM_{\AC}(X,Y)$ is equipped with a $p$-nilpotent \emph{differential} of degree $2$,
\[
\dif:\HOM_{\AC}^\bullet (X,Y)\lra \HOM_{\AC}^{\bullet+2}(X,Y),
\]
such that, for any $X,Y,Z\in \AC$, $f\in \HOM_{\AC}(X,Y)$ and $g\in \HOM_{\AC}(Y,Z)$, the following Leibnitz rule is satisfied
\[
\dif(g\circ f)=\dif(g)\circ f + g\circ \dif(f)\in \HOM_{\AC}(X,Z).
\]
\item[(ii)]Let $\AC$ be a $p$-DG category. A \emph{left (resp.~right) $p$-DG module} over $\AC$ is a covariant functor from $\AC$ (resp. $\AC^{\mathrm{op}}$) to the graded category $\Bbbk_\dif\dmod$ of $p$-complexes, which preserves the $p$-differential structures on morphism spaces. They form an abelian category $\AC_\dif\dmod$.
\item[(iii)]Let $\MC$ be a left (resp.~right) $p$-DG module over a $p$-DG category $\AC$. The module $\MC$ is called \emph{representable} if it is isomorphic to a module of the form
    \[
    \HOM_{\AC}(M,-):\AC\lra \Bbbk_\dif\dmod,\quad (\textrm{resp.}~\HOM_{\AC}(-,M):\AC^{\mathrm{op}}\lra \Bbbk_\dif\dmod),
    \]
    where $M$ is some object in $\AC$ depending on $\MC$.
\end{enumerate}
\end{defn}

A $p$-DG category with a unique object is the same concept as a $p$-DG algebra. Note that for any $X \in \Obj(\AC)$, $\dif(\1_X)=0$.

Let $\MC$ be a representable left (resp.~right) $p$-DG module over a $p$-DG category $\AC$, represented by an object $M\in \AC$. If we view the $p$-DG algebra $\END_{\AC}(M)$ as a $p$-DG category with a unique object, the natural functor
\[
\jmath_M:\END_{\AC}(M)\lra \AC,
\]
which sends the unique object to $M$, is a map of $p$-DG categories. It is clear that $\MC$ can then be identified with the induced module
\begin{equation}\label{eqn-parabolic-induced-mod}
\MC\cong \AC\otimes_{\jmath_M,\END_{\AC}(M)}\END_{\AC}(M) \quad \left(\textrm{resp.}~\MC\cong \END_{\AC}(M)\otimes_{\END_{\AC}(M),\jmath_M}\AC\right).
\end{equation}
This identification extends to an isomorphism of $p$-DG bimodules over $(\AC, \END_{\AC}(M))$ (resp. $(\END_{\AC}(M),\AC)$).  Note that $\END_{\AC}(M) \cong \END_{\AC_\dif\dmod}(\MC)$, and the latter is thus equipped with the structure of a $p$-DG algebra. We refer to this $p$-DG algebra as $A_M$.

In this fashion, our discussion of idempotents in $p$-DG algebras from Section \ref{subsec-differential-on-image-of-idempotent} can be transferred to modules over $p$-DG categories. Let $\e \in A_M$ be a subquotient idempotent. The induced module
$$
\AC\e\cong \AC\otimes_{\jmath_M,A_M}A_M\e \quad \left(\textrm{resp.}~\e \AC \cong \e A_M\otimes_{A_M,\jmath_M}\AC\right)
$$
is naturally a left (resp.~right) $p$-DG module over $\AC$. Furthermore, the endomorphism algebra of this module is identified with
\[
\END_{\AC}(\AC\e)\cong \e\cdot A_M\cdot\e \quad \left(\textrm{resp.}~\END_{\AC}(\e\AC)\cong \e\cdot A_M\cdot \e\right),
\]
which inherits the $p$-differential $\odif$ described by Lemma \ref{lemma-dif-endo-algebra-idempotent}.

Abusing notation, we can also write $\AC$ for the direct sum of all representable left modules, i.e. for $\oplus_{M \in \Obj(\AC)} \AC \1_M$, or for the direct sum of all representable
right modules.

More generally, we have the following construction for families of idempotents.

\begin{defn}\label{def-p-dg-partial-completion}Let $\AC$ be a $p$-DG category, and $\mathbb{X}$ a set of subquotient idempotents.
The \emph{partial completion $p$-DG category} $(\AC(\mathbb{X}),\overline{\dif})$ is the (pre-)additive category $\AC(\mathbb{X})$ (Definition \ref{def-proj-C-mod}), equipped with the differential $\overline{\dif}$ defined as follows. Given any morphism
\[
f\in \HOM_{\AC(\mathbb{X})}((M_i,\e_i),(M_j,\e_j)):= \e_j\cdot \HOM_{\AC}(M_i,M_j) \cdot \e_i,
\]
which is written as $f=\e_j\cdot g \cdot \e_i$ for some $g\in \HOM_{\AC}(M_i,M_j)$ and $\e_i,~\e_j$ either idempotents in $\mathbb{X}$ or identity morphisms in $\AC$, then $\overline{\dif}$ acts on $f$ by
\[
\overline{\dif}(f):=\e_j\cdot \dif(\e_j\cdot g\cdot \e_i) \e_i.
\]
\end{defn}

This definition is justified by Lemma \ref{lemma-dif-hom-idempotent}. Note that one may not be able to define the Karoubi envelope of a $p$-DG category, because not every idempotent is a
subquotient idempotent. However, there may be enough subquotient idempotents such that a partial idempotent completion $\AC(\XM)$ is Karoubian, and this (when it exists) will be our
approximation to the Karoubi envelope. For an example when there are not enough subquotient idempotents for any partial idempotent completion to be Karoubian, see Remark
\ref{source-of-counterexamples}.

The following result is an analogue of Lemma \ref{lemma-classical-Morita}.

\begin{lemma}\label{lemma-Morita-p-DG-abelian}
Let $\AC$ be a $p$-DG category, and $\mathbb{X}$ a set of subquotient idempotents in $\AC$. Then the abelian categories of left (resp. right) $p$-DG modules over $\AC$ and $\AC(\mathbb{X})$ are Morita equivalent.
\end{lemma}
\begin{proof}The $p$-DG bimodules $\AC\oplus (\oplus_{\e_i\in \mathbb{X}} \AC\cdot \e_i)$ and its dual module $\HOM_{\AC}(\AC\oplus (\oplus_{\e_i\in \mathbb{X}} \AC\cdot \e_i),\AC)$ give rise to tensor product functors, which are inverse of each other, between the left $p$-DG module categories. The proof is left to the reader.  \end{proof}
	
%

The above result is no surprise for the abelian category of $p$-DG modules. However, it is not the abelian category but the derived category of $p$-DG modules which controls the $p$-DG
Grothendieck group. We shall see soon that this Morita equivalence will descend to an equivalence between the corresponding homotopy categories of $p$-DG modules, but it may fail to descend
to a derived equivalence.

\subsection{Hopfological algebra for \texorpdfstring{$p$}{p}-DG categories}

This material is a straightforward generalization of the corresponding notions over a $p$-DG algebra. We refer the reader to \cite{QYHopf} for proofs and further details.

Given a $p$-DG category $\AC$, one can form its homotopy category $\HC(\AC)$ by taking the quotient of $\AC_\dif\dmod$ by the ideal of null-homotopic morphisms. Recall that a morphism $f$ between two $p$-DG modules $\NC_1$ and $\NC_2$ are called \emph{null-homotopic}, if there exists a morphism $h:\NC_1\to \NC_2$ such that
\[
f=\dif^{p-1}(h)=\sum_{i=0}^{p-1}\dif^{p-1-i}\circ h \circ \dif^i.
\]
The homotopy category $\HC(\AC)$ is triangulated.
Given a $p$-complex of vector spaces and an $\AC$-module, one can take their tensor product to obtain another $\AC$-module. This equips $\HC(\AC)$ with an action of the (monoidal) homotopy
category of $p$-complexes.

Let $\AC$, $\BC$ be two $p$-DG categories, and $_{\BC}\MC_{\AC}$ be a $p$-DG bimodule over them. It is clear that if $f=\dif^{p-1}(h)$ is a null-homotopic morphism between $\NC_1$ and $\NC_2$, then
\[
\Id_{\MC}\otimes f =\dif^{p-1}(\Id_{\MC}\o h): \MC\o_{\AC} \NC_1\lra \MC\o_{\AC} \NC_2
\]
is also null-homotopic. Therefore, the following result is a direct consequence of its abelian analogue Lemma \ref{lemma-Morita-p-DG-abelian}, by taking $\MC$ to be the $p$-DG bimodule $\AC\oplus(\oplus_{\e_i\in\mathbb{X}}\AC\e_i)$.

\begin{lemma}\label{lemma-Morita-homotopy-category}
Let $\AC$ be a $p$-DG category, and $\mathbb{X}$ a set of subquotient idempotents in $\AC$. Then the homotopy categories of left (resp. right) $p$-DG modules $\HC(\AC)$ and $\HC(\AC(\mathbb{X}))$ are Morita equivalent triangulated categories.
\hfill $\square$
\end{lemma}

The derived category $\DC(\AC)$ of a $p$-DG algebra is obtained from $\HC(\AC)$ by inverting quasi-isomorphisms. However, this definition does not make it easy to describe morphism spaces
in the derived category. It is better to use a description analogous to the description in ordinary homological algebra which uses complexes of projective modules.

 Let $\AC$ be a $p$-DG category. A $p$-DG module $\MC$ is \emph{cofibrant} if, given any surjective quasi-isomorphism $f:\NC_1\lra \NC_2$ of $p$-DG modules over $\AC$, the induced map
\[
\HOM_{\AC}(\MC,\NC_1)\lra \HOM_{\AC}(\MC,\NC_2)
\]
is a homotopy equivalence. A natural class of cofibrant $p$-DG modules are given by direct summands of $\AC$, i.e., representable modules. A $p$-DG module $\MC$ is said to satisfy \emph{property-P} if there is an exhaustive, possibly infinite, $p$-DG $\AC$-module filtration on $\MC$, whose subquotients are isomorphic to finite direct sums of representable modules. We say that $\MC$ is a \emph{finite-cell module}, if the module satisfies property-P with a finite-step filtration. An alternative construction of the derived category of $p$-DG modules over $\AC$ is realized by restricting, in the homotopy category $\HC(\AC)$, to the full subcategory formed by cofibrant $p$-DG modules.

A $p$-DG module $\MC$ over $\AC$, considered as an object in the derived category $\DC(\AC)$, is called \emph{compact}, if the functor $\Hom_{\DC(\AC)}(\MC,-)$ commutes with taking arbitrary direct sums. It is easily seen that finite-cell modules are compact. Furthermore, any direct summand, in the derived category, of a finite cell module is also compact. In fact, compact modules are characterized as objects that are isomorphic to derived direct summands of finite cell modules (\cite[Corollary 7.15]{QYHopf}). The strictly full subcategory $\DC^c(\AC)$ consisting of compact objects in $\DC(\AC)$ is called the \emph{compact derived category}, whose Grothendieck group will be denoted by $K_0(\AC)$. Computing this invariant of certain $p$-DG categories will be our primary concern in the next chapter.

%
\subsection{DG filtrations and Fc filtrations}
%

As per the discussion in Section \ref{subsec-basic-problem}, we now define certain filtrations on $p$-DG modules which help to compute their images in the Grothendieck group. These notions were originally defined in \cite[Section 5.1]{EQ1}.

\begin{defn} \label{defnDGfiltration}
Let $(\AC,\dif)$ be a $p$-DG category, and $M\in \Obj(\AC)$. A \emph{left $p$-DG filtration} on $M$ is a finite set of endomorphisms $\{\e_i\}_{i \in I} \in
\END^0_{\AC}(M)$, satisfying the following conditions.
\begin{itemize}
\item These maps give a decomposition of the $\AC$-module represented by $M$. In formulae,
\begin{subequations}
\begin{eqnarray}
\e_i\e_i &=& \e_i, \label{theyareidempts} \\
\e_i\e_j &=& 0~ \textrm{ for } i \ne j, \label{theyareorthogonalidempotents} \\
\1_M &=& \sum_{i\in I} \e_i. \label{theydecomposedontthey}
\end{eqnarray}
\end{subequations}
\item There exists some total order ``$<$'' on $I$ for which the above decomposition satisfies the following \emph{left triangularity condition}:
\begin{equation}
\label{conditionforDGfilt}
\dif(\e_i)\cdot\e_j = 0 ~\textrm{ for } j > i.
\end{equation}
\end{itemize}
\end{defn}

Notice that, upon differentiating \eqref{theyareorthogonalidempotents}, we have
$
\e_j\cdot\dif(\e_i)+\dif(\e_j)\cdot \e_i=0,
$
and the left triangularity condition is equivalent to the following \emph{right triangularity condition}
\begin{equation}
\label{rightconditionforDGfilt}
\e_i\cdot\dif(\e_j)= 0 ~\textrm{ for } i < j.
\end{equation}
Therefore, a left $p$-DG filtration on $M$ gives rise to a right $p$-DG filtration on $M$, consisting of the same data except with the opposite order $>$.

Let $A_M:=\END_{\AC}(M)$. By equation \eqref{theydecomposedontthey}, if $\e_1$ is the minimal idempotent in the total order, then
 \[
 \dif(\e_1) = \dif (\e_1)\cdot\1_M = \dif(\e_1)\cdot \e_1.
 \]
Hence $\dif (\e_1)\in A_M\cdot \e_1$. In similar fashion, $\sum_{j \le i} \e_j$ is also a left submodule idempotent for each $i \in I$, and therefore each $\e_i$ is a subquotient idempotent. It follows that a $p$-DG filtration as in Definition \ref{defnDGfiltration} on an object $M$ gives rise to an increasing filtration on the representable left $p$-DG module $\HOM_{\AC}(M, -)$ whose subquotients are isomorphic, as $p$-DG modules over $\AC$, to ones of the form $(\HOM_{\AC}((M, \e_i), -),\overline{\dif})$, for which the differentials $\overline{\dif}$ are defined by
\[
\overline{\dif}(   f \cdot \e_i ) := \dif(  f \cdot\e_i) \cdot\e_i.
\]
Likewise, using the opposite order, one also has a decreasing filtration on the representable right $p$-DG module $\HOM_{\AC}(-,M)$, whose subquotients are isomorphic to $\HOM_{\AC}(-, (M,\e_i))$, and equipped with the subquotient differential
$$\overline{\dif}( \e_i\cdot f ) := \e_i \cdot \dif(\e_i \cdot f ) .$$

However, these subquotient $p$-DG modules may not be well-behaved. For example, they need not be compact in $\DC(\AC)$. To ensure good behavior, one can impose the condition that the subquotients are isomorphic to representable modules.

\begin{defn} \label{defnFcfiltration}
Let $(\AC,\dif)$ be a $p$-DG category with an object $M$. A \emph{Fc-filtration} on $M$ is a finite set $\{N_i\}_{i \in I}$ of objects in $\AC$, equipped with maps $u_i \co N_i \lra M$ and $v_i \co M \lra N_i$, satisfying the following conditions.
\begin{itemize}
\item These maps give an additive decomposition of $M$ with summands $N_i$. In formulae,
\begin{subequations}
\begin{eqnarray} \label{factoringthroughN}
v_i u_i &=& \1_{N_i}, \label{theygive1N} \\
v_i u_j &=& 0 \textrm{ for } i \ne j, \label{theyareorthogonal} \\
\1_M &=& \sum_i u_i v_i. \label{theydecompose}
\end{eqnarray}
\end{subequations}
\item There exists some total order ``$>$'' on $I$ for which the above decomposition is a $p$-DG filtration. In formulae,
\begin{equation}
\label{conditionforfantastic}
\dif(v_i) u_j = 0 \textrm{ for } j \ge i.
\end{equation}
\end{itemize}
\end{defn}

Note the key difference between \eqref{conditionforDGfilt} and \eqref{conditionforfantastic}: the latter requires $j \ge i$, while the former only requires $j > i$. This is to ensure that the subquotient modules of this filtration on $M$ are isomorphic, as $p$-DG modules, to the modules represented by $N_i$'s with their natural differentials.

\begin{prop}
Given a Fc-filtration as above, the collection of idempotents $\{\e_i=u_iv_i|i\in I\}$ and the total ordering on $I$ defines a $p$-DG filtration on $M$, whose subquotients are $p$-DG isomorphic to the modules represented by $N_i$'s. \hfill $\square$
\end{prop}

The result tells us that, in the derived category $\DC(\AC)$, the module represented by $M$ is naturally a convolution product of those represented by $N_i$'s. Furthermore, this convolution relation takes place inside the compact derived category $\DC^c(\AC)$, and thereby results in a relation in the Grothendieck group $K_0(\AC)=K_0(\DC^c(\AC))$:
\begin{equation}\label{eqn-fc-filtration-K0-relation}
[M] = \sum_{i \in I} [N_i].
\end{equation}

A $p$-DG filtration does lead to a (boring) Fc-filtration in the partial Karoubi envelope; adjoining the subquotients as new objects essentially forces them to be cofibrant.

\begin{prop} \label{prop-DGtoFc} Let $(\AC,\dif)$ be a $p$-DG category, $M\in \Obj(\AC)$, and suppose that $M$ has a $p$-DG filtration with idempotents $\XM = \{\e_i\}$. Then, in the
partial idempotent completion $\AC(\XM)$, there is a Fc-filtration of $M$ whose subquotients are $(M,\e_i)$, in the same order.
\end{prop}

\begin{proof}
Exercise.
\end{proof}

However, Fc-filtrations within $\AC$ itself are more ``interesting," as they result in relations in the Grothendieck group between existing symbols of compact objects. For example, suppose that one has an
idempotent decomposition $\XM = \{\e_i = u_i v_i\}$ of $M$, factoring through objects $N_i$ as in \eqref{factoringthroughN}, and satisfying \eqref{conditionforDGfilt}. One desires
\eqref{eqn-fc-filtration-K0-relation} to hold, but without the additional condition \eqref{conditionforfantastic} it need not. In $\AC(\XM)$ one has a Fc-filtration of $M$ by the images of
$\e_i$, but these images need not be isomorphic to $N_i$ as $p$-DG modules; in fact, this is precisely condition \eqref{conditionforfantastic}.

\begin{example}\label{eg-fc-condtion-for-Es}
As an instance of Fc-filtrations, we revisit one example inside the $p$-DG category $\sym$ discussed in the previous Chapter. Here, we take $M$ to be the $p$-DG derived category $\DC(\sym_{a,b})$, and each $N_{\lambda}$ to be a copy of $\DC(\sym_{a+b})$, where the indices $\lambda \in P(a,b)$. Define functors
\[
u_{\lambda}:=
\begin{DGCpicture}
\DGCPLstrand[thick](1,0)(1,1)[$^{a+b}$]
\DGCPLstrand[thick](1,1)(0,2)[`$_a$]
\DGCPLstrand[thick](1,1)(2,2)[`$_b$]
\DGCcoupon(0.2,1.25)(0.8,1.75){$\pi_{\lambda}$}
\end{DGCpicture}~\quad \quad
v_{\lambda}:=
(-1)^{|\hat{\lambda}|}
\begin{DGCpicture}
\DGCPLstrand[thick](0,0)(1,1)[$^a$]
\DGCPLstrand[thick](2,0)(1,1)[$^b$]
\DGCPLstrand[thick](1,1)(1,2)[`$_{a+b}$]
\DGCcoupon(1.2,0.25)(1.8,0.75){$\pi_{\hat{\lambda}}$}
\end{DGCpicture}.
\]
Now, conditions \eqref{factoringthroughN} through \eqref{theydecompose} in Definition \ref{defnFcfiltration} are satisfied thanks to the diagrammatic relations \eqref{eq-duality} and \eqref{eq-identitydecomp}. Placing a partial order on $\l \in P(a,b)$ by the number of boxes, the smallest one being the empty diagram in $P(a,b)$, it is easy to see that this factored idempotent decomposition
upgrades to a Fc-filtration, i.e., equation \eqref{conditionforfantastic} is satisfied. The corresponding equalities of symbols in $K_0$ gives us the relation in $\dot{U}^+$
\[
E^{(a)}
E^{(b)} = {a+b \brack a}E^{(a+b)}.
\]
\end{example}

\begin{rmk} In future work, we will consider filtrations for which only some of the idempotents factor. It is straightforward to mix the definitions of a DG filtration and a Fc-filtration to provide a framework which will guarantee the desired relations on the Grothendieck group. For example, one will have a poset $I$, and for each $i \in I$, either a pair of morphisms $(u_i,v_i)$ factoring through an object $N_i$, or an idempotent endomorphism $e_i$. Conditions \eqref{conditionforDGfilt} and \eqref{conditionforfantastic} will hold when they make sense, as will
\begin{subequations}
\label{conditionformixed}
\begin{eqnarray}
\dif(e_i) u_j = 0 &\textrm{ for }& j > i. \\
\dif(v_i) e_j = 0 &\textrm{ for }& j > i.
\end{eqnarray}
\end{subequations}
\end{rmk}

%
\subsection{A key example}\label{sec-key-example}
%

The purpose of this chapter is to provide an illustrative example of a $p$-DG algebra which is Morita equivalent, but not $p$-DG Morita equivalent, to a Karoubian partial idempotent
completion. First we state a result which gives a necessary criterion for the homotopy Morita equivalence discussed previously to descend to a derived equivalence. The special case when $\AC$ is a $p$-DG algebra is proven in \cite[Proposition 8.8]{QYHopf}, and the general case is proven in a similar fashion.

\begin{prop}\label{prop-criteria-derived-equivalence}
Let $\AC$ be a $p$-DG category and $\mathbb{X}$ a set of subquotient idempotents in $\AC$. Suppose each $\AC\e_i$ is cofibrant as a left $p$-DG module over $\AC$. Set $\MC:=\AC\oplus (\oplus_{\e_i\in \mathbb{X}}\AC\e_i)$ to be $p$-DG bimodule over $(\AC,\AC(\mathbb{X}))$. Then the derived tensor functor
\[
\MC\otimes^{\mathbf{L}}_{\AC(\mathbb{X})}(-): \mc{D}(\AC(\mathbb{X}))\lra \mc{D}(\AC)
\]
is an equivalence of triangulated categories if and only if the canonical map of $p$-DG categories
\[
\AC\lra \END_{\AC(\mathbb{X})}(\MC)
\]
is a quasi-isomorphism of $p$-DG categories. \hfill$\square$
\end{prop}

The following example is studied in \cite[Section 2.3]{KQ}, and we use it to illustrate the general discussion of this chapter. Let $(U,\dif_U)$ be a finite dimensional $p$-complex, and let $A:=\END_{\Bbbk}(U)$ be the endomorphism algebra of $U$ equipped with the natural differential $\dif_A$: for any $f\in A$, $x\in U$,
\[
\dif_A(f)(x):=\dif_U(f(x))-f(\dif_U(x)).
\]
It is easy to check that $(A,\dif_A)$ is a $p$-DG algebra, which we regard as a $p$-DG category with a unique object $\ab$.

\begin{example} \label{eg-coordinates} For those who prefer coordinates, it is easiest to work with the following subexample. Let $U$ be the $p$-complex $\Bbbk[\dif]/(\dif^n)$ for $1 \le n
\le p$. Then $A = \Mat_n(\Bbbk)$, equipped with the doubled diagonal grading: if $\{e_{ij}\}$ is the standard basis, then $\deg e_{ij} = 2(j-i)$. Let $D = \sum_{i=1}^{n-1} e_{i(i+1)}$ be
the degree $2$ element with $1$ everywhere on the first off-diagonal. Then $\dif_A(x) = [D,x]$, and the fact that $n \le p$ forces $\dif_A^p = 0$. Observe that the idempotents $e_i =
e_{ii}$ form a DG-filtration, whose subquotients, the column modules, are all $p$-DG isomorphic. \end{example}

Let $\e \in A$ be an indecomposable left submodule idempotent\footnote{Such an idempotent always exists, since if $x\in U$ is part of a homogeneous basis of $U$ such that $\dif_U(x)=0$, then under the isomorphism $\END_\Bbbk(U)\cong U^\vee\otimes U$, $\e:=x^\vee\o x$ is left $\dif$-stable.}. We will take $\mathbb{X}:=\{\e\}$, and consider the category $\BC = \AC(\mathbb{X})$; we denote the new object $(\ab,\e)$ by $\vb$. One easily checks that $\BC$, as an algebroid, can be identified with the block-matrix
\begin{equation}\label{eqn-algebra-B}
B:=
\left(
\begin{matrix}
A & A\e\\
\e A & \e A \e
\end{matrix}
\right)
\cong
\left(
\begin{matrix}
A & U\\
U^\vee & \Bbbk
\end{matrix}
\right),
\end{equation}
where $U^\vee$ stands for the graded dual $p$-complex of $U$. In other words, in the new category $\BC$, we have
$$\HOM(\ab,\ab)= A, \quad \HOM(\vb,\ab) = U, \quad \HOM(\ab,\vb)=U^\vee, \quad \HOM(\vb,\vb)= \Bbbk.$$
The bimodule in Proposition \ref{prop-criteria-derived-equivalence} is equal to
\begin{equation}\label{eqn-bimod-B-to-A}
\MC ={_A
(A\oplus U)_B}.
\end{equation}
Notice that the canonical map
\[
A \lra \END_{B}(A\oplus U)\cong A
\]
is always an isomorphism of $p$-complexes. Proposition \ref{prop-criteria-derived-equivalence} tells us that, if $U$ is cofibrant as a $p$-DG module over $A$, then the derived tensor functor
\[
( A\oplus U)\otimes_{B}(-): \DC (B)\lra \DC(A)
\]
induces an equivalence of triangulated categories. The question now becomes, when is $U$ cofibrant as a $p$-DG $A$-module, and this is answered by the following.

\begin{lemma}\label{lemma-cofibrance-over-matrix}
Let $U$ be a finite dimensional $p$-complex, and let $A=\END_\Bbbk(U)$ be its graded endomorphism algebra equipped with the natural differential. Then $U$ is cofibrant as a $p$-DG module over $A$ if and only if $U$ is not contractible.
\end{lemma}
\begin{proof} If $U$ is not contractible, then $U$ contains a finite-dimensional, indecomposable $p$-complex $V$ as a direct summand that is not contractible. Consider the natural surjective map of $p$-DG modules
\[
A\o U \lra U, \quad (f,x)\mapsto f(x).
\]
The left $A$-module $A\o U\cong U\otimes U^\vee \otimes U$ contains as a direct summand $U\otimes V^\vee\o V$. Since $V$ is not acyclic, it has dimension between $1$ and $p-1$, and
\[
\mathrm{Tr}: V^\vee \otimes V\lra \Bbbk, \quad (x^\vee, x)\mapsto x^\vee(x)
\]
is a surjective map of $p$-complexes, with a section given by sending $1\in \Bbbk$ to the image of $\frac{1}{\mathrm{dim}(V)}\mathrm{Id}_V\in \END(V)\cong V^\vee \otimes V$. Hence $U$ is a direct summand, as a $p$-DG module over $A$, of $U\otimes V^\vee \otimes V$, and therefore of $U\otimes U^\vee\otimes U\cong A\o U$. The cofibrance of $U$ follows.

Conversely, if $U$ is contractible, then the natural map above $A\otimes U\lra U$ is a surjective quasi-isomorphism. It follows that
\[
\HOM_A(A\o U, U)\cong U^\vee \otimes U
\]
is a contractible $p$-complex. If $U$ were cofibrant over $A$, then $U$ is a direct summand of $A\otimes U$ as a $p$-DG module, and the inclusion
\[
\Bbbk\cong \HOM_A(U,U)\hookrightarrow \HOM_A(A\otimes U, U)
\]
would split as a map of $p$-complexes, and would be a contradiction.
\end{proof}

\begin{example}
Let us continue Example \ref{eg-coordinates}. The idempotent $e_n$ is an indecomposable left submodule idempotent, although we obtain identical results using any of the other idempotents $e_i$. The column module $U$ is acyclic if and only if $n = p$, so let us restrict to $n=p$. Then $A$ is contractible as a $p$-complex, and thus $\DC(A) = 0$. On the other hand, the inclusion of $\Bbbk$ into the lower right corner of $B$ is a quasi-isomorphism, as $A$, $U$, and $U^\vee$ are all contractible. Thus $\DC(B) \cong \DC(\Bbbk) \ne 0$.
Clearly then, $A$ and $B$ are not $p$-DG derived equivalent.
\end{example}

It is a common custom to abuse notation and identify a (representable) module over an additive category $\AC$ with a module over its Karoubi envelope $\Kar(\AC)$, even though these are two
distinct categories. This leads to confusion in the $p$-DG setting; let us reiterate what happens in the case of $\AC$ and $\BC$ above, when $U$ is acyclic.
\begin{itemize}
\item The
representable left module $\HOM_\AC(\ab,-)$ sends the unique object $\ab$ of $\AC$ to the $p$-complex $A$, which is contractible. Thus this representable module is acyclic.
\item The
non-representable summand $\HOM_\AC(\vb,-)$ sends the unique object $\ab$ of $\AC$ to $U$, which is contractible. Thus this module is also acyclic. \item The representable module
$\HOM_\BC(\ab,-)$ sends the object $\ab$ to $A$ and the object $\vb$ to $U^\vee$, both of which are contractible; thus this module is acyclic. \item However, the representable module
$\HOM_\BC(\vb,-)$ sends the object $\ab$ to $U$ (contractible) and the object $\vb$ to $\Bbbk$ (not contractible!). Thus this module is not acyclic.
\end{itemize}
For this reason it is important whether $\ab$ and $\vb$ indicate modules over $\AC$ or $\BC$. The representable module $\ab$ for $\AC$ remains acyclic when extended to $\BC$, but its acyclic summand $\vb$ is ``no longer acyclic" when extended to $\BC$.

Now let us examine the categories $\DC(A)$ and $\DC(B)$ and their Grothendieck groups. When $U$ is a \emph{non-contractible} $p$-complex, Lemma \ref{lemma-cofibrance-over-matrix} and Proposition \ref{prop-criteria-derived-equivalence} tells us that $\DC(A)$ and $\DC(B)$ are derived equivalent, giving rise to isomorphic invariants $K_0(A)\cong K_0(B)$.

However, when $U$ is \emph{contractible}, it is not a cofibrant module over $A$ by the previous Lemma. Nevertheless, Lemma~\ref{lemma-Morita-homotopy-category} tells us that the usual tensor product functor
\begin{equation}\label{eqn-homotopy-equivalence-A-B}
(A\oplus U)\otimes_{B}(-): \HC(B)\lra \HC(A)
\end{equation}
is a triangulated equivalence of homotopy categories. Yet, upon localization this fails to become a derived equivalence. This is because, on the one hand, $A$ is acyclic, and $\DC(A)\cong 0$. On the other hand, the map $\Bbbk \hookrightarrow B $, which sends $1$ to the identity map of $B$, is a quasi-isomorphism of $p$-DG algebras by \eqref{eqn-algebra-B}, and therefore the restriction functor induces a derived equivalence
\[
\DC(\Bbbk)\cong \DC(B).
\]
The same reasoning tells us that the tensor product with the block column bimodule
\[
\left(
\begin{matrix}
A\\
U^\vee
\end{matrix}
\right)\otimes_A(-): \HC(A)\lra \HC(B)
\]
is also a homotopy equivalence, and descends to a fully-faithful embedding of the zero category inside $\DC(B)$.

On the Grothendieck group level, we see two different behaviors upon taking a partial idempotent completion. When $U$ is non-contractible, the object added to $A$ is cofibrant. Adding such a
nice object to the category does not change the $p$-DG Morita-equivalence class. When $U$ is acyclic, the object added formally creates a new (nonzero) compact cofibrant object, and
enlarges the Grothendieck group!

%
\subsection{Computing the Grothendieck group of $\UC$ and $\dot{\UC}$}
%

Let us remind the reader how the $p$-DG Grothendieck group of $\UC$ was computed in \cite{EQ1}, before contrasting with our general arguments here.

In \cite{EQ1}, we had two collections of objects: a set $\XM$ of indecomposable compact cofibrant objects in $\UC\Proj$, and a collection $\YM$ of objects in $\UC$ whose Yoneda modules in $\UC\Proj$ were contractible. Then one constructed enough Fc-filtrations to demonstrate that the Yoneda module of every object was filtered by objects from the set $\XM \cup \YM$ (and their grading shifts). Therefore, $\XM$ is a generating class of $\UC$, and $\UC$ is $p$-DG Morita equivalent to $A = \END(\oplus_{X \in \XM} X)$. Using Lauda's computation of the endomorphism space \cite[Proposition 9.9]{Lau1}, we see that $A$ is a positively graded $p$-DG algebra, whose degree zero part consists only of scalar multiples of identity maps. Therefore, by \cite[Corollary 2.18]{EQ1}, the classes $\{[X] \}_{X \in \XM}$ will form an $\OM_p$-basis for $K_0(\Uthick)$.

To compare this situation to the matrix example of the previous section: when $1 \le n < p$ one would set $\XM = \{\vb\}$ and $\YM = \emptyset$; when $n = p$ one would set $\XM =
\emptyset$ and $\YM = \{\ab\}$.

In the event that a Karoubian partial idempotent completion exists, the computation of its Grothendieck group can be simplified. By adding summands of $Y \in \YM$ as new cofibrant objects,
we enlarge $\XM$ until $\YM$ disappears.

Let us state this in some generality. For a Karoubian category $\AC$, any indecomposable object in $\AC\Proj$ is representable, and therefore is compact and cofibrant. Moreover, every object $M$ has an idempotent decomposition $\1_M = \sum u_i v_i$ whose summands $N_i$ are indecomposable. However, this decomposition need not be a Fc-filtration.

Suppose that our $p$-DG category $\AC$ is Karoubian and mixed. Essentially, mixedness is the property that each indecomposable object has a unique grading shift which is self-dual (under
some duality functor on $\AC$), and that, letting $\XM$ denote the collection of these self-dual indecomposables, the algebra $A = \END(\oplus_{X \in \XM} X)$ is positively graded, with
only identity maps in degree zero. For instance, it was proven in \cite{KLMS} that $\Uthick \cong \Kar(\UC)$ is mixed.

\begin{defn} Let $\AC$ be a Karoubian, mixed $p$-DG category, with self-dual indecomposables $\XM$. Suppose that every $M \in \AC$ has a Fc-filtration as in Definition
\ref{defnFcfiltration}, whose summands $N_i$ are grading shifts of objects in $\XM$. We call $\AC$ \emph{fantastically filtered}, or \emph{Fc-filtered} for short.
\end{defn}

For the next result, denote by $K_0^\prime(\AC)$ the usual Grothendieck group of $\AC$ regarded as a graded (pre-)additive category, ignoring the $p$-differential. The graded structure on $\AC$ makes $K_0^\prime(\AC)$ into a $\Z[q,q^{-1}]$-module.

\begin{prop} \label{propKaroubianMixedGroth} If a Karoubian, mixed $p$-DG category $\AC$ is fantastically filtered, then it is $p$-DG Morita equivalent to the positive $p$-DG algebra
$\END(\oplus_{X \in \XM} X)$. There is an isomorphism
\[
K_0(\AC) \cong K_0^\prime(\AC) \ot_{\Z[q,q^{-1}]} \OM_p.
\]
\end{prop}
\begin{proof}
The proof is analogous to \cite[Corollary 2.18]{EQ1}, and we leave it to the reader as an exercise.
\end{proof}

Thus in a Karoubian mixed category, we need never deal with a class $\YM$ of acyclic objects, as we were forced to do in the non-Karoubian category $\UC$. These same objects $\YM
\subset \UC$ are still acyclic in $\Kar(\UC)$, but the freshly-added summands of a contractible object $Y \in \YM$ are ``no longer acyclic."

In the remainder of this paper, we:
\begin{itemize}
\item Construct a differential on the partial idempotent completion $\Uthick$, and show that it agrees with the canonically induced
differential $\odif$.
\item Prove that $\Uthick$ is fantastically filtered, by showing that the most important idempotent decomposition in \cite{KLMS}, the so-called Sto\v{s}i\'{c} formula, is actually a Fc-filtration.
\end{itemize}
Having accomplished this, our main Theorem \ref{thm-U-thick} is a corollary of Proposition \ref{propKaroubianMixedGroth}.

\section{Differential thick calculus}\label{sec-thick}

We now proceed to place a differential on the thick calculus $\dot{\UC}$ developed in \cite{KLMS}. The resulting $p$-DG category $(\Uthick, \dif)$ is essentially just the ``categorical double," a la Khovanov-Lauda-Rouquier, of the $p$-DG category $(\sym,\dif)$ constructed in \S\ref{sec-grassmannian}.  We will show in the last chapter that the relations in $\dot{U}$ are lifted to the $p$-DG categorical level.

%
\subsection{Generators and relations for the thick calculus}
%

In \cite{KLMS}, diagrammatics are given for morphisms between certain objects in the Karoubi envelope $\Kar(\UC)$. These objects are compositions of the objects $\EC$ and $\FC$ from
$\UC$, together with the divided powers $\EC^{(a)}$ and $\FC^{(a)}$. Explicitly, $\EC^{(a)}$ is the image of the following endomorphism of $\EC^a$, with notation borrowed from \cite[\S 2.2]{KLMS}.

\begin{equation} \label{eq_def_ea}
   \xy
 (0,0)*{\includegraphics[scale=0.5]{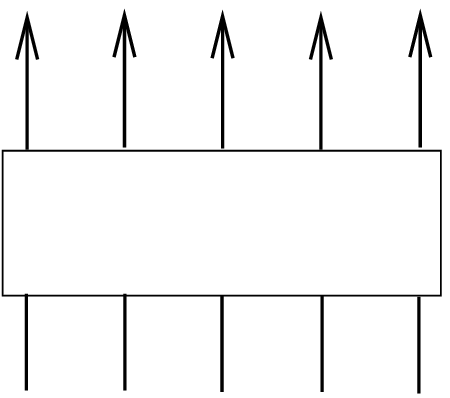}};
 (0,-1.5)*{\e_{a}};
  \endxy
\quad := \quad
  \xy
 (0,0)*{\includegraphics[scale=0.5]{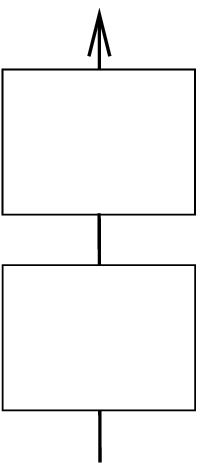}};
 (0,-5.5)*{D_a};(0,4.5)*{\delta_a};
  \endxy
  \quad = \quad
 \xy
 (0,0)*{\includegraphics[scale=0.5]{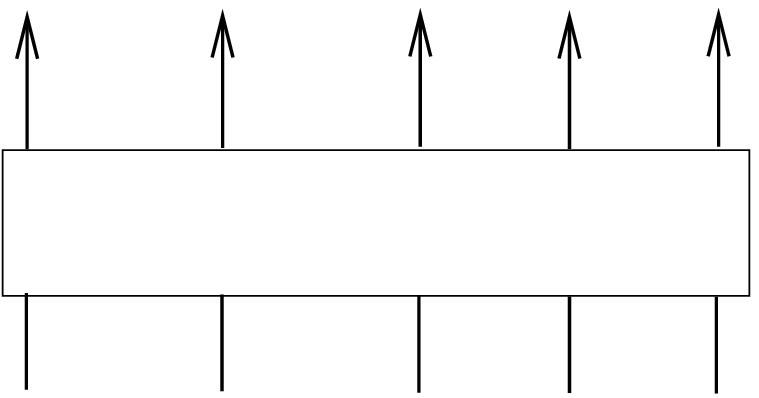}};
 (-17.7,5)*{\bullet}+(-4,1)*{\scs a-1};
 (-7.7,5)*{\bullet}+(-4,1)*{\scs a-2};
 (-3,4)*{\cdots};(-3,-7)*{\cdots};
 (2.3,5)*{\bullet}+(-2,1)*{\scs 2};
 (9.9,5)*{\bullet};(0,-1.5)*{D_a};
  \endxy
\end{equation}

Drawing the divided power $\EC^{(a)}$ as an upward-oriented line of thickness $a$, they provide a host of diagrams which can be thought of as new symbolic notation for certain morphisms
in $\Kar(\UC)$. For example (using citation numbers from their paper), one has boxes on strands labelled by symmetric polynomials (2.63), thick crossings (2.42), splitters (before
Proposition 2.4), thick cups and caps (page 38), and fake thick bubbles (4.33, 4.34). It is not hard to observe that all their new symbols can be expressed in terms of
morphisms in $\UC$ and the \emph{complete splitters} (2.69):

\begin{subequations}
\noindent\begin{tabularx}{\textwidth}{@{}XXX@{}}
 \label{completesplitters}
 \begin{equation} \label{completesplitterdown}
	{
	\labellist
	\small\hair 2pt
	 \pinlabel {$a$} [ ] at 15 40
	\endlabellist
	\centering
	\ig{1}{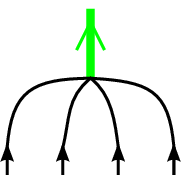}
	} \end{equation} &
	\begin{equation} \label{completesplitterup}
		{
		\labellist
		\small\hair 2pt
		 \pinlabel {$a$} [ ] at 15 11
		\endlabellist
		\centering
		\ig{1}{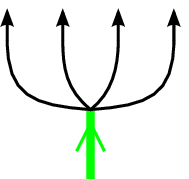}
		}. \end{equation}
\end{tabularx}
\end{subequations}
These complete splitters take a strand of thickness $a$, and split it into $a$ thin strands. Both complete splitters are given degree $-a(a-1)/2$.

Instead of thinking of their thick calculus as notation for specific morphisms in $\Kar(\UC)$, we prefer to think of it as the partial idempotent completion $\Uthick := \UC(\{ \e_a \})$ obtained by adjoining the divided powers (i.e., the images of the idempotents above). It turns out that every idempotent in $\Uthick$ will factor, implying that $\Uthick$ is Karoubian, or that $\Uthick \cong \Kar(\UC)$ as additive categories. This fact follows from the classification of indecomposable objects in $\Kar(\UC)$, given in \cite[\S 5.3]{KLMS}. Nonetheless, for reasons discussed in the previous chapter, it is important to distinguish between $\Uthick$ and the true Karoubi envelope $\Kar(\UC)$.

Given an arbitrary additive category $\AC$ with a presentation by generators and relations, it is simple to come up with a presentation of the partial idempotent completion $\AC(\XM)$ for any family of idempotents $\XM$. Namely, when adjoining the new object $(M,\e)$, one should add morphisms (inclusion and projection maps) $u \co (M,\e) \to M$ and $v \co M \to (M,\e)$ which satisfy
\begin{subequations}
\begin{eqnarray}
uv &=& \e, \\
vu &=& \1_{(M,\e)}.
\end{eqnarray}
\end{subequations}
It is clear that all morphisms involving the object $(M,\e)$ can be
obtained using morphisms involving $M$ and these new morphisms $u$ and $v$. Following this procedure, one obtains a presentation for $\Uthick$, which we state below in Proposition \ref{prop:Uthickpresentation}.

Note that the complete splitters \eqref{completesplitters} are not meant to be the inclusion and projection maps on the nose. For our chosen idempotent $\e_a$ above, the projection map $v$ is
\[
\ig{1}{projection},
\]
while the inclusion map $u$ is
\[	
{
\labellist
\small\hair 2pt
 \pinlabel {$\d_a$} [ ] at 27 41
\endlabellist
\centering
\ig{1}{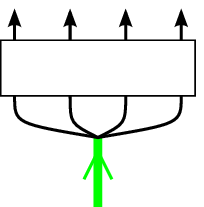}
},
\]
which can be expressed in terms of the complete splitter and morphisms in $\UC$. Said another way, the morphism in $\Uthick$ represented by \eqref{completesplitterdown} is $\e_a \in \Hom(\EC^a, (\EC^a,\e_a))$, while the morphism represented by \eqref{completesplitterup} is $D_a = D_a \e_a \in \Hom((\EC^a,\e_a),\EC^a)$.

\begin{rmk}
Complete splitters are more natural than inclusion and projection maps for any given idempotent. This is because $\EC^a \cong [a]! \EC^{(a)}$ splits into a number of isomorphic summands having different grading shifts. The projection map to the minimal degree summand is canonical, as is the inclusion map from the maximal degree summand; these are the complete splitters. The other inclusions and projections are non-canonical, and can be obtained by placing polynomials above and below the complete splitters.
\end{rmk}

\begin{prop} \label{prop:Uthickpresentation} $\Uthick$ has a presentation by generators and relations, extending that of $\UC$. One adds oriented thick lines labeled by $a > 0$ as new $1$-morphisms (with the convention that a thick line labeled $1$ is just our usual generator $\EC$ or $\FC$). The new $2$-morphism generators are the complete
splitters \eqref{completesplitters}. The new relations are:

\begin{subequations}
\noindent\begin{tabularx}{\textwidth}{@{}XXX@{}}
\begin{equation} 	{
	\labellist
	\small\hair 2pt
	 \pinlabel {$a$} [ ] at 18 39
	 \pinlabel {$D_a$} [ ] at 96 37
	\endlabellist
	\centering
	\ig{1}{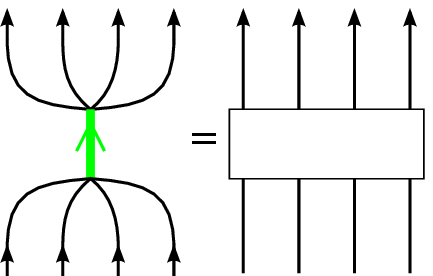}
	} \end{equation} &
\begin{equation}
	{
	\labellist
	\small\hair 2pt
	 \pinlabel {$\bullet$} [ ] at 1 37
	 \pinlabel {$\scs a-1$} [ ] at -10 41
	 \pinlabel {$\bullet$} [ ] at 21 37
	 \pinlabel {$\scs a-2$} [ ] at 12 41
	 \pinlabel {$\ldots$} [ ] at 36 37
	 \pinlabel {$\bullet$} [ ] at 52 37
	 \pinlabel {$\scs 1$} [ ] at 57 41
	\endlabellist
	\centering
	\ig{1}{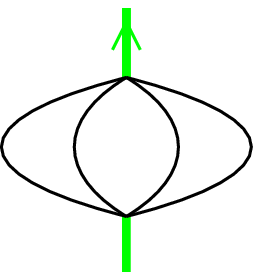}
	}
	\quad = \quad
	{
	\labellist
	\small\hair 2pt
	 \pinlabel {$a$} [ ] at -5 23
	\endlabellist
	\centering
	\ig{1}{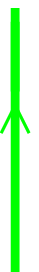}
	}
\end{equation}
\end{tabularx}
\end{subequations}
\end{prop}

\begin{proof} This is little more than a tautological unwinding of the above discussion and the definitions in \cite{KLMS}. \end{proof}

Now we can view all the other diagrammatics from \cite{KLMS} as being notation for specific morphisms in $\Uthick$.

%
\subsection{A differential on \texorpdfstring{$\Uthick$}{U dot}}
%

\begin{defn} We place a differential $\dif$ on $\Uthick$, extending the differential $\dif = \dif_1$ on $\UC$ from \cite[Definition 5.9]{EQ1}. It satisfies
\begin{subequations}
\begin{equation} \label{difsplitup}
\dif\left( \ig{1}{splitterup} \right) \quad =  - \quad 	{
	\labellist
	\small\hair 2pt
	 \pinlabel {$\ll_a$} [ ] at 27 41
	\endlabellist
	\centering
	\ig{1}{inclusion}
	} \end{equation}
\begin{equation} \label{difsplitdown}
\dif\left( \ig{1}{projection} \right) \quad =  - \quad 	{
	\labellist
	\small\hair 2pt
	 \pinlabel {$\lr_a$} [ ] at 28 20
	\endlabellist
	\centering
	\ig{1}{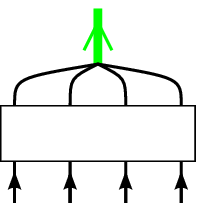}
	} \end{equation}
\end{subequations}
The notation $\ll, \lr$ stand for the linear functions defined in \eqref{lltri} and \eqref{lrtri}.
\end{defn}

Note the agreement with \eqref{eqn-d-full-flag} and \eqref{eqn-d-dull-full-flag}.

\begin{prop}
The differential defined by \eqref{difsplitup} and \eqref{difsplitdown} agrees with the canonical differential induced on a partial idempotent completion.
\end{prop}

\begin{proof} To check that \eqref{difsplitup} agrees with Lemma \ref{lemma-dif-endo-algebra-idempotent}, we need to compute $\dif(D_a) \e_a$. We have computed $\dif(D_a)$ in \eqref{eq-dif-of-Dn}. Thus
\[ \dif(D_a) \e_a = -	{
	\labellist
	\small\hair 2pt
	 \pinlabel {$\ll_a$} [ ] at 8 40
	 \pinlabel {$D_a$} [ ] at 8 17
	\endlabellist
	\centering
	\ig{1.5}{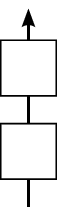}
	} - {
	\labellist
	\small\hair 2pt
	 \pinlabel {$D_a$} [ ] at 8 88
	 \pinlabel {$\lr_a$} [ ] at 8 65
	 \pinlabel {$\d_a$} [ ] at 8 40
	 \pinlabel {$D_a$} [ ] at 8 17
	\endlabellist
	\centering
	\ig{1.5}{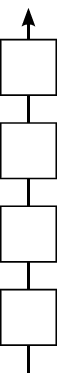}
	}. \]
The second term is zero. This is because each monomial in $\lr_a \d_a$ has two adjacent strands with the same exponent, and so is killed by some divided difference operator. The first term is the right hand side of \eqref{difsplitup}, as desired.

To check that \eqref{difsplitdown} agrees with Lemma \ref{lemma-dif-endo-algebra-idempotent} as well, we need to compute $\e_a \dif(\e_a)$. Again using \eqref{eq-dif-of-Dn}, and the fact that $\dif(\d_a)=\d_a\ll_a$, we see that
\[
\e_a \dif(\e_a)  = - {
	\labellist
	\small\hair 2pt
	 \pinlabel {$\e_a$} [ ] at 8 40
	 \pinlabel {$\lr_a$} [ ] at 8 17
	\endlabellist
	\centering
	\ig{1.5}{linewithtwoboxes}
	},\]
precisely as desired.
 \end{proof}

Because of this proposition, we can compute the differentials on many other diagrams, such as downward-oriented splitters, or thick cups and caps, using our knowledge of the differential
on $\UC$.

%
\subsection{Compatibility with symmetric polynomials and symmetric functions}
%

Recall from \cite{KLMS} that
\begin{equation}
\END(\EC^{(a)}) \cong \sym_a \boxtimes \Lambda,
\end{equation}
where the ring $\sym_a$ is realized by placing Schur polynomials on a thick
strand, as in \cite{KLMS} (2.72), and the ring $\Lambda$ of symmetric functions is realized by placing (clockwise or counterclockwise) thin bubbles alongside.

We now check that the induced differential on $\END(\EC^{(a)})$ agrees with the differential defined in \eqref{difpil} on $\sym_a$. In other words

\begin{equation} \label{eq-dif-schur} \dif \left( {
\labellist
\small\hair 2pt
 \pinlabel {$\pi_\a$} [ ] at 22 31
 \pinlabel {$a$} [ ] at 13 7
\endlabellist
\centering
\ig{1}{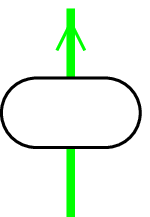}
} \right) = \sum_{\a + \square} C(\square) \; \; {
\labellist
\small\hair 2pt
 \pinlabel {$\pi_{\a + \square}$} [ ] at 22 31
 \pinlabel {$a$} [ ] at 13 7
\endlabellist
\centering
\ig{1}{linewithschur}
}.
\end{equation}

This can be formally deduced using properties of the nilHecke algebra, but it may be useful to see a diagrammatic proof.

\begin{proof} Recall the definition of a Schur polynomial on a thick strand from KLMS (2.72):
\begin{equation} \label{eq-defschur} 	{
	\labellist
	\small\hair 2pt
	 \pinlabel {$\pi_\a$} [ ] at 22 31
	 \pinlabel {$a$} [ ] at 13 7
	\endlabellist
	\centering
	\ig{1}{linewithschur}
	} \quad = \quad \quad {
	\labellist
	\small\hair 2pt
	 \pinlabel {$\bullet$} [ ] at 1 37
	 \pinlabel {$\scs \a_1 +$} [ ] at -10 41
	 \pinlabel {$\scs a-1$} [ ] at -10 35
	 \pinlabel {$\bullet$} [ ] at 21 37
	 \pinlabel {$\scs \a_2 +$} [ ] at 12 41
	 \pinlabel {$\scs a-2$} [ ] at 12 35
	 \pinlabel {$\ldots$} [ ] at 36 37
	 \pinlabel {$\bullet$} [ ] at 52 37
	 \pinlabel {$\scs \a_{a-1}$} [ ] at 60 41
	 \pinlabel {$\scs + 1$} [ ] at 60 35
	 \pinlabel {$\bullet$} [ ] at 72 37
	 \pinlabel {$\scs \a_{a}$} [ ] at 80 41	
	\endlabellist
	\centering
	\ig{1}{explode}
	}
\end{equation}

Consider what happens when one places an extra dot on a single strand in the RHS of \eqref{eq-defschur}. If one can add a $\square$ to the $i$-th row of $\a$ to obtain another partition
$\b$, then $\pi_\b$ will be realized by placing an extra dot on the $i$-th strand. If adding a $\square$ to the $i$-th row of $\a$ will not create a partition, then placing an extra dot
on the $i$-th strand will yield zero, because the exponents of the dots on the $i$-th and $(i-1)$-st strands become equal.

When we apply the differential to the RHS there will be three terms, each of which contributes a dot to the polynomial in the middle: $-\ll_a$ from one splitter, $-\lr_a$ from the other
splitter, and $\sum_i (\a_i + a - i) x_i$ from the polynomial. Because $-\ll_a - \lr_a = -(a-1) \sum_i x_i$, the overall contribution is $\sum_i (\a_i - i + 1) x_i$. Therefore, if $\b =
\a + \square$, then $\pi_\b$ appears with coefficient precisely $C(\square)$, as desired. \end{proof}

We have already checked in \cite[Corollary 4.8]{EQ1} that the action of $\dif$ on thin bubbles agrees with its canonical action on $\Lambda$. From this compatibility, we can immediately
derive the formula for the action of $\dif$ on thick bubbles.

\begin{subequations} \label{eq-dif-thick-bubbles}
\begin{eqnarray}
	\dif \left(
	{
	\labellist
	\small\hair 2pt
	 \pinlabel {$\pi_\a^\spadesuit$} [ ] at 15 17
	 \pinlabel {$a$} [ ] at 40 8
	\pinlabel {$n$} [ ] at 57 30
	\endlabellist
	\centering
	\ig{1.25}{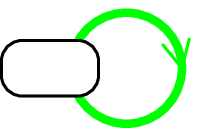}
	} \quad
	\right) & \quad = \quad \sum_{\a + \square} C(\square) \;&
	{
	\labellist
	\small\hair 2pt
	 \pinlabel {$\pi_{\a+\square}^\spadesuit$} [ ] at 15 17
	 \pinlabel {$a$} [ ] at 40 8
	\pinlabel {$n$} [ ] at 57 30
	\endlabellist
	\centering
	\ig{1.25}{thickbubbleCW}
	}
\\
\nonumber \\
\dif \left(
{
\labellist
\small\hair 2pt
 \pinlabel {$\pi_\a^\spadesuit$} [ ] at 15 17
 \pinlabel {$a$} [ ] at 40 8
	\pinlabel {$n$} [ ] at 57 30
\endlabellist
\centering
\ig{1.25}{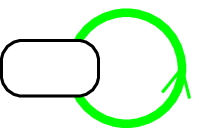}
} \quad
\right) & \quad = \quad \sum_{\a + \square} C(\square) \;&
{
\labellist
\small\hair 2pt
 \pinlabel {$\pi_{\a+\square}^\spadesuit$} [ ] at 15 17
 \pinlabel {$a$} [ ] at 40 8
	\pinlabel {$n$} [ ] at 57 30
\endlabellist
\centering
\ig{1.25}{thickbubbleCCW}
} \end{eqnarray}
\end{subequations}

We let $a$ denote the strand thickness, and $n$ the ambient weight. The ``spade" notation below is borrowed from \cite{KLMS}: see page 31 of \cite{KLMS} for the spade convention for thin
bubbles, and pages 43-44 for their generalization to thick bubbles. The symbol $\pi_\a^\spadesuit$ above represents a Schur polynomial for a partition depending on $\a$, $a$, $n$, and the
orientation of the bubble. Spade notation is useful because it allows formulas involving thick bubbles to be written in a form seemingly independent of $a$ and $n$.

Recall that, depending on a comparison between $\a$, $a$, $n$, and the orientation, these thick bubbles can be either fake or real. However, regardless of whether they are fake or real,
\cite[Theorem 4.9]{KLMS} allows one to define them using a Giambelli determinant formula in terms of the thin bubbles. Because this same Giambelli formula yields the Schur functions in
terms of elementary (resp. complete) symmetric functions, and because $\dif$ acts on thin bubbles as it does on elementary (resp. complete) symmetric functions, we can simply transfer
\eqref{difpil} to obtain \eqref{eq-dif-thick-bubbles}.

One can also compute \eqref{eq-dif-thick-bubbles} directly, using \eqref{eq-dif-schur} and the formulas for differentials of cups and caps, to be found in the following section.

%
\subsection{More formulas for the differential}
%

We now describe how the differential acts on various diagrams introduced in \cite{KLMS}. We will check some of these equalities, and leave the remainder as exercises.

First, we introduce some new notation. A ``thick dot" on a thick strand will indicate the symmetric polynomial $e_1$, which is $\pi_{\square}$.
\begin{equation} \label{def-thickdot}
	\ig{1}{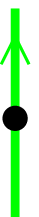} \quad := \quad
	{
		\labellist
		\small\hair 2pt
		 \pinlabel {$e_1$} [ ] at 22 31
		\endlabellist
		\centering
		\ig{1}{linewithschur}
		}
\end{equation}
When the thickness of a strand is $1$, the thick dot above is, of course, equal to the usual dot. Thanks to \cite{KLMS} (2.67), one has the following relation for any thick splitter.
\begin{equation} \label{thickdotslide} \ig{1}{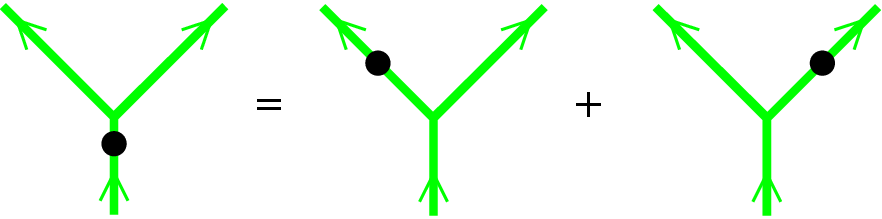} \end{equation}
The thick crossing is defined in KLMS (2.42); it merges two thick strands into a ``crossbar," and then splits them again. We allow ourselves to place dots on a thick crossing as well, by the following convention.
\begin{equation} \label{thickdotoncrossing} \ig{1}{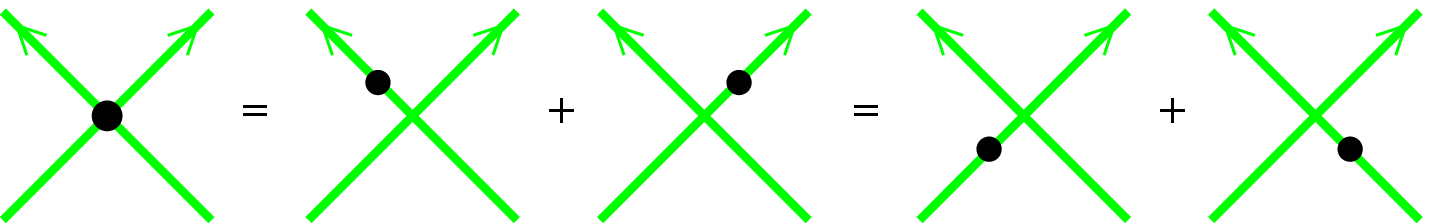} \end{equation}
In other words, a thick dot on a crossing is actually a thick dot on the ``crossbar."

Now we compute some differentials. The relations below hold independent of the ambient region labeling. We put scalars in parentheses to avoid confusion.

\begin{subequations} \label{eq-dif-thick-splitter}
\begin{eqnarray}
\label{eq-dif-thick-splitter-up} \dif \left(
{
\labellist
\small\hair 2pt
 \pinlabel {$a$} [ ] at 37 15
 \pinlabel {$k$} [ ] at 24 45
 \pinlabel {$a-k$} [ ] at 65 45
\endlabellist
\centering
\ig{1}{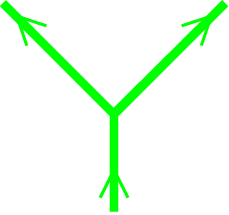}
} \quad
\right) & \quad = \quad &
(k-a)\;
{
\labellist
\small\hair 2pt
 \pinlabel {$a$} [ ] at 37 15
 \pinlabel {$k$} [ ] at 24 45
 \pinlabel {$a-k$} [ ] at 65 45
\endlabellist
\centering
\ig{1}{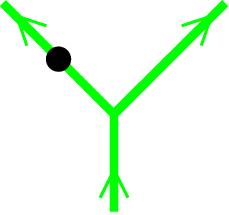}
}
\\
\nonumber \\
\label{eq-dif-thick-splitter-down} \dif \left(
{
\labellist
\small\hair 2pt
 \pinlabel {$a$} [ ] at 38 46
 \pinlabel {$k$} [ ] at 17 8
 \pinlabel {$a-k$} [ ] at 75 8
\endlabellist
\centering
\ig{1}{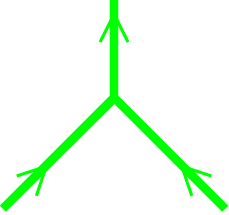}
} \quad \; \; \;
\right) & \quad = \quad &
(-k)\;
{
\labellist
\small\hair 2pt
 \pinlabel {$a$} [ ] at 38 46
 \pinlabel {$k$} [ ] at 17 8
 \pinlabel {$a-k$} [ ] at 75 8
\endlabellist
\centering
\ig{1}{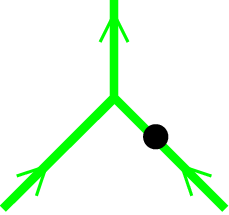}
}
\end{eqnarray}
\end{subequations}

\begin{proof} We check \eqref{eq-dif-thick-splitter-up}, leaving the similar check of \eqref{eq-dif-thick-splitter-down} to the reader. We can express the thick splitter by the formula
\begin{equation}
	{
	\labellist
	\small\hair 2pt
	 \pinlabel {$a$} [ ] at 37 15
	 \pinlabel {$k$} [ ] at 24 45
	 \pinlabel {$a-k$} [ ] at 65 45
	\endlabellist
	\centering
	\ig{1}{thicksplitterup}
	} \quad
	\quad = \quad
	{
	\labellist
	\small\hair 2pt
	 \pinlabel {$a$} [ ] at 49 5
	 \pinlabel {$k$} [ ] at 21 116
	 \pinlabel {$a-k$} [ ] at 75 116
	 \pinlabel {$\d_k$} [ ] at 18 61
	 \pinlabel {$\d_{a-k}$} [ ] at 66 61
	 \pinlabel {$\dots$} [ ] at 28 43
	 \pinlabel {$\dots$} [ ] at 28 78
	 \pinlabel {$\dots$} [ ] at 60 43
	 \pinlabel {$\dots$} [ ] at 60 78
	\endlabellist
	\centering
	\ig{1}{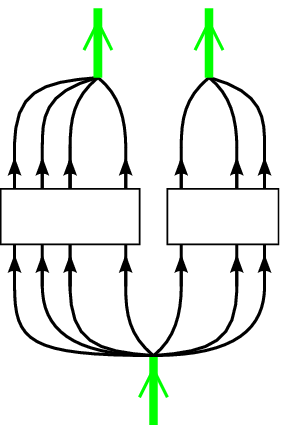}
	}.
\end{equation}
When applying $\dif$ to the right hand side (RHS), each term will consist of the same diagram with an extra dot on one of the central thin strands. The only extra dots which contribute a non-zero result
are a dot on the first strand or the $(k+1)$-st strand. One can easily compute that the coefficient of the extra dot on the first strand is $(i-a)$, which yields the RHS of
\eqref{eq-dif-thick-splitter-up}, while the coefficient of the extra dot on the $(k+1)$-st strand is zero.
\end{proof}

The following consequence is straightforward.

\begin{equation} \label{eq-dif-thick-crossing}  \dif \left(
	{
	\labellist
	\small\hair 2pt
	 \pinlabel {$a$} [ ] at 18 8
	 \pinlabel {$b$} [ ] at 65 8
	 \pinlabel {$b$} [ ] at 18 56
	 \pinlabel {$a$} [ ] at 65 56
	\endlabellist
	\centering
	\ig{1}{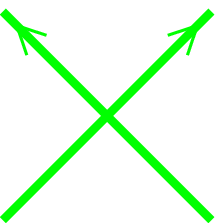}
	} \quad \right) \quad = (-a) \quad
	{
	\labellist
	\small\hair 2pt
	 \pinlabel {$a$} [ ] at 18 8
	 \pinlabel {$b$} [ ] at 65 8
	 \pinlabel {$b$} [ ] at 18 56
	 \pinlabel {$a$} [ ] at 65 56
	\endlabellist
	\centering
	\ig{1}{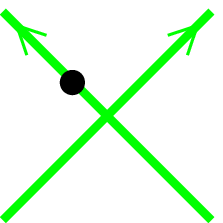}
	} \quad + (-a) \quad
	{
	\labellist
	\small\hair 2pt
	 \pinlabel {$a$} [ ] at 18 8
	 \pinlabel {$b$} [ ] at 65 8
	 \pinlabel {$b$} [ ] at 18 56
	 \pinlabel {$a$} [ ] at 65 56
	\endlabellist
	\centering
	\ig{1}{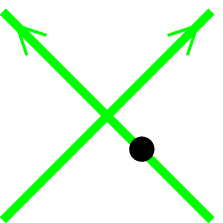}
	}
\end{equation}

Applying the symmetries of $(\UC,\dif)$, one can quickly compute the differentials of downward-oriented splitters and crossings.

\begin{subequations} \label{eq-dif-thick-splitter-D}
\begin{eqnarray}
\label{eq-dif-thick-splitter-up-D} \dif \left(
{
\labellist
\small\hair 2pt
 \pinlabel {$a$} [ ] at 37 15
 \pinlabel {$k$} [ ] at 24 45
 \pinlabel {$a-k$} [ ] at 65 45
\endlabellist
\centering
\ig{1}{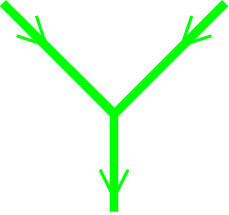}
} \quad
\right) & \quad = \quad &
(-k)\;
{
\labellist
\small\hair 2pt
 \pinlabel {$a$} [ ] at 37 15
 \pinlabel {$k$} [ ] at 24 45
 \pinlabel {$a-k$} [ ] at 65 45
\endlabellist
\centering
\ig{1}{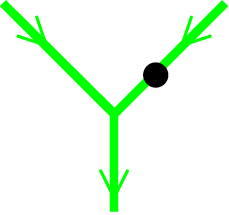}
}
\\
\nonumber \\
\label{eq-dif-thick-splitter-down-D} \dif \left(
{
\labellist
\small\hair 2pt
 \pinlabel {$a$} [ ] at 38 46
 \pinlabel {$k$} [ ] at 17 8
 \pinlabel {$a-k$} [ ] at 75 8
\endlabellist
\centering
\ig{1}{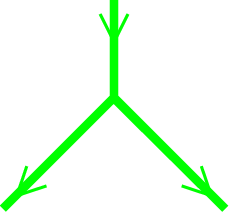}
} \quad \; \; \;
\right) & \quad = \quad &
(k-a)\;
{
\labellist
\small\hair 2pt
 \pinlabel {$a$} [ ] at 38 46
 \pinlabel {$k$} [ ] at 17 8
 \pinlabel {$a-k$} [ ] at 75 8
\endlabellist
\centering
\ig{1}{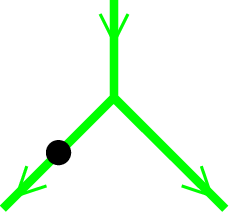}
}
\end{eqnarray}
\end{subequations}

\begin{equation} \label{eq-dif-thick-crossing-D}  \dif \left(
	{
	\labellist
	\small\hair 2pt
	 \pinlabel {$a$} [ ] at 18 8
	 \pinlabel {$b$} [ ] at 65 8
	 \pinlabel {$b$} [ ] at 18 56
	 \pinlabel {$a$} [ ] at 65 56
	\endlabellist
	\centering
	\ig{1}{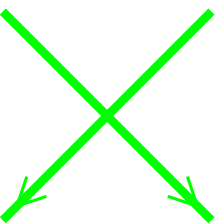}
	} \quad \right) \quad = (-b) \quad
	{
	\labellist
	\small\hair 2pt
	 \pinlabel {$a$} [ ] at 18 8
	 \pinlabel {$b$} [ ] at 65 8
	 \pinlabel {$b$} [ ] at 18 56
	 \pinlabel {$a$} [ ] at 65 56
	\endlabellist
	\centering
	\ig{1}{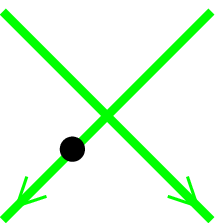}
	} \quad + (-b) \quad
	{
	\labellist
	\small\hair 2pt
	 \pinlabel {$a$} [ ] at 18 8
	 \pinlabel {$b$} [ ] at 65 8
	 \pinlabel {$b$} [ ] at 18 56
	 \pinlabel {$a$} [ ] at 65 56
	\endlabellist
	\centering
	\ig{1}{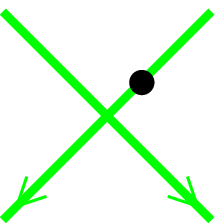}
	}
\end{equation}

Now we compute the differentials of thick cups and caps, which do depend on the ambient region labeling. Here, $n$ is a label on a region, while $a$ is the thickness of a strand. The
notation for thin bubbles is preserved from \cite[Section 4.1]{EQ1}.

\begin{subequations} \label{eq-dif-thick-cupcap}
\begin{eqnarray}
\label{eq-dif-thick-CCW-cap} \dif \left(
{
\labellist
\small\hair 2pt
 \pinlabel {$a$} [ ] at 52 10
 \pinlabel {$n$} [ ] at 65 34
\endlabellist
\centering
\ig{1}{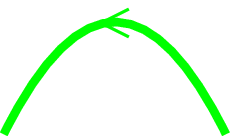}
} \right)
&\quad=\quad&
(n+a)\;
{
\labellist
\small\hair 2pt
 \pinlabel {$a$} [ ] at 52 10
 \pinlabel {$n$} [ ] at 65 34
\endlabellist
\centering
\ig{1}{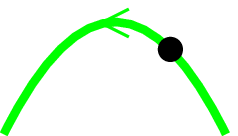}
} \\
\nonumber \\
\label{eq-dif-thick-CW-cup} \dif \left(
{
\labellist
\small\hair 2pt
\pinlabel {$a$} [ ] at 52 31
\pinlabel {$n$} [ ] at 65 9
\endlabellist
\centering
\ig{1}{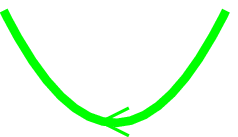}
} \right)
&\quad=\quad&
(a-n)\;
{
\labellist
\small\hair 2pt
\pinlabel {$a$} [ ] at 52 31
\pinlabel {$n$} [ ] at 65 9
\endlabellist
\centering
\ig{1}{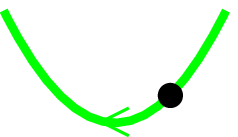}
} \\
\nonumber \\
\label{eq-dif-thick-CW-cap} \dif \left(
{
\labellist
\small\hair 2pt
 \pinlabel {$a$} [ ] at 52 10
 \pinlabel {$n$} [ ] at 65 34
\endlabellist
\centering
\ig{1}{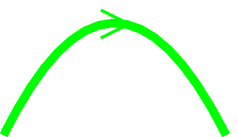}
} \right)
&\quad=\quad&
(a)\;
{
\labellist
\small\hair 2pt
 \pinlabel {$a$} [ ] at 52 10
 \pinlabel {$n$} [ ] at 65 34
\endlabellist
\centering
\ig{1}{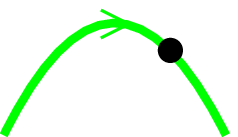}
}
\; - (a)\;
{
\labellist
\small\hair 2pt
 \pinlabel {$a$} [ ] at 52 10
 \pinlabel {$n$} [ ] at 65 34
\endlabellist
\centering
\ig{1}{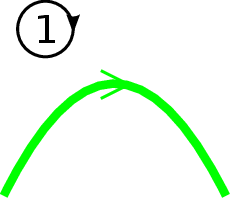}
}
\\
\nonumber \\
\label{eq-dif-thick-CCW-cup} \dif \left(
{
\labellist
\small\hair 2pt
\pinlabel {$a$} [ ] at 52 31
\pinlabel {$n$} [ ] at 65 9
\endlabellist
\centering
\ig{1}{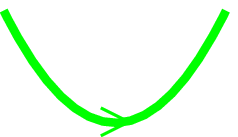}
} \right)
&\quad=\quad&
(a)\;
{
\labellist
\small\hair 2pt
\pinlabel {$a$} [ ] at 52 31
\pinlabel {$n$} [ ] at 65 9
\endlabellist
\centering
\ig{1}{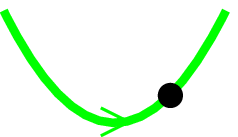}
}
\; + (a)\;
{
\labellist
\small\hair 2pt
\pinlabel {$a$} [ ] at 52 41
\pinlabel {$n$} [ ] at 65 19
\endlabellist
\centering
\ig{1}{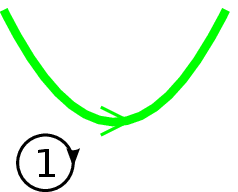}
}
\end{eqnarray}
\end{subequations}

\begin{proof} That these formulas all hold when $a=1$ was shown in \cite{EQ1}. Let us check \eqref{eq-dif-thick-CCW-cap}. Unraveling the definition of the thick cap, one has
\[
{
\labellist
\small\hair 2pt
 \pinlabel {$=$} [ ] at 79 32
 \pinlabel {$\d_a$} [ ] at 196 38
 \pinlabel {$a$} [ ] at 54 5
 \pinlabel {$n$} [ ] at 50 54
 \pinlabel {$a$} [ ] at 204 8
 \pinlabel {$n$} [ ] at 219 89
\endlabellist
\centering
\ig{1}{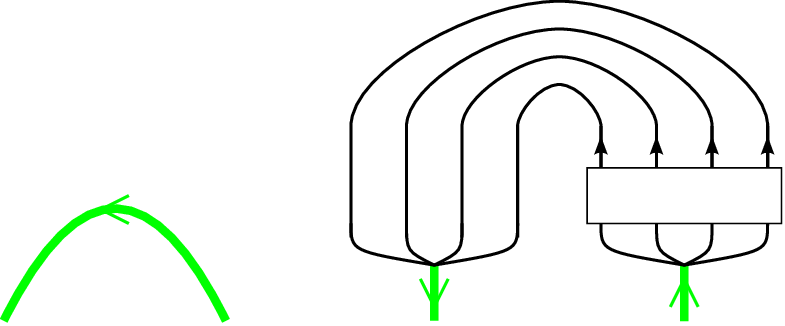}
}.
\]
Thus its differential will have a number of terms: $\ll_a$ coming from the upward splitter, $\lr_a$ coming from the downward splitter, $\ll_a$ coming from $\dif(\d_a)$, and a dot on each thin strand coming from the $a=1$ case of \eqref{eq-dif-thick-CCW-cap} itself. Each dot will have a different coefficient, because each thin strand lives in a different ambient weight space. Putting it all together, we have
\[
\begin{array}{l}
\dif \left(
{
\labellist
\small\hair 2pt
 \pinlabel {$a$} [ ] at 52 10
 \pinlabel {$n$} [ ] at 65 34
\endlabellist
\centering
\ig{1}{thickCCWcap}
} \right)
\;  = \;
(-1) {
\labellist
\small\hair 2pt
 \pinlabel {$\d_a$} [ ] at 60 85
 \pinlabel {$\ll_a$} [ ] at 96 37
\endlabellist
\centering
\ig{1}{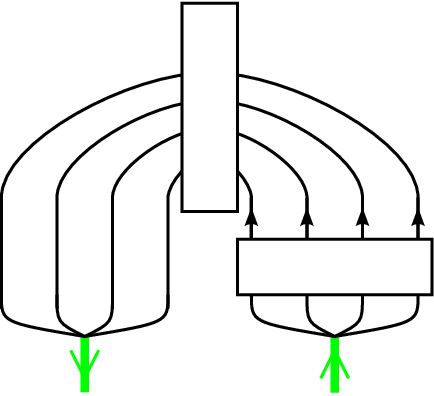}
} \; + \;
(-1) {
\labellist
\small\hair 2pt
 \pinlabel {$\d_a$} [ ] at 64 85
 \pinlabel {$\lr_a$} [ ] at 28 37
\endlabellist
\centering
\ig{1}{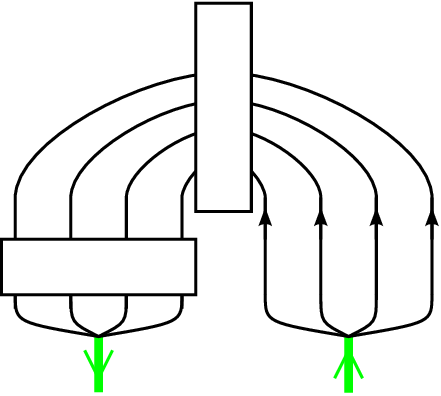}
} \;\\
\hspace{0.6in} +
{
\labellist
\small\hair 2pt
 \pinlabel {$\d_a$} [ ] at 60 85
 \pinlabel {$\ll_a$} [ ] at 96 37
\endlabellist
\centering
\ig{1}{thickCCWcapproof1}
} \; + \;
(n + 1) {
\labellist
\small\hair 2pt
 \pinlabel {$\d_a$} [ ] at 60 85
\endlabellist
\centering
\ig{1}{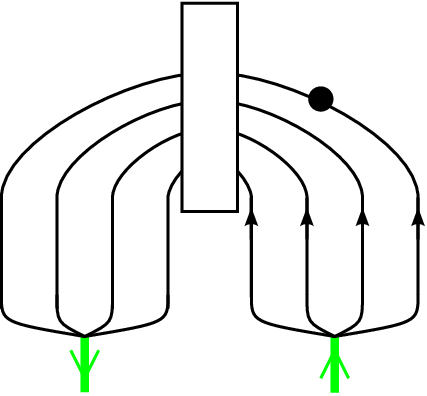}
} \;
\\
\hspace{0.6in}+(n+3) {
\labellist
\small\hair 2pt
 \pinlabel {$\d_a$} [ ] at 60 85
\endlabellist
\centering
\ig{1}{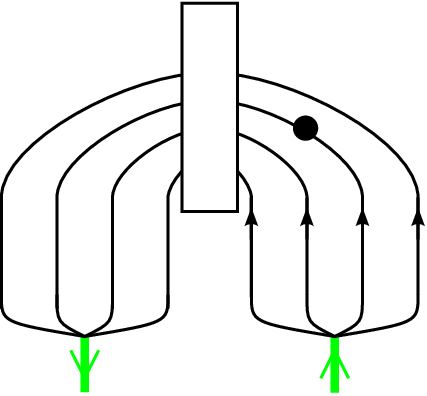}
}
  +  \dots +
(n+2a-1){
\labellist
\small\hair 2pt
 \pinlabel {$\d_a$} [ ] at 60 85
\endlabellist
\centering
\ig{1}{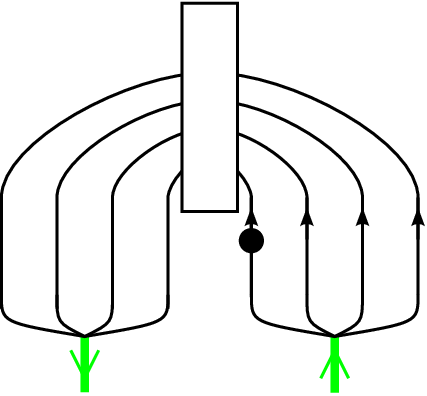}
}.
\end{array}
\]
The two instances of $\ll_a$ cancel out. Moreover, several of the diagrams above represent the zero morphism. Because $\d_a x_i$ is symmetric in $x_i$ and $x_{i-1}$ whenever $i>1$, placing a dot on one of the thin strands will yield the zero morphism, unless that strand is innermost. For example, two of the diagrams on the second row will vanish, while the diagram on the last row will survive. The coefficient of the surviving dot is $-(a-1)$ from $-\lr_a$ on the first row, and $n+2a-1$ on the third row. Thus we obtain
\[
\dif \left(
{
\labellist
\small\hair 2pt
 \pinlabel {$a$} [ ] at 52 10
 \pinlabel {$n$} [ ] at 65 34
\endlabellist
\centering
\ig{1}{thickCCWcap}
} \right)
\;=\;
(n+a){
\labellist
\small\hair 2pt
 \pinlabel {$\d_a$} [ ] at 60 85
\endlabellist
\centering
\ig{1}{thickCCWcapproof5}
} \; = \;
(n+a)\;
{
\labellist
\small\hair 2pt
 \pinlabel {$a$} [ ] at 52 10
 \pinlabel {$n$} [ ] at 65 34
\endlabellist
\centering
\ig{1}{thickCCWcapwdot}
} \]
as desired.

The proofs of the other formulas are similar in nature, and we leave them to the reader. For instance, when checking \eqref{eq-dif-thick-CW-cap}, one rewrites the thick cap using splitters
and thin caps, and then applies the $a=1$ case of \eqref{eq-dif-thick-CW-cap}. An additional complication is that bubbles appear in the intermediate regions between thin strands, and must
be slid to the outer region using bubble slide relations. Nonetheless, because almost all dots contribute the zero morphism, it is not hard to compute the final result. \end{proof}

Now we can compute the differential of a sideways thick crossing, which does depend on the ambient region label.

\begin{subequations}
\begin{eqnarray}
\label{eq-dif-left-crossing} \dif \left(
	{
	\labellist
	\small\hair 2pt
	 \pinlabel {$a$} [ ] at 18 8
	 \pinlabel {$b$} [ ] at 65 8
	 \pinlabel {$b$} [ ] at 18 56
	 \pinlabel {$a$} [ ] at 65 56
	 \pinlabel {$n$} [ ] at 50 32
	\endlabellist
	\centering
	\ig{1}{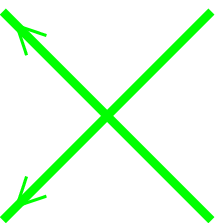}
	} \quad \right) & \hspace{-0.15in}= (a-b-n)& \hspace{-0.15in}
	{
	\labellist
	\small\hair 2pt
	 \pinlabel {$a$} [ ] at 18 8
	 \pinlabel {$b$} [ ] at 65 8
	 \pinlabel {$b$} [ ] at 18 56
	 \pinlabel {$a$} [ ] at 65 56
	\endlabellist
	\centering
	\ig{1}{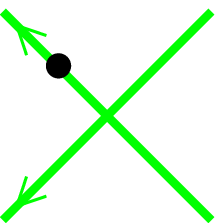}
	}  + (n+b-a)
	{
	\labellist
	\small\hair 2pt
	 \pinlabel {$a$} [ ] at 18 8
	 \pinlabel {$b$} [ ] at 65 8
	 \pinlabel {$b$} [ ] at 18 56
	 \pinlabel {$a$} [ ] at 65 56
	\endlabellist
	\centering
	\ig{1}{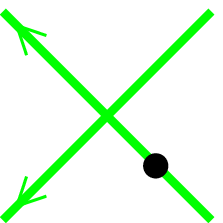}
	}
\\ 	& \hspace{-0.15in} = (a-b-n) & \hspace{-0.15in}
	{
	\labellist
	\small\hair 2pt
	 \pinlabel {$a$} [ ] at 18 8
	 \pinlabel {$b$} [ ] at 65 8
	 \pinlabel {$b$} [ ] at 18 56
	 \pinlabel {$a$} [ ] at 65 56
	\endlabellist
	\centering
	\ig{1}{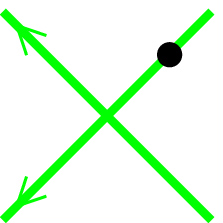}
	}  + (n+b-a)
	{
	\labellist
	\small\hair 2pt
	 \pinlabel {$a$} [ ] at 18 8
	 \pinlabel {$b$} [ ] at 65 8
	 \pinlabel {$b$} [ ] at 18 56
	 \pinlabel {$a$} [ ] at 65 56
	\endlabellist
	\centering
	\ig{1}{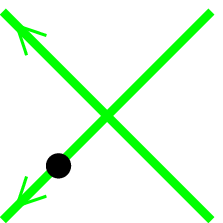}
	}  \nonumber \\
\nonumber \\
\label{eq-dif-right-crossing} \dif \left(
	{
	\labellist
	\small\hair 2pt
	 \pinlabel {$a$} [ ] at 18 8
	 \pinlabel {$b$} [ ] at 65 8
	 \pinlabel {$b$} [ ] at 18 56
	 \pinlabel {$a$} [ ] at 65 56
	 \pinlabel {$n$} [ ] at 50 32
	\endlabellist
	\centering
	\ig{1}{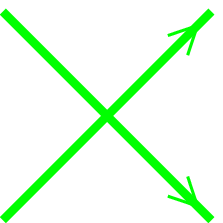}
	} \quad \right) & = & 0
\end{eqnarray}
\end{subequations}

\begin{proof} We check \eqref{eq-dif-left-crossing}. The equality of the two terms on the RHS is easy from \eqref{thickdotoncrossing}.  We show that the differential is equal to the second line on the RHS.
	
First, we express the sideways crossing in terms of an upwards crossing, with cups and caps.
\[ {
\labellist
\small\hair 2pt
 \pinlabel {$a$} [ ] at 18 8
 \pinlabel {$b$} [ ] at 65 8
 \pinlabel {$b$} [ ] at 18 56
 \pinlabel {$a$} [ ] at 65 56
 \pinlabel {$n$} [ ] at 50 32
\endlabellist
\centering
\ig{1}{thickcrossingleft}
} \quad = \quad {
\labellist
\small\hair 2pt
 \pinlabel {$b$} [ ] at 86 111
 \pinlabel {$a$} [ ] at 119 111
 \pinlabel {$a$} [ ] at 7 7
 \pinlabel {$b$} [ ] at 38 7
 \pinlabel {$n$} [ ] at 85 11
 \pinlabel {$n+2b-2a$} [ ] at 40 111
\endlabellist
\centering
\ig{1}{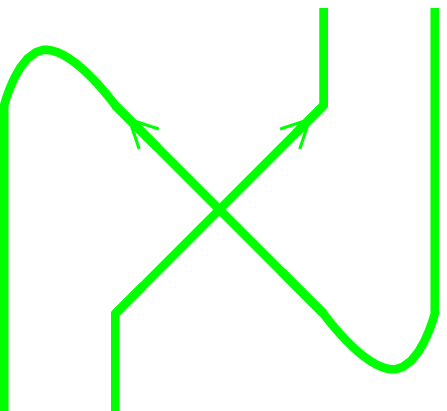}
} \]
Now we apply $\dif$, using \eqref{eq-dif-thick-CCW-cap}, \eqref{eq-dif-thick-CW-cup}, and \eqref{eq-dif-thick-crossing}. All terms add a dot to the strand of thickness $a$. For the dot
which meets the lower left output, one has coefficient $-b$ from the crossing and $(a+n+2b-2a)$ from the cap, yielding $n+b-a$ overall. For the dot which meets the upper right input, one
has coefficient $+b$ from the crossing and $(a-n)$ from the cup, yielding $(a-b-n)$ overall.

The check of \eqref{eq-dif-right-crossing} is slightly trickier, involving bubble slides, and we leave it as an exercise. \end{proof}

Let us rewrite \eqref{eq-dif-left-crossing} in the form we will use below.

\begin{align} \label{eq-dif-left-crossing-alt} \dif \left(
	{
	\labellist
	\small\hair 2pt
	 \pinlabel {$a$} [ ] at 18 8
	 \pinlabel {$b$} [ ] at 65 8
	 \pinlabel {$b$} [ ] at 18 56
	 \pinlabel {$a$} [ ] at 65 56
	 \pinlabel {$n$} [ ] at 50 32
	\endlabellist
	\centering
	\ig{1}{thickcrossingleft}
	} \quad \right) & = (a-b-n) {
	\labellist
	\small\hair 2pt
	 \pinlabel {$a$} [ ] at 18 8
	 \pinlabel {$b$} [ ] at 65 8
	 \pinlabel {$b$} [ ] at 18 56
	 \pinlabel {$a$} [ ] at 65 56
	 \pinlabel {$n$} [ ] at 50 32
	\endlabellist
	\centering
	\ig{1}{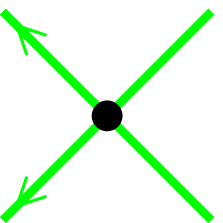}
	} \nonumber \\
 & - (a-b-n) \left( \quad {
	\labellist
	\small\hair 2pt
	 \pinlabel {$a$} [ ] at 18 8
	 \pinlabel {$b$} [ ] at 65 8
	 \pinlabel {$b$} [ ] at 18 56
	 \pinlabel {$a$} [ ] at 65 56
	\endlabellist
	\centering
	\ig{1}{thickcrossingleftdotDL}
	} \quad + \quad {
	\labellist
	\small\hair 2pt
	 \pinlabel {$a$} [ ] at 18 8
	 \pinlabel {$b$} [ ] at 65 8
	 \pinlabel {$b$} [ ] at 18 56
	 \pinlabel {$a$} [ ] at 65 56
	\endlabellist
	\centering
	\ig{1}{thickcrossingleftdotDR}
	} \quad \right).
\end{align}
	
\section{A fantastic filtration from the Sto\v{s}i\'{c} formula}\label{sec-categorification}

\subsection{The main theorem}
Now we define the BLM form $\dot{U}$ of quantum $\mathfrak{sl}_2$ specialized over $\mathbb{O}_p$, which is obtained from the its integral generic form $\dot{U}_{\Z[q^{\pm}]}$ via base change:
\[
\dot{U}:=\dot{U}_{\Z[q^{\pm}]}\otimes_{\Z[q^{\pm}]}\mathbb{O}_p.
\]
It is best described as an idempotented algebra (i.e., a category) as follows.

\begin{defn}
\begin{enumerate}
\item The $\mathfrak{sl}_2$-BLM quantum algebra $\dot{U}$ is an $\mathbb{O}_p$-linear category which has as objects integers $n\in \Z$. We will also write $1_n$ for the identity morphism of $n$. The generating morphisms are given by\footnote{We will just write $E^{(a)}1_{n}$ or $F^{(a)}1_{n}$ for short.}
\[
1_{n+2a}E^{(a)}1_{n}, \quad 1_{n-2a}F^{(a)}1_{n},
\]
for each $a\in \N$ and $n\in \N$. They satisfy the following relations
\begin{subequations}
\begin{eqnarray}
E^{(a)}E^{(b)}1_n & = & {a+b \brack a}E^{(a+b)}1_n, \label{eqn-EE}\\
F^{(a)}F^{(b)}1_n & = & {a+b \brack a}F^{(a+b)}1_n, \label{eqn-FF}\\
E^{(a)}F^{(b)}1_n & = & \sum_{j=0}^{\min (a,b)}{a-b+n \brack j}F^{(b-j)}E^{(a-j)}1_n,\label{eqn-EF}\\
F^{(a)}E^{(b)}1_n & = & \sum_{j=0}^{\min (a,b)}{a-b+n \brack j}E^{(b-j)}F^{(a-j)}1_n. \label{eqn-FE}
\end{eqnarray}
\end{subequations}
\item The small quantum group $\dot{u}$ is an $\mathbb{O}_p$-linear subcategory (or sub-algebroid) of $\dot{U}$, with the same objects, while generating morphisms are taken to be
    \[
    1_{n+2}E1_n, \quad 1_{n-2}F1_n,
    \]
    for each $n\in \Z$.
\end{enumerate}
\end{defn}
By specializing Lusztig's canonical basis for $\dot{U}_{\Z[q^{\pm}]}$ over $\mathbb{O}_p$, one obtains an $\mathbb{O}_p$-basis for $\dot{U}$, which will be denoted
\begin{equation}
\dot{\mathbb{B}}=\{E^{(a)}F^{(b)}1_n,~F^{(b)}E^{(a)}1_m|a,b\in \N, m,n\in \Z, n\leq b-a, m\geq b-a\}.
\end{equation}
The corresponding $\mathbb{O}_p$-basis for $\dot{u}$ is given by
\begin{equation}
\dot{\mathbb{B}}_{\dot{u}}=\{E^{(a)}F^{(b)}1_n,~F^{(b)}E^{(a)}1_m|0\leq a,b\leq p-1, m,n\in \Z, n\leq b-a, m\geq b-a\}.
\end{equation}

Now we are ready to state our main result of this chapter. Let $(\Uthick,\dif)$ be the $p$-DG category defined in \S\ref{sec-thick} (see equations \eqref{eq-dif-thick-splitter} through \eqref{eq-dif-thick-cupcap} for the differential).

\begin{thm} \label{thm-U-thick}
The $p$-DG Grothendieck group of the derived category $\DC(\Uthick,\dif)$ is isomorphic to the $\mathbb{O}_p$-alegbra $\dot{U}$. Furthermore, in $K_0(\Uthick)$, the symbols for the representable modules in the set
\begin{equation} \label{canonicalobjects}
\{\EC^{(a)}\FC^{(b)}\1_n,~\FC^{(b)}\EC^{(a)}\1_m|a,b\in \N, m,n\in \Z, n\leq b-a, m\geq b-a\}
\end{equation}
are identified with the images of Lusztig's canonical basis elements $\dot{\BM}$ in $\dot{U}$.
\end{thm}

\begin{proof}
Relations \eqref{eqn-EE} through \eqref{eqn-FE} were shown for the additive decategorification of $\Uthick$ in \cite{KLMS} by constructing, for each relation, an idempotent decomposition. Now, we check that each of these idempotent decompositions is actually a Fc-filtration. Earlier, we haven seen in Example \ref{eg-fc-condtion-for-Es} that the decomposition of $\EC^{(a)} \EC^{(b)}$ into copies of $\EC^{(a+b)}$ was a Fc-filtration, which implies that \eqref{eqn-EE} holds in the $p$-DG Grothendieck group. A similar computation (rotated and horizontally flipped) implies that \eqref{eqn-FF} holds. That the decompositions leading to \eqref{eqn-EF} and \eqref{eqn-FE} are Fc-filtrations is the content of Proposition \ref{hardfcfilt}.

Assuming this proposition, it was shown in \cite{KLMS} that these decompositions are sufficient to decompose any object in $\Uthick$ into a direct sum of representable modules of the form
\eqref{canonicalobjects}, which are themselves indecomposable. Therefore,these objects are generators of the
compact derived category. Their endomorphism rings are positively graded, i.e. $\Uthick$ is mixed, as was shown in \cite{KLMS}. Therefore, Proposition
\ref{propKaroubianMixedGroth} implies that the $p$-DG Grothendieck group of $\Uthick$ is the specialization of the ordinary Grothendieck group at a $p$-th root of unity, whence the desired
result. \end{proof}

\begin{prop} \label{hardfcfilt}
The representable $p$-DG modules $\EC^{(a)}\FC^{(b)}\1_n$ $\FC^{(b)}\EC^{(a)}\1_n$, where $a,b\in \N$ and $n\in \Z$, admit a Fc-filtration. More specifically, the direct sum decompositions of the Sto\v{s}i\'{c} formula \cite[Theorem 5.9]{KLMS} form this Fc-filtration.
\end{prop}

The subquotients of this filtration are all representable modules, categorifying the right hand side of the relations \eqref{eqn-EF} and \eqref{eqn-FE}.

The rest of this paper is our proof of Proposition \ref{hardfcfilt}. We follow Section 5.2 of \cite{KLMS} rather closely, and we will frequently reference their propositions and equation
numbers. To avoid conflict with \S\ref{sec-thick} of this paper, we will refer to their work as Section V.2, equation (V.27), and so forth. All citations of this form and all page numbers refer to \cite{KLMS}.

Finally, we should point out the relationship between the main result of \cite[Theorem 6.11]{EQ1} and Theorem \ref{thm-U-thick}. The following result is a consequence of Lemma \ref{lemma-classical-Morita} and Lemma \ref{lemma-Morita-homotopy-category}, as well as our construction of $({\Uthick},\dif)$ as a partial idempotent completion of $(\UC,\dif)$ (see Definition \cite[Definition]{EQ1}). The proof is in parallel with the ``one-half'' case of Corollary \ref{cor-categorical-embedding-half-u}, and will be left to the reader.

\begin{cor}
 The abelian and homotopy categories of $p$-DG modules over $({\Uthick},\dif)$ and $(\UC, \dif)$ are Morita equivalent. On the level of derived categories, $\DC(\UC, \dif)$ embeds fully faithfully inside $\DC(\Uthick,\dif)$. This embedding of derived categories categorifies the inclusion of the small quantum group $\dot{u}$ inside the BLM form $\dot{U}$ at a prime root of unity. \hfill$\square$
\end{cor}

%
\subsection{Preliminaries}
%

Let us quickly recall what Section V.2 accomplishes.

Theorem V.9 proves that there are isomorphisms in $\Uthick$: when $n \ge b-a$, one has
\begin{equation} \label{firstone}
   \EC^{(a)}\FC^{(b)}\1_n \cong \bigoplus_{j=0}^{\min(a,b)} \bigoplus_{\alpha \in P(j,n+a-b-j)}\FC^{(b-j)}\EC^{(a-j)}\1_n\{2|\alpha|-j(n+a-b)\},
\end{equation}
while for $n \le b-a$ one has
\begin{equation} \label{secondone}
 \FC^{(b)} \EC^{(a)}\1_n \cong \bigoplus_{j=0}^{\min(a,b)} \bigoplus_{\alpha \in P(j,-n+b-a-j)}\EC^{(a-j)}\FC^{(b-j)}\1_n\{2|\alpha|-j(b-a-n)\}.
\end{equation}
We focus on \eqref{firstone}, and will discuss \eqref{secondone} briefly at the end. The case $n = b-a$ is special, as both equations apply.

The inclusion and projection maps for the direct sum decomposition \eqref{firstone} are $\sigma_\a^j$ and $\l_\a^j$ respectively. We recall the formulas here.
\begin{eqnarray}
 \l_{\a}^i :=
{
\labellist
\small\hair 2pt
 \pinlabel {$\pi_\a$} [ ] at 31 38
 \pinlabel {$a$} [ ] at 0 56
 \pinlabel {$b$} [ ] at 61 56
 \pinlabel {$i$} [ ] at 44 49
 \pinlabel {$b-i$} [ ] at 1 4
 \pinlabel {$a-i$} [ ] at 63 4
 \pinlabel {$n$} [ ] at 70 22
\endlabellist
\centering
\ig{1.5}{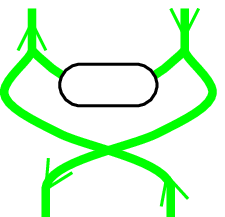}
}
\\ \nonumber \\
\sigma_{\a}^i  :=   (-1)^{ab}\sum_{\beta, \gamma, x, y}(-1)^{\frac{i(i+1)}{2}+|x|+|y|}
c_{\alpha,\beta,\gamma,x, y}^{K_i} &
{
\labellist
\small\hair 2pt
 \pinlabel {$\pi_\g^\spadesuit$} [ ] at 28 42
 \pinlabel {$\pi_{\overline{x}}$} [ ] at 8 34
 \pinlabel {$\pi_{\overline{y}}$} [ ] at 68 34
 \pinlabel {$\pi_\b$} [ ] at 37 22
 \pinlabel {$a$} [ ] at 6 7
 \pinlabel {$b$} [ ] at 69 7
 \pinlabel {$b-i$} [ ] at 10 70
 \pinlabel {$a-i$} [ ] at 66 70
 \pinlabel {$i$} [ ] at 19 24
 \pinlabel {$i$} [ ] at 53 40
 \pinlabel {$n$} [ ] at 80 49
\endlabellist
\centering
\ig{1.5}{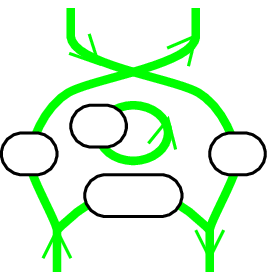}
}
\end{eqnarray}
That these maps collectively satisfy \eqref{theydecompose} is the content of the Stosic formula, Theorem V.6. That they satisfy \eqref{theygive1N} and \eqref{theyareorthogonal} follows from Lemmas V.2 and V.3.

Our goal in this section is to prove that they also satisfy \eqref{conditionforfantastic}. More specifically, fix the partial order on the set $I = \{(j, \a) \; | \; 0 \le j \le \min(a,b), \a \in
P(j,n+a-b-j)\}$ which satisfies $(i,\a) \le (j,\b)$ if and only if $i \le j$, or $i=j$ and $\a \le \b$. Recall that, for two partitions, $\a \le \b$  means that $\b$ can be obtained by adding boxes iteratively to $\a$. We seek to show that
\begin{equation} \label{conditionforfantasticNow}
\sigma_\b^j \dif(\l_\a^i) = 0 \textrm{ for } (j,\b)
\le (i,\a).
\end{equation}
(We have reversed the partial order from our original condition \eqref{conditionforfantastic}, to make $\le$ seem more natural.) As this computation is close to the computation of $\sigma_b^j \l_\a^i$, one should expect it to closely follow Lemmas V.2 and V.3.

In the discussion before Corollary V.8, \cite{KLMS} show that the idempotent decomposition above for \eqref{firstone}, after applying the symmetries of $\Uthick$, yields an idempotent decomposition for \eqref{secondone}. A similar argument will allow \eqref{conditionforfantasticNow} to imply the corresponding condition for the idempotent decomposition of \eqref{secondone}. For this reason, we stick to the $n \ge b-a$ case.

%
\subsection{Checking the fantastical condition}
%

First we compute the differential of the top half of $\l_\a^i$.

\begin{equation} \label{dif-lambda} \dif \left( \quad \; \;
	{
	\labellist
	\small\hair 2pt
	 \pinlabel {$\pi_\a$} [ ] at 31 38
	 \pinlabel {$a$} [ ] at 0 56
	 \pinlabel {$b$} [ ] at 61 56
	 \pinlabel {$i$} [ ] at 44 49
	 \pinlabel {$a-i$} [ ] at 1 4
	 \pinlabel {$b-i$} [ ] at 63 4
	 \pinlabel {$n$} [ ] at 70 22
	\endlabellist
	\centering
	\ig{1.5}{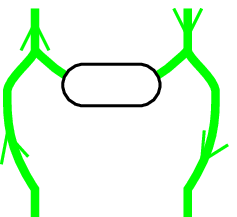}
	} \quad \; \; \right)
\quad = \sum_{\a + \square} (C(\square) + i - n + b - a) \quad
{
\labellist
\small\hair 2pt
 \pinlabel {$\pi_{\a + \square}$} [ ] at 31 38
 \pinlabel {$a$} [ ] at 0 56
 \pinlabel {$b$} [ ] at 61 56
 \pinlabel {$i$} [ ] at 44 49
 \pinlabel {$a-i$} [ ] at 1 4
 \pinlabel {$b-i$} [ ] at 63 4
 \pinlabel {$n$} [ ] at 70 22
\endlabellist
\centering
\ig{1.5}{lambdatop}
}
\end{equation}

\begin{proof} Four terms contribute to this differential: the $a$ splitter, the $b$ splitter, the clockwise thick cup, and the Schur polynomial $\pi_\a$. The first three terms all yield the following diagram, with various coefficients.
\begin{equation} \label{lambdatopproof}	{
	\labellist
	\small\hair 2pt
	 \pinlabel {$\pi_{\a}$} [ ] at 23 38
	 \pinlabel {$a$} [ ] at 0 56
	 \pinlabel {$b$} [ ] at 61 56
	 \pinlabel {$i$} [ ] at 44 49
	 \pinlabel {$a-i$} [ ] at 1 4
	 \pinlabel {$b-i$} [ ] at 63 4
	 \pinlabel {$n$} [ ] at 70 22
	\endlabellist
	\centering
	\ig{1.5}{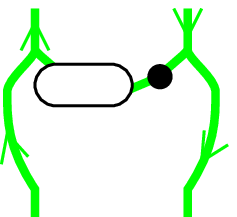}
	} \end{equation}
The $a$ splitter yields coefficient $i-a$, by \eqref{eq-dif-thick-splitter-down}. The $b$ splitter yields coefficient $i-b$, by \eqref{eq-dif-thick-splitter-down-D}. The cup yields coefficient $i-(n-2b+2i)$, by \eqref{eq-dif-thick-CW-cup}, because the bottommost region has label $n - 2b + 2i$. Summing together these coefficients, the overall result is $i-n+b-a$. Now the diagram in \eqref{lambdatopproof} has polynomial $e_1 \pi_\a$ on the crossbeam, and by the Pieri rule, \[ e_1 \pi_\a = \sum_{\a + \square} \pi_{\a + \square}. \] Combining this with \eqref{eq-dif-schur}, we get the desired result. \end{proof}

Using this formula and \eqref{eq-dif-left-crossing-alt}, we now compute $\dif(\l^i_\a)$.
\begin{align}
\dif(\l^i_\a) & = \sum_{\a + \square} (C(\square) + i + b - a - n) {
\labellist
\small\hair 2pt
 \pinlabel {$\pi_{\a+\square}$} [ ] at 31 38
 \pinlabel {$a$} [ ] at 0 56
 \pinlabel {$b$} [ ] at 61 56
 \pinlabel {$i$} [ ] at 44 49
 \pinlabel {$b-i$} [ ] at 1 4
 \pinlabel {$a-i$} [ ] at 63 4
 \pinlabel {$n$} [ ] at 70 22
\endlabellist
\centering
\ig{1.5}{lambda}
} \label{term1}  \\
& + (n + a - b)  \left( \quad  {
\labellist
\small\hair 2pt
 \pinlabel {$\pi_{\a}$} [ ] at 31 38
 \pinlabel {$a$} [ ] at 0 56
 \pinlabel {$b$} [ ] at 61 56
 \pinlabel {$i$} [ ] at 44 49
 \pinlabel {$b-i$} [ ] at 1 4
 \pinlabel {$a-i$} [ ] at 63 4
 \pinlabel {$n$} [ ] at 70 22
\endlabellist
\centering
\ig{1.5}{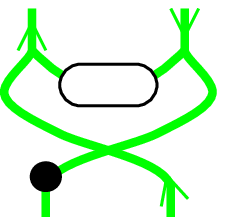}
} \quad + \quad {
\labellist
\small\hair 2pt
 \pinlabel {$\pi_{\a}$} [ ] at 31 38
 \pinlabel {$a$} [ ] at 0 56
 \pinlabel {$b$} [ ] at 61 56
 \pinlabel {$i$} [ ] at 44 49
 \pinlabel {$b-i$} [ ] at 1 4
 \pinlabel {$a-i$} [ ] at 63 4
 \pinlabel {$n$} [ ] at 70 22
\endlabellist
\centering
\ig{1.5}{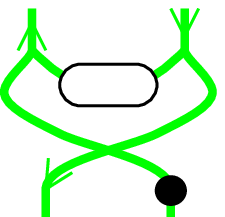}
} \quad \right) \label{term2} \\ &
  - (n + a - b) {
\labellist
\small\hair 2pt
 \pinlabel {$\pi_{\a}$} [ ] at 31 38
 \pinlabel {$a$} [ ] at 0 56
 \pinlabel {$b$} [ ] at 61 56
 \pinlabel {$i$} [ ] at 44 49
 \pinlabel {$b-i$} [ ] at 1 4
 \pinlabel {$a-i$} [ ] at 63 4
 \pinlabel {$n$} [ ] at 70 22
\endlabellist
\centering
\ig{1.5}{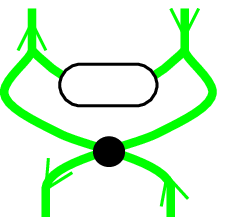}
} \label{term3}
\end{align}

Consider the first term \eqref{term1}, which we call $T_1$. Consider a single summand in $T_1$, which we call $x$, corresponding to the partition $\a + \square$. When the partition $\a +
\square$ lies within $P(i,n+a-b-i)$, then $x = \l^i_{\a + \square}$. When $\a + \square$ is not contained in $P(i,n+a-b-i)$, it was because the box was added in the first row, and has content
precisely $n+a-b-i$. Therefore, $x$ appears with coefficient zero in $T_1$.

One concludes that
\begin{equation} \label{term1check}
\sigma^j_\b T_1 = 0 \textrm{ unless } j=i \textrm{ and } \b \in P(i,n+a-b-i) \textrm{ has the form } \a + \square.
\end{equation}
In particular, $\sigma^j_\b T_1 = 0$ when $(j,\b) \le (i,\a)$.

Suppose that $n=b-a$. Then the remaining terms have coefficient zero, and \eqref{conditionforfantasticNow} is proven. Thus we can assume below that $n > b-a$.

Consider the second term \eqref{term2}, which we call $(n+a-b) T_2$ (i.e. the sum of diagrams is $T_2$). Clearly $T_2$ is none other than $\l^i_\a$ with $\ig{.5}{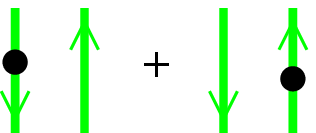}$ on the
bottom. In particular,
\begin{equation} \label{term2check}
\sigma^j_\b T_2 = \begin{cases} \ig{1}{dotplusdot} &\mbox{if } j = i, \a = \b, \\ 0 &\mbox{else.} \end{cases}
\end{equation}

Consider the third term \eqref{term3}, which we call $-(n+a-b) T_3$. Then $T_3$ looks much the same as $\l_\a^i$, except with an extra dot on the sideways crossing near the bottom. Let us
now resolve the diagram $\sigma^j_\b T_3$, closely following Lemma V.2 which resolves $\sigma^j_\b \l^i_\a$. Lemma V.2 goes through several pages of local manipulations, none of which involve the
sideways crossing near the bottom, until the middle paragraph of page 57. Thus we can apply these manipulations verbatim to our computation.

In the middle paragraph of page 57, \cite{KLMS} apply the Square Flop Lemma iteratively, which either swaps a rightward-oriented line with a leftward-oriented line, or removes them both. They
argue that every rightward-oriented line must eventually cancel against some leftward-oriented line. Otherwise, if some line remains uncanceled, it will create one of the two diagrams in
(V.27), which both vanish. In our circumstance, the first diagram in (V.27) has an extraneous dot on the sideways crossing, and this does not vanish! However, the second term in (V.27) is
unchanged, and vanishes as before. If $j < i$ then some leftward-oriented line remains uncanceled; one obtains the second diagram in (V.27), and the result is zero. We have shown
\begin{equation} \label{term3checkpart1}
\sigma^j_\b T_3 = 0 \textrm{ when } j<i.
\end{equation}

If $j > i$, it is possible for some rightward-oriented line not to cancel, and the result to be nonzero, but this is allowed by our partial order.

If $i=j$ and some line remains uncanceled, one also has an instance of the second diagram in (V.27), and the result is zero. Thus when $i=j$ we can assume all lines cancel, as they do on
page 57, and we can continue to follow their argument all the way through the end of the lemma. We obtain
\begin{equation} \label{term3checkpart2}
\sigma^i_\b T_3 =
\begin{cases}
(-1)^{(a-i)(b-i)} \ig{1}{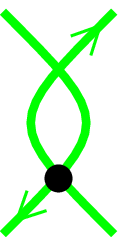} &\mbox{if } \a = \b \\ 0 &\mbox{else.}
\end{cases}
\end{equation}

Putting together the results of \eqref{term1check}, \eqref{term2check}, \eqref{term3checkpart1}, and \eqref{term3checkpart2}, we have proven that
\[
\sigma^j_\b \dif(\l^i_\a) = 0
\]
whenever $j < i$, or whenever $i=j$ and $\b \ngeq \a$. We have also shown that $\sigma^i_\a \dif(\l^i_\a) = 0$ when $n = b-a$. It remains to show, when $i=j$ and $\b = \a$ and $n > b-a$, that $\sigma^i_\a \dif(\l^i_\a)=0$, which now amounts to
\begin{equation} \label{5.3analog} \ig{1}{doublecrosswithdotonX} = (-1)^{cd} \left(\;\; {
\labellist
\small\hair 2pt
 \pinlabel {$d$} [ ] at -6 9
 \pinlabel {$c$} [ ] at 31 9
\endlabellist
\centering
\ig{1}{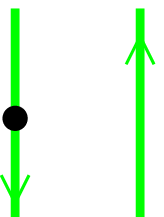}
} \quad + \quad {
\labellist
\small\hair 2pt
 \pinlabel {$d$} [ ] at -6 9
 \pinlabel {$c$} [ ] at 31 9
\endlabellist
\centering
\ig{1}{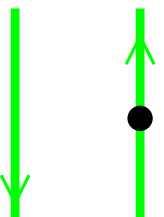}
} \right),
\end{equation}
where $c=a-i$ and $d=b-i$. This is our final lemma.

\begin{lemma}
The equality \eqref{5.3analog} holds, when $n > d-c$.
 \end{lemma}

\begin{proof} First we use the relation \eqref{thickdotoncrossing}, and simplify using Lemma V.3 itself, to reduce the lemma to the following statement.
\begin{equation} \label{reduced5.3analog} \ig{1}{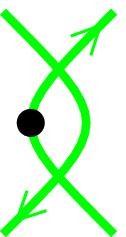} = (-1)^{cd} 	{
	\labellist
	\small\hair 2pt
	 \pinlabel {$d$} [ ] at -6 9
	 \pinlabel {$c$} [ ] at 31 9
	\endlabellist
	\centering
	\ig{1}{dotonrighttall}
	}. \end{equation}
	
Now we simplify the left diagram in \eqref{reduced5.3analog}, following the proof of Lemma V.3. After exploding the diagram as in (V.43) we obtain a ladder as in (V.44), except that where (V.44) had $c-1$ dots on the leftmost strand, we now have $c$ dots.

\begin{equation} \label{eq_sideways_ladder}
 = \;\;
{
\labellist
\small\hair 2pt
 \pinlabel {$\ldots$} [ ] at 56 49
 \pinlabel {$\ldots$} [ ] at 56 17
 \pinlabel {$\ldots$} [ ] at 159 49
 \pinlabel {$\ldots$} [ ] at 159 17
 \pinlabel {$\bullet$} [ ] at 27.5 33
 \pinlabel {$\bullet$} [ ] at 83.5 33
 \pinlabel {$\bullet$} [ ] at 11.5 33
 \pinlabel {$\bullet$} [ ] at 203.5 33
 \pinlabel {$\bullet$} [ ] at 131.5 33
 \pinlabel {$\bullet$} [ ] at 187.5 33
 \pinlabel {$\scs c-2$} [ ] at 22 38
 \pinlabel {$\scs 1$} [ ] at 79 38
 \pinlabel {$\scs c$} [ ] at 6 38
 \pinlabel {$\scs d-1$} [ ] at 210 38
 \pinlabel {$\scs 1$} [ ] at 137 38
 \pinlabel {$\scs d-2$} [ ] at 194 38
 \pinlabel {$d$} [ ] at -3 56
 \pinlabel {$d$} [ ] at -3 7
 \pinlabel {$d+1$} [ ] at 21 57
 \pinlabel {$d+1$} [ ] at 21 10
 \pinlabel {$c+d$} [ ] at 108 57
 \pinlabel {$c+d$} [ ] at 108 11
 \pinlabel {$c+1$} [ ] at 194 57
 \pinlabel {$c+1$} [ ] at 194 11
 \pinlabel {$c$} [ ] at 216 56
 \pinlabel {$c$} [ ] at 216 10
 \pinlabel {$n$} [ ] at 220 35
\endlabellist
\centering
\ig{1.5}{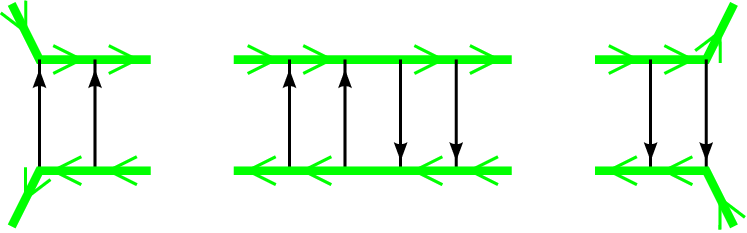}
}
\end{equation}

Now we repeatedly apply the Square Flop Lemma, as they do. They argue that, in this case, the Square Flop Lemma reduces to (V.45) because $x+y \le m-2$ (see the top of page 62). For our
situation with an extra dot, the statement $x+y \le m-2$ still holds; we may have increased $x$ by one, but we also ignored the edge case $n = d-c$ (i.e. $n = b-a$ in application) where
this might have caused problems. Thus we can follow their argument exactly, obtaining the following modification of (V.46).

\begin{equation}
= (-1)^{cd} \;\;
{
\labellist
\small\hair 2pt
 \pinlabel {$\ldots$} [ ] at 56 49
 \pinlabel {$\ldots$} [ ] at 56 17
 \pinlabel {$\ldots$} [ ] at 159 49
 \pinlabel {$\ldots$} [ ] at 159 17
 \pinlabel {$\bullet$} [ ] at 27.5 33
 \pinlabel {$\bullet$} [ ] at 83.5 33
 \pinlabel {$\bullet$} [ ] at 99.5 33
 \pinlabel {$\bullet$} [ ] at 115.5 33
 \pinlabel {$\bullet$} [ ] at 131.5 33
 \pinlabel {$\bullet$} [ ] at 187.5 33
 \pinlabel {$\scs 1$} [ ] at 23 38
 \pinlabel {$\scs d-2$} [ ] at 77 38
 \pinlabel {$\scs d-1$} [ ] at 93 38
 \pinlabel {$\scs c$} [ ] at 119 38
 \pinlabel {$\scs c-2$} [ ] at 139 38
 \pinlabel {$\scs 1$} [ ] at 192 38
 \pinlabel {$d$} [ ] at -3 56
 \pinlabel {$d$} [ ] at -3 7
 \pinlabel {$d-1$} [ ] at 21 57
 \pinlabel {$d-1$} [ ] at 21 10
 \pinlabel {$1$} [ ] at 92 57
 \pinlabel {$1$} [ ] at 92 11
 \pinlabel {$0$} [ ] at 108 57
 \pinlabel {$0$} [ ] at 108 11
 \pinlabel {$1$} [ ] at 123 57
 \pinlabel {$1$} [ ] at 123 11
 \pinlabel {$c-1$} [ ] at 194 57
 \pinlabel {$c-1$} [ ] at 194 11
 \pinlabel {$c$} [ ] at 216 56
 \pinlabel {$c$} [ ] at 216 10
 \pinlabel {$n$} [ ] at 220 35
\endlabellist
\centering
\ig{1.5}{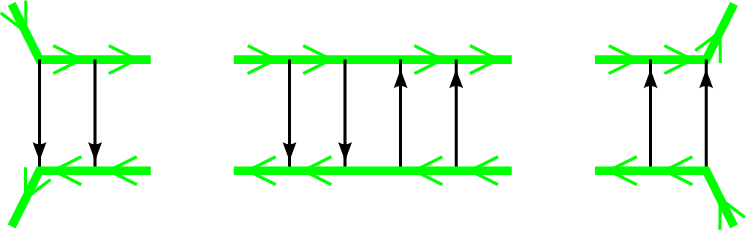}
}
\nonumber
\end{equation}
\begin{equation}
  = (-1)^{cd} \;\;
	{
	\labellist
	\small\hair 2pt
	 \pinlabel {$\bullet$} [ ] at 21 40
	 \pinlabel {$\bullet$} [ ] at 51 40
	 \pinlabel {$\bullet$} [ ] at 72 40
	 \pinlabel {$\scs 1$} [ ] at 16 47
	 \pinlabel {$\scs d-2$} [ ] at 58 47
	 \pinlabel {$\scs d-1$} [ ] at 80 47
	 \pinlabel {$\ldots$} [ ] at 35 40
	 \pinlabel {$d$} [ ] at 25 11
	\endlabellist
	\centering
	\ig{1}{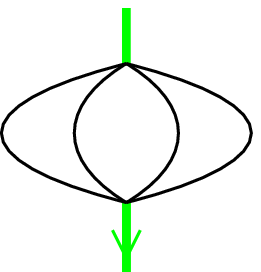}
	} \qquad
	{
	\labellist
	\small\hair 2pt
	 \pinlabel {$\bullet$} [ ] at 1 37
	 \pinlabel {$\scs c$} [ ] at -6 41
	 \pinlabel {$\bullet$} [ ] at 21 37
	 \pinlabel {$\scs c-2$} [ ] at 12 41
	 \pinlabel {$\ldots$} [ ] at 36 37
	 \pinlabel {$\bullet$} [ ] at 52 37
	 \pinlabel {$\scs 1$} [ ] at 57 41
	\pinlabel {$c$} [ ] at 25 9
	\endlabellist
	\centering
	\ig{1}{explode}
	} \qquad = \quad (-1)^{cd} \;\; {
	\labellist
	\small\hair 2pt
	 \pinlabel {$d$} [ ] at -6 9
	 \pinlabel {$c$} [ ] at 31 9
	\endlabellist
	\centering
	\ig{1}{dotonrighttall}
	}
\end{equation}
This is the desired result. 
\end{proof}


\bibliographystyle{alpha}
\bibliography{qy-bib}


%

\vspace{0.1in}

\noindent B.~E.: {\small Department of Mathematics, University of Oregon,
Eugene, OR 97403, USA} \newline\noindent  {\tt \small email: belias@uoregon.edu}

\vspace{0.1in}

\noindent Y.~Q.: {\small Department of Mathematics, Yale University, New
Haven, CT 06511, USA} \newline \noindent {\tt \small email: you.qi@yale.edu}

%
\end{document}